\documentclass[12pt]{article}

\pdftrailerid{}

\usepackage[T1]{fontenc}
\usepackage[english]{babel}
\usepackage{lmodern}

\usepackage{amsmath}
\usepackage{amsthm}
\usepackage[bbgreekl]{mathbbol} 
\usepackage{amssymb}
\DeclareSymbolFontAlphabet{\mathbb}{AMSb}
\DeclareSymbolFontAlphabet{\mathbbl}{bbold}
\usepackage{mathrsfs}
\usepackage{bm}
\usepackage{bbm}
\usepackage{mathtools}
\usepackage{thmtools}

\usepackage[final,babel]{microtype}
\usepackage{csquotes}

\usepackage[backend=biber,sorting=nyt,style=alphabetic,backref=true,date=year,useprefix=true,maxnames=9,labelalpha,maxalphanames=9]{biblatex}

\usepackage{shortmathj}
\DeclareFieldFormat[article]{titlecase}{\shortifyAMSjournalname{#1}}
\DeclareFieldFormat{postnote}{#1}
\DeclareFieldFormat{pages}{#1}
\DeclareFieldFormat[article, incollection]{title}{#1}
\DeclareFieldFormat{urldate}{(#1)}
\DeclareFieldFormat[article]{volume}{\textbf{#1}\nopunct}
\DeclareFieldFormat[article]{number}{(\textbf{#1})}
\DeclareFieldFormat{pageref}{}
\DefineBibliographyStrings{english}{
	backrefpage={},
	backrefpages={},
}
\usepackage{xpatch}
\DeclareFieldFormat{backrefparens}{\addperiod\allowbreak#1\addperiod}
\xpatchbibmacro{pageref}{parens}{backrefparens}{}{}
\renewbibmacro{in:}{%
	\ifentrytype{article}{}{\printtext{\bibstring{in} }}}

\usepackage[useregional]{datetime2}
\usepackage{geometry}
\usepackage{hyperref}
\usepackage{xurl} 
\hypersetup{
	breaklinks,
	unicode,
	colorlinks
}
\usepackage[capitalise]{cleveref}
\crefname{equation}{}{} 
\usepackage[title]{appendix}

\usepackage{booktabs}
\usepackage{float} 
\usepackage{threeparttable}

\usepackage[table,xcdraw]{xcolor}
\usepackage{soul}
\usepackage{enumitem}
\usepackage{caption}
\usepackage{cancel}
\usepackage{tikz}
\usetikzlibrary{cd}
\usetikzlibrary{decorations.markings}
\usetikzlibrary{decorations.pathreplacing}
\usetikzlibrary{arrows.meta}

\declaretheorem{thm}[
	numberwithin=section,
	name=Theorem,
	refname={Theorem,Theorems},
	Refname={Theorem,Theorems},
]
\declaretheorem{lem}[
	sharenumber=thm,
	name=Lemma,
	refname={Lemma,Lemmas},
	Refname={Lemma,Lemmas},
]
\declaretheorem{cor}[
	sharenumber=thm,
	name=Corollary,
	refname={Corollary,Corollaries},
	Refname={Corollary,Corollaries},
]
\declaretheorem{prop}[
	sharenumber=thm,
	name=Proposition,
	refname={Proposition,Propositions},
	Refname={Proposition,Propositions},
]

\declaretheorem{claim}[
	sharenumber=thm,
	style=definition,
	name=Claim,
	refname={Claim,Claims},
	Refname={Claim,Claims},
]
\declaretheorem{defn}[
	sharenumber=thm,
	style=definition,
	name=Definition,
	refname={Definition,Definitions},
	Refname={Definition,Definitions},
	]
\declaretheorem{eg}[
	sharenumber=thm,
	style=definition,
	name=Example,
	refname={Example,Examples},
	Refname={Example,Examples},
]
\declaretheorem{rmk}[
	sharenumber=thm,
	style=definition,
	name=Remark,
	refname={Remark,Remarks},
	Refname={Remark,Remarks},
]

\declaretheorem{prob}[
sharenumber=thm,
style=definition,
name=Problem,
refname={Problem,Problems},
Refname={Problem,Problems},
]

\newenvironment{claimproof}[1][\proofname]
{%
	\proof[#1]%
}
{%
	\endproof%
}

\newcommand{\tb}{\textbf}
\newcommand{\bb}[1]{\mathbb{#1}}
\newcommand{\res}[2]{#1 {\upharpoonright} #2}

\newcommand{\N}{\bb{N}}
\newcommand{\Z}{\bb{Z}}

\DeclareMathOperator{\Mod}{Mod}
\DeclareMathOperator{\proj}{proj}
\DeclareMathOperator{\dom}{dom}
\DeclareMathOperator{\ran}{ran}
\DeclareMathOperator{\freq}{Freq}
\DeclareMathOperator{\aut}{Aut}
\DeclareMathOperator{\age}{Age}

\newcommand{\LL}{\mathcal{L}}
\newcommand{\K}{\mathcal{K}}

\newcommand{\bA}{{\bm{A}}}
\newcommand{\bB}{{\bm{B}}}
\newcommand{\bbA}{{\bb{A}}}
\newcommand{\bbB}{{\bb{B}}}

\renewcommand{\P}{\bb{P}}
\newcommand{\E}{\bb{E}}

\title{Definable expansions on countable groups and countable Borel equivalence relations}
\author{Michael Wolman}
\date{May 9, 2025} 

\begin{document}

\maketitle

\begin{abstract}
	We define and study expansion problems on countable structures in the setting of descriptive combinatorics. We consider both expansions on countable Borel equivalence relations and on countable groups, in the Borel, measure and category settings, and establish some basic correspondences between the two notions. We also prove some general structure theorems for measure and category. We then explore in detail many examples, including finding spanning trees in graphs, finding monochromatic sets in Ramsey's Theorem, and linearizing partial orders.
\end{abstract}

\section{Introduction}

A \emph{countable Borel equivalence relation (CBER)} on a Polish space $X$ is a Borel equivalence relation $E \subseteq X^2$ whose equivalence classes are countable. Given a CBER $E$, a \emph{(Borel) structuring of $E$} is a Borel assignment of a first-order structure on each $E$-class $C$ (see \cref{sec:structures-and-expansions-on-CBER,sec:descriptions} for precise definitions).

In this paper, we are primarily concerned with the \emph{descriptive combinatorics} of \emph{locally countable structures}. Broadly speaking, given a combinatorial problem on countable structures, we are interested in solving it in a ``uniformly Borel'' way, possibly after throwing away a meagre set or a null set. For instance, given a Borel structuring of a CBER $E$ by countable graphs, we may be interested in characterizing exactly when one can find a Borel colouring of these graphs with countably many colours, i.e., a colouring so that the assignment of the colour classes to the vertices in each $E$-class is a Borel structuring of $E$. Other examples of combinatorial problems include finding proper edge colourings, perfect matchings or spanning trees in graphs, finding infinite monochromatic sets (as in Ramsey's Theorem), and extending a given partial order into a linear order; see \cref{sec:expansions-def} for more. We refer the reader to \cite{KM-survey,Pikhurko-survey} for a survey of results in descriptive combinatorics, and to \cite{CK,BC} for more on the structurability of CBER.

For ``locally finite'' structures, many of these combinatorial problems can be expressed in terms of \emph{constrain satisfaction} or \emph{locally checkable labelling} problems on graphs. In this setting, there has been a lot of recent progress towards finding solutions to various expansion problems, using tools from theoretical computer science and finite combinatorics such as the Lovász Local Lemma and connections with LOCAL algorithms in distributed computing \cite{Bernshteyn-LLL,BCGGRV,GR}. However, we note that many problems are not locally finite and hence do not fit within this framework; for example, linearizations of partial orders, or Ramsey's Theorem (see e.g. \cite{GX}).

Here, we consider these problems in the more general framework of \emph{expansions}. Given first-order languages $\LL \subseteq \LL^*$ and an $\LL$-structure $\bA$, we call an $\LL^*$-structure $\bA^*$ an \emph{expansion} of $\bA$ if $\bA = \res{\bA^*}{\LL}$, where $\res{\bA^*}{\LL}$ denotes the \emph{reduct} of $\bA^*$ to $\LL$. If $\K$ is a class of $\LL$ structures and $\K^*$ is a class of $\LL^*$-structures, the \emph{expansion problem for $(\K, \K^*)$} is the problem of determining whether every structure in $\K$ admits an expansion in $\K^*$. For a countably infinite set $X$, let $\K(X)$ denote the set of structures in $\K$ whose universe is $X$, and call $\K$ a \emph{Borel class of structures} if $\K(X)$ is Borel for all countably infinite sets $X$.

Given an expansion problem $(\K, \K^*)$ for which every element of $\K$ admits an expansion to an element of $\K^*$, we get a corresponding ``uniformly Borel'' expansion problem: For every CBER $E$ and any structuring of $E$ with elements of $\K$, is there a structuring of $E$ with elements of $\K^*$ which is an expansion of the original structuring on every $E$-class? In general, one can view this Borel expansion problem as asking if there is a ``canonical'' assignment of an expansion in $\K^*$ to every element of $\K$; this is made precise in \cites[arXiv~version, Appendix~B]{CK}{BC}.

One can also interpret the Borel expansion problem in terms of definable equivariant maps. If $\Gamma$ is a group acting on a countably infinite set $X$ and $(\K, \K^*)$ is an expansion problem with $\K, \K^*$ Borel, we may consider the \emph{$\Gamma$-equivariant expansion problem}: Is there a Borel map $f: \K(X) \to \K^*(X)$, taking $\bA \in \K(X)$ to an expansion $f(\bA)$, which is equivariant with respect to the induced action of $\Gamma$ on $\K(X), \K^*(X)$?

There is a natural correspondence between $\Gamma$-equivariant expansions for countable groups $\Gamma$, in the special case where $\Gamma$ acts on $X = \Gamma$ by multiplication on the left, and CBER which arise via free Borel actions of $\Gamma$ (see \cref{sec:free-actions}). By the Feldman--Moore Theorem, every CBER is induced by a Borel action of a countable group, though in general we cannot expect this action to be free \cite[Section~11]{CBER}. Nevertheless, CBER induced by free Borel actions of countable groups are a great source of (counter-)examples in the study of Borel expansions on CBER (especially with respect to the Schreier graphs of their actions), and remain very relevant in the study of the descriptive combinatorics of locally countable structures.

The primary objective of this paper is to study this correspondence between the Borel expansion problem on CBER and the Borel equivariant expansion for countable groups $\Gamma$. More generally, we study also the connection between these problems in the settings of \emph{measure} and \emph{category}, i.e., when we are allowed to solve these problems after possibly removing a null or meagre set. We shall see that by exploiting this connection, we can apply results and techniques from the theory of CBER to prove theorems about equivariant expansions on countable groups (see e.g. \cref{sec:inv-rand-expansion-cber,sec:generic-expansions,sec:examples}). Conversely, we apply tools from symbolic dynamics and probability theory (such as the mass transport principle and random walks on groups) to the study of equivariant expansions, which gives in some cases precise characterizations of exactly when certain structurings of CBER admit definable expansions (c.f. \cref{sec:canonical-random-expansions,sec:enfocing-smoothness,sec:examples}).

The connection between expansions on CBER and equivariant expansions in the purely Borel setting has been studied independently in \cite{BC}, with the goal of making precise the relation between the existence of Borel expansions on CBER and ``canonical'' expansions from $\K$ to $\K^*$. There, \textcite[Corollary~3.29, Remark~2.25]{BC} show that every Borel structuring of a CBER admits a Borel expansion if and only if there is a Borel $S_\N$-equivariant expansion map, for classes $\K$ of structures that interpret the theories of \emph{Lusin--Novikov functions} and \emph{countable separating families} (c.f. \cite[Definitions~3.18, 3.23]{BC}). (We note however that the classes we study in this paper do not interpret these theories.) Expansion problems have also been studied in the context of \emph{invariant random structures} on groups, i.e., invariant probability measures on $\K(\Gamma)$, and one can view equivariant expansions as a natural strengthening of this notion; see e.g. \cite[Sections~6, 15]{KM-survey} for some examples in graph combinatorics, or \cite{GLM,A} for linearizations of partial orders.

\paragraph*{Organization.} The structure of this paper is as follows. In \cref{sec:preliminaries} we give precise definitions of expansions on CBER and equivariant expansions on groups for expansion problems, in the Borel, Baire category and measurable settings. We also give examples of various expansion problems of interest, that we study in detail in \cref{sec:examples}.

In \cref{sec:general-results}, we prove several general theorems relating equivariant expansions on groups $\Gamma$ with Borel expansions on CBER induced by free Borel actions of $\Gamma$. We describe in \cref{sec:free-actions} a weak duality between the two notions, which can be viewed as an analogue of \cite[Corollary~3.29]{BC} for this setting. We then consider \emph{random expansions} for countable groups, where we say an invariant measure $\nu$ on $\K^*(\Gamma)$ is a random expansion of an invariant measure $\mu$ on $\K(\Gamma)$ when the reduct of $\nu$ is equal to $\mu$ (c.f. \cref{sec:def-equivariant-expansions}). We show that the existence of random expansions on $\Gamma$ depends only on its orbit equivalence class, where we say groups $\Gamma, \Lambda$ are \emph{orbit equivalent} if there is a CBER $E$ induced by free probability-measure-preserving actions of both $\Gamma$ and $\Lambda$.
\begin{thm}[\cref{thm:invariant-random-expansions}]
	Let $(\K, \K^*)$ be an expansion problem and $\Gamma, \Lambda$ be countably infinite groups. If $\Gamma, \Lambda$ are orbit equivalent, then $\Gamma$ admits random expansions from $\K$ to $\K^*$ if and only if $\Lambda$ admits random expansions from $\K$ to $\K^*$.
\end{thm}
\noindent We note that this has already been observed in some special cases, for example with linearizations in \cite{A}, though we show here that it holds more generally for all expansion problems. We also give a sort of converse in \cref{prop:converse-ire}.

Next, we consider \emph{generic} equivariant expansions on $G_\delta$ classes of structures, i.e., equivariant expansions on comeagre subsets of $\K(\Gamma)$ (c.f. \cref{sec:def-equivariant-expansions}). Given an expansion problem $(\K, \K^*)$ and a countably infinite group $\Gamma$, we say $\K$ admits $\Gamma$-equivariant expansions \emph{generically} if there is a comeagre invariant Borel set $X \subseteq \K(\Gamma)$ such that there is a Borel $\Gamma$-equivariant expansion map $X \to \K^*(\Gamma)$. We show that when $\K$ consists of structures with \emph{trivial algebraic closure} that are not definable from equality, whether or not $\K$ admits $\Gamma$-equivariant expansions to $\K^*$ generically is independent of the group $\Gamma$. (A structure is said to have trivial algebraic closure if its automorphism group has infinite orbits, even after fixing finitely many points, and is definable from equality when relations between tuples of points depend only on their equality types; see \cref{def:TAC} for precise definitions of these terms.)
\begin{thm}[\cref{thm:generic-expansions}]
	Let $(\K, \K^*)$ be an expansion problem. Suppose that $\K$ is $G_\delta$ and the generic element of $\K$ has trivial algebraic closure and is not definable from equality. Then the following are equivalent:
	\begin{enumerate}
		\item For every countably infinite group $\Gamma$, $\K$ admits $\Gamma$-equivariant expansions to $\K^*$ generically.
		\item There exists a countably infinite group $\Gamma$ for which $\K$ admits $\Gamma$-equivariant expansions to $\K^*$ generically.
	\end{enumerate}
\end{thm}

A CBER $E$ is \emph{smooth} if there is a Borel set that contains exactly one point from every $E$-class. We give in \cref{sec:enfocing-smoothness} sufficient conditions for an expansion problem to satisfy (a) that every structuring of a smooth CBER admits a Borel expansion (\cref{prop:smooth-expandable,rmk:smooth-expandable}), or (b) that every non-smooth CBER admits a structuring with no Borel expansion (\cref{prop:enforce-smoothness-generic-exp,cor:enforcing-smoothness-generic}).

In \cref{sec:examples} we analyze in detail the expansion problem for the examples described in \cref{sec:expansions-def}, using in particular the tools we developed in \cref{sec:general-results}. We summarize our results in \cref{fig:table}; we highlight a few of these below.

Call a CBER \emph{aperiodic} if it has only infinite equivalence classes. In \cite{KST} it is shown that for every non-smooth aperiodic CBER $E$, there are Borel sets $A, B$ which have infinite intersection with every $E$-class, but for which there is no Borel bijection $f: A \to B$ whose graph is contained in $E$. By contrast, we have the following:

\begin{prop}[Generic bijections (\cref{prop:bijection-generically-exp})]
	Let $E$ be an aperiodic CBER on $X$ and $A, B \subseteq X$ be sets that have infinite intersection with every $E$-class. Then there is a comeagre $E$-invariant set $Y \subseteq X$ and a Borel bijection $f: A \cap Y \to B \cap Y$ whose graph is contained in $E$, i.e., such that $x E f(x)$ for all $x \in A \cap Y$.
\end{prop}

The problem of whether an invariant random partial order on a countably infinite group $\Gamma$ can be linearized was studied in \cite{GLM,A}. \Textcite{A} has shown that this random expansion property holds for $\Gamma$ if and only if $\Gamma$ is amenable. By contrast, for equivariant maps and CBER we have the following:

\begin{prop}[Linearizations (\cref{prop:linearization-not-ae,prop:linearization-generic-expansion})]
	\begin{enumerate}
		\item[]

		\item Let $\K$ be the class of partial orders and $\K^*$ be the class of linear orders extending a given partial order. For every countably infinite group $\Gamma$, $\K$ does not admit $\Gamma$-equivariant expansions to $\K^*$ generically.

		\item For every non-smooth CBER $E$, there is a Borel assignment of a partial order to every $E$-class so that there is no Borel way of extending these partial orders to linear orders on every $E$-class. Moreover, if $E$ is aperiodic then one can ensure that for every $E$-invariant probability Borel measure $\mu$, there is no Borel extension of the partial orders to linear orders $\mu$-a.e.
	\end{enumerate}
\end{prop}

A CBER $E$ is \emph{treeable} if there is a Borel assignment of a connected acyclic graph to every $E$-class. The class of treeable CBER has been studied extensively; see e.g. \cite[Section~10]{CBER}. In \cref{sec:spanning-trees}, we consider CBER $E$ that admit Borel spanning trees for \emph{every} Borel assignment of a connected graph to every $E$-class. Clearly every such CBER is treeable. We show that the \emph{hyperfinite} CBER have this property, where a CBER is said to be hyperfinite if it can be written as an increasing union of CBER with finite equivalence classes.

\begin{prop}[Spanning trees (\cref{prop:spanning-trees})]
	Let $E$ be a CBER. If $E$ is hyperfinite, then for every Borel assignment of a connected graph to each $E$-class, there is a Borel assignment of a spanning tree to each $E$-class.
\end{prop}

It is unknown whether the class of CBER with this spanning tree property coincides with the treeable CBER or the hyperfinite CBER, or if it lies somewhere in between.

As a final example, we consider the problem of choosing from a linear order without endpoints a subset that is order-isomorphic to $\Z$, in a Borel way. We give a complete classification of the invariant random structures on countably infinite groups that admit random expansions for this problem, and show moreover that these expansions can always be taken to come from equivariant Borel maps. In particular, this characterizes exactly when a Borel structuring of a CBER admits a Borel expansion for this problem $\mu$-a.e., for any invariant measure $\mu$.

\begin{prop}[$\Z$-lines (\cref{prop:z-line-generic-expansion,prop:z-line-classification})]
	\begin{enumerate}
		\item[]

		\item Let $\K$ be the class of linear orders without endpoints, and $\K^*$ be the class of linear orders without endpoints along with a subset of order-type $\Z$. For any countably infinite group $\Gamma$, $\K$ does not admit $\Gamma$-equivariant expansions to $\K^*$ generically. In particular, every non-smooth CBER $E$ admits a Borel assignment of linear orders to every $E$-class so that there is no Borel way to choose an infinite subset of each $E$-class that has order-type $\Z$.

		\item There is a Borel $\Gamma$-invariant set $X \subseteq \K(\Gamma)$ and a Borel equivariant expansion map $f: X \to \K^*(\Gamma)$ such that, for all invariant random $\K$-structures $\mu$ on $\Gamma$, $\mu$ admits a random expansion to $\K^*$ if and only if $\mu(X) = 1$, in which case $f_*\mu$ gives such an expansion. Moreover, we can choose $f$ so that for all $L \in X$, $f(L)$ picks out an interval in $L$.
	\end{enumerate}
\end{prop}

We also give in \cref{sec:vizing} a survey of recent results regarding the existence of Borel proper edge colourings of bounded-degree graphs (i.e. definable Vizing's Theorem), and in \cref{sec:matchings} a survey of the current landscape regarding the existence of Borel perfect matchings in bipartite graphs (i.e. definable Hall's Theorem).

We end with a list of open problems in \cref{sec:problems}.

\paragraph*{Acknowledgements.} This research was partially supported by NSF Grant DMS-1950475. We would like to thank Alexander Kechris for the initial discussion that inspired this project, and Alexander Kechris, Tom Hutchcroft, Omer Tamuz, Garrett Ervin, Edward Hou and Sita Gakkhar for many helpful conversations throughout. We thank Tom Hutchcroft for telling us about \cref{lem:existence-of-frequencies}, and Minghao Pan for sharing with us his notes on the proof. We also thank Marcin Sabok for their many comments and suggestions regarding our survey in \cref{sec:matchings}. Finally, thank you to Esther Nam for the encouragement and support.

\section{Preliminaries}\label{sec:preliminaries}

For a background on general descriptive set theory, see \cite{CDST}. For a survey of the theory of CBER, see \cite{CBER}. For the basics of structurability of CBER, see \cite{CK}.

\subsection{Languages and structures}

By a \tb{language}, we will always mean a countable relational first-order language, i.e., a countable set $\LL = \{R_i : i \in I\}$, where each $R_i$ is a relation symbol with associated arity $n_i \geq 1$.

Fix now a language $\LL$ and let $X$ be a set. An \tb{$\LL$-structure on $X$} is a tuple $\bA = (X, R^\bA)_{R \in \LL}$ where $R^\bA \subseteq X^n$ for each $n$-ary relation symbol $R \in \LL$. We call $X$ the \tb{universe of $\bA$}, and let $\Mod_{\LL}(X)$ denote the \tb{space of $\LL$-structures on $X$}.

For $\bA \in \Mod_{\LL}(X)$ and $Y \subseteq X$, let $\res{\bA}{Y} \in \Mod_{\LL}(Y)$ denote the \tb{restriction of $\bA$ to $Y$}, given by
\[R^{\res{\bA}{Y}}(y_1, \dots, y_n) \iff R^\bA(y_1, \dots, y_n)\]
for all $n$-ary $R \in \LL$ and $y_1, \dots, y_n \in Y$. For $\LL' \subseteq \LL$, we let $\res{\bA}{\LL'} = (X, R^{\bA})_{R \in \LL'}$ denote the \tb{reduct of $\bA$ to $\LL'$}, i.e., the structure we get when we ``forget'' the relations in $\LL \setminus \LL'$. (Note that the notation $\res{\bA}{(-)}$ is used both for restrictions and reducts; which one we are referring to throughout this paper will be clear from context.)

If $\bA$ is an $\LL$-structure on $X$ and $f: X \to Y$ is a bijection, we write $f(\bA)$ for the \tb{push-forward structure on $Y$}, i.e., the structure on $Y$ given by
\[R^\bA(x_0, \dots, x_{n-1}) \iff R^{f(\bA)}(f(y_0), \dots, f(y_{n-1}))\]
for all $n$-ary relations $R \in \LL$. When $X = Y$ this defines the \tb{logic action} of $S_X$ on $\Mod_{\LL}(X)$, where $S_X$ is the group of bijections of $X$.

We are primarily interested in the cases where $X$ is a countably infinite set, or when $X$ is a Polish space. We will reserve the symbols $\bA, \bB, \dots$ for $\LL$-structures on countable sets, and the symbols $\bbA, \bbB, \dots$ for $\LL$-structures on Polish spaces.

When $X$ is a countable set, one can view $\Mod_{\LL}(X)$ as a compact Polish space, namely,
\[\Mod_{\LL}(X) = \prod_{i \in I} 2^{X^{n_i}}.\]
The logic action of $S_X$ on $\Mod_{\LL}(X)$ is continuous for this topology.

By a \tb{class of $\LL$-structures}, we mean a class $\K$ of countably infinite $\LL$-structures closed under isomorphism. Given such a class $\K$ and a countably infinite set $X$, we let $\K(X) = \K \cap \Mod_{\LL}(X)$ denote the space of $\LL$-structures in $\K$ with universe $X$. We call $\K$ a \tb{Borel class of $\LL$-structures} (resp. a \tb{$G_\delta$ class of $\LL$-structures}, a \tb{closed class of $\LL$-structures}) if $\K(X)$ is Borel (resp. $G_\delta$, closed) as a subset of $\Mod_{\LL}(X)$ for some (equivalently any) countably infinite set $X$.

For a countably infinite set $X$ and an $\LL$-structure $\bA$ let $\age_X(\bA)$ denote the \tb{age of $\bA$ (on $X$)}, that is, the set finite $\LL'$-structures that that embed into $\bA$ and whose universe is contained in $X$, for finite $\LL' \subseteq \LL$. Let
\[\age_X(\K) = \bigcup \{\age_X(\bA) : \bA \in \K(X)\}\]
for any class $\K$ of $\LL$-structures. For $\bA_0 \in \age_X(\Mod_{\LL})$ and $\bA \in \Mod_{\LL}(X)$, we write $\bA_0 \sqsubseteq \bA$ if $\res{(\res{\bA}{\LL'})}{F} = \bA_0$, where $\bA_0$ is an $\LL'$-structure with universe $F \subseteq X$. The topology of $\K(X)$ is generated by basic clopen sets of the form
\[N(\bA_0) = \{\bA \in \K(X) : \bA_0 \sqsubseteq \bA\}\]
for $\bA_0 \in \age_X(\K)$. We note that there is an analogous logic action of $S_X$ on $\age_X(\K)$.

We note that our definition of the age of $\bA$ differs from the usual one (see e.g. \cite[Section~7]{Hodges}), which considers all finite $\LL$-structures that embed into $\bA$ \emph{without} restricting to finite sublanguages. We choose here to restrict to finite sublanguages so that $\age_X(\K)$ corresponds naturally to a basis for the topology on $\K(X)$. (We specify the universe $X$ in $\age_X(\K)$ for the same reason.)

\subsection{Expansions}\label{sec:expansions-def}

Let $\LL \subseteq \LL^*$ be languages, $\bA$ be an $\LL$-structure and $\bA^*$ be an $\LL^*$-structure. We say $\bA^*$ is an \tb{expansion} of $\bA$ if $\res{\bA^*}{\LL} = \bA$.

Given a class of $\LL$-structures $\K$ and a class of $\LL^*$-structures $\K^*$ with $\LL \subseteq \LL^*$, the \tb{expansion problem for $(\K, \K^*)$} is the question of whether or not every element of $\K$ admits an expansion in $\K^*$. We call such pairs $(\K, \K^*)$ \tb{expansion problems}.

Below we give examples of expansion problems $(\K, \K^*)$ we will consider in this paper. In all of these examples the expansion problem will have a positive solution, i.e., every element of $\K$ admits an expansion in $\K^*$; we will be interested in finding ``definable'' expansions, for various notions of definability that we make precise below.

We omit $\LL, \LL^*$ from these examples, as they will be clear from context.

\begin{eg}[Bijections]\label{eg:bijections}
	\begin{align*}
		\K &= \{(X, R, S) \mid R, S \subseteq X \And X, R, S ~ \text{are all countably infinite}\},\\
		\K^* &= \{(X, R, S, T) \mid (X, R, S) \in \K \And T ~ \text{is the graph of a bijection} ~ R \to S\}.
	\end{align*}
	In this case, $\K, \K^*$ are $G_\delta$.
\end{eg}

\begin{eg}[Ramsey's Theorem]\label{eg:ramsey}
	\begin{align*}
		\K &= \{(X, R, S) \mid R, S ~ \text{partition} ~ [X]^2\},\\
		\K^* &= \{(X, R, S, T) \mid (X, R, S) \in \K \And T \subseteq X ~ \text{is infinite}\\
		& \qquad\qquad\qquad\qquad \text{and homogeneous for the partition} ~ R, S\},
	\end{align*}
	where $[X]^2$ is the set of two-element subsets of $X$. Here $\K$ is closed and $\K^*$ is $G_\delta$. Ramsey's Theorem is exactly the statement that every element of $\K$ admits an expansion in $\K^*$.
\end{eg}

\begin{eg}[Linearizations]\label{eg:linearizations}
	\begin{align*}
		\K &= \{(X, P) \mid P ~ \text{is a partial order on} ~ X\},\\
		\K^* &= \{(X, P, L) \mid (X, P) \in \K \And P \subseteq L \And L ~ \text{is a linear order on} ~ X\}.
	\end{align*}
	$\K, \K^*$ are both closed classes of structures.
\end{eg}

\begin{eg}[Vertex colourings]\label{eg:colourings}
	Fix $d \geq 2$, and let
	\begin{align*}
		\K &= \{(X, E) \mid (X, E) ~ \text{is a connected graph of max degree} ~ \leq d\},\\
		\K^* &= \{(X, E, S_0, \dots, S_d) \mid (X, E) \in \K \And S_0, \dots, S_d ~ \text{is a vertex colouring of} ~ (X, E)\}.
	\end{align*}
	Here $\K, \K^*$ are $G_\delta$.
\end{eg}

\begin{eg}[Spanning trees]\label{eg:trees}
	\begin{align*}
		\K &= \{(X, E) \mid (X, E) ~ \text{is a connected graph}\},\\
		\K^* &= \{(X, E, T) \mid (X, E) \in \K \And (X, T) ~ \text{is a spanning subtree of} ~ (X, E)\}.
	\end{align*}
	$\K, \K^*$ are both $G_\delta$.
\end{eg}

\begin{eg}[$\Z$-lines]\label{eg:zline}
	\begin{align*}
		\K &= \{(X, L) \mid (X, L) ~ \text{is a linear order without endpoints}\},\\
		\K^* &= \{(X, L, Z) \mid (X, L) \in \K \And Z \subseteq X \And (Z, \res{L}{Z}) \cong (\Z, <)\},
	\end{align*}
	where $<$ is the usual order on $\Z$. Here $\K$ is $G_\delta$, and $\K^*$ is Borel.
\end{eg}

\begin{eg}[Vizing's Theorem]\label{eg:vizing}
	Fix $d \geq 2$, and let
	\begin{align*}
		\K &= \{(X, E) \mid (X, E) ~ \text{is a connected graph of max degree} ~ \leq d\},\\
		\K^* &= \{(X, E, S_0, \dots, S_d) \mid (X, E) \in \K \And S_0, \dots, S_d ~ \text{is an edge colouring of} ~ (X, E)\}.
	\end{align*}
	Here $\K, \K^*$ are $G_\delta$. Vizing's Theorem states that every element of $\K$ admits an expansion in $\K^*$.
\end{eg}

\begin{eg}[Matchings]\label{eg:matchings}
	A bipartite graph is said to satisfy \tb{Hall's Condition} if $|A| \leq |N(A)|$ for every finite set of vertices $A$ contained in one part of the graph, where $N(A)$ denotes the set of neighbours of $A$. We say a graph is \tb{locally-finite} if every vertex has finitely-many neighbours. Let now
	\begin{align*}
		\K &= \{(X, E) \mid (X, E) ~ \text{is a connected, bipartite,}\\
		&\qquad\qquad\qquad\text{locally finite graph satisfying Hall's Condition}\},\\
		\K^* &= \{(X, E, M) \mid (X, E) \in \K \And M \subseteq E ~ \text{is a perfect matching}\}.
	\end{align*}
	Here $\K, \K^*$ are Borel, and by Hall's Theorem every element of $\K$ admits an expansion in $\K^*$.
\end{eg}

\subsection{Equivariant and random expansions}\label{sec:def-equivariant-expansions}

Let $(\K, \K^*)$ be an expansion problem, $X$ be a countably infinite set, $\Gamma \leq S_X$ be a subgroup of $S_X$ and $Z \subseteq \K(X)$ be $\Gamma$-invariant. We say a map $f: Z \to \K^*(X)$ is \tb{$\Gamma$-equivariant} if it commutes with the $\Gamma$ action, i.e., $\gamma \cdot f(\bA) = f(\gamma \cdot \bA)$ for all $\gamma \in \Gamma, \bA \in \K(X)$. We call a function $f: Z \to \K^*(X)$ an \tb{expansion map} if $f(\bA)$ is an expansion of $\bA$ for all $\bA \in \K(X)$.

In this paper, we will always consider the case where $\Gamma$ is a countably infinite group acting on $X = \Gamma$ by multiplication on the left. It may also be interesting to consider the more general setting where the action of $\Gamma$ on $X$ is not free, though we do not explore this here.

Let $\Gamma$ be a countably infinite group. Given a $\Gamma$-invariant Borel set $Z \subseteq \K(\Gamma)$, we say \tb{$Z$ admits $\Gamma$-equivariant expansions to $\K^*$} if there is a Borel $\Gamma$-equivariant expansion map $Z \to \K^*(\Gamma)$. If $Z = \K(\Gamma)$, we say \tb{$\K$ admits $\Gamma$-equivariant expansions to $\K^*$}. If $\K$ is a $G_\delta$ class of structures, we say \tb{$\K$ admits $\Gamma$-equivariant expansions to $\K^*$ generically} if $Z$ admits a $\Gamma$-equivariant expansion to $\K^*$ for some $\Gamma$-invariant dense $G_\delta$ set $Z \subseteq \K(\Gamma)$.

We let $P(\K(\Gamma))$ denote the space of probability Borel measures on $\K(\Gamma)$. Note that the action of $\Gamma$ on $\K(\Gamma)$ gives rise to an action of $\Gamma$ on $P(\K(\Gamma))$, where $\gamma \cdot \mu = \gamma_* \mu$ is the push-forward of $\mu$ under $\gamma: \K(\Gamma) \to \K(\Gamma)$. We say $\mu$ is \tb{$\Gamma$-invariant} if it is fixed by the $\Gamma$-action, in which case we say $\mu$ is an \tb{invariant random $\K$-structure on $\Gamma$}.

If $\mu$ is an invariant random $\K$-structure on $\Gamma$, we say \tb{$\K$ admits $\Gamma$-equivariant expansions to $\K^*$ $\mu$-a.e.} if $Z$ admits a $\Gamma$-equivariant expansion to $\K^*$ for some $\Gamma$-invariant $\mu$-conull set $Z \subseteq \K(\Gamma)$.

Note that the reduct $\pi: \K^*(\Gamma) \to \K(\Gamma)$ induces a map $\pi_*: P(\K^*(\Gamma)) \to P(\K(\Gamma))$. If $\mu$ (resp. $\nu$) is an invariant random $\K$-structure (resp. $\K^*$-structure) on $\Gamma$, we say $\nu$ is a \tb{$\Gamma$-invariant random expansion of $\mu$ to $\K^*$} if $\pi_* \nu = \mu$. We say \tb{$\mu$ admits a $\Gamma$-invariant random expansion to $\K^*$} if such a $\nu$ exists, and that \tb{$\Gamma$ admits random expansions from $\K$ to $\K^*$} if this holds for all such $\mu$. Note that if $\K$ admits $\Gamma$-equivariant expansions to $\K^*$ $\mu$-a.e., then $\mu$ admits a $\Gamma$-invariant random expansion to $\K^*$.

We may omit $\Gamma$ from these definitions when it is clear from context.

\subsection{Countable Borel equivalence relations}

A \tb{countable Borel equivalence relation (CBER)} is an equivalence relation $E$ on a standard Borel space $X$ which is Borel as a subset of $X^2$, and whose equivalence classes $[x]_E$ are countable for all $x \in X$.

If $\Gamma$ is a group acting on a set $X$, we let $E^X_\Gamma \subseteq X^2$ denote the \tb{orbit equivalence relation}
\[x E^X_\Gamma y \iff \exists \gamma \in \Gamma (\gamma \cdot x = y)\]
induced by the action of $\Gamma$ on $X$. When $\Gamma$ is countable, $X$ is standard Borel and the action of $\Gamma$ is Borel, then $E^X_\Gamma$ is a CBER. Conversely, by the Feldman--Moore Theorem \cite{FM}, every CBER $E$ on a standard Borel space $X$ is the orbit equivalence relation induced by a Borel action of some countable group $\Gamma$ on $X$.

By the \tb{free part} of an action of $\Gamma$ on $X$ we mean the set
\[Fr(X) = \{x \in X : \forall \gamma \neq 1_\Gamma (\gamma \cdot x \neq x)\}.\]
Note that when $X$ is standard Borel and the action is Borel, $Fr(X)$ is Borel in $X$. Moreover, if $X$ is Polish and the action of $\Gamma$ is continuous, then $Fr(X)$ is $G_\delta$ in $X$, hence Polish in the subspace topology.

Given a CBER $E$ on $X$ and $A \subseteq X$, we let $[A]_E = \{x \in X : \exists y \in A (x E y)\}$ denote the \tb{($E$-)saturation} of $A$. We say $A$ is \tb{$E$-invariant} if $[A]_E = A$. Note that by the Feldman--Moore Theorem, if $A$ is Borel then so is $[A]_E$. We call $A$ a \tb{complete section} for $E$ if $X = [A]_E$.

Let $E, F$ be CBER on $X, Y$ respectively. We say that a Borel map $f: X \to Y$ is a \tb{homomorphism} from $E$ to $F$, denoted $f: E \to_B F$, if $x E y \implies f(x) F f(y)$. It is a \tb{reduction} $f: E \leq_B F$ if the converse holds as well, i.e., $x E y \iff f(x) F f(y)$; an \tb{embedding} $f: E \sqsubseteq_B F$ if it is an injective reduction; an \tb{isomorphism} $f: E \cong_B F$ if it is a surjective embedding; a \tb{class-bijective homomorphism} $f: E \to^{cb}_B F$ if it is a homomorphism for which $f: [x]_E \to [f(x)]_F$ is a bijection for all $x \in X$; and an a \tb{invariant embedding} $f: E \sqsubseteq^i_B F$ if it is a class-bijective reduction.

A CBER $E$ is \tb{finite} if all of its classes are finite, and \tb{aperiodic} if all of its classes are infinite. Given a CBER $E$ on $X$, we can always partition $X$ into Borel pieces $Y, Z$ on which $E$ is respectively finite and aperiodic.

A CBER $E$ is \tb{smooth} if $E \leq_B \Delta_Y$, where $Y$ is equality on a standard Borel space $Y$. Equivalently, $E$ is smooth iff it admits a Borel \tb{selector}, i.e., a Borel map $s: X \to X$ which is \tb{$E$-invariant} (a homomorphism $s: E \to \Delta_X$) and so that $s(x) E x$ for all $x \in X$. The \tb{Glimm--Effros Dichotomy} for CBER states that for every CBER $E$, either $E$ is smooth or $E_0 \sqsubseteq_B E$, where $E_0$ is the eventual equality relation
\[x E_0 y \iff \exists n \forall k (x_{n+k} = y_{n+k})\]
on $2^\N$, c.f. \cite[Theorem~6.5]{CBER}.

We say $E$ is \tb{hyperfinite} if $E$ can be written as an increasing union of finite CBER; see \cite[Theorem~8.2]{CBER} for alternate characterizations of hyperfiniteness. In particular, $E_0$ and all smooth CBER are hyperfinite.

An \tb{invariant probability measure} for a Borel action of a countable group $\Gamma$ on $X$ is a probability Borel measure $\mu$ on $X$ such that $\gamma_* \mu = \mu$ for all $\gamma \in \Gamma$, where $\gamma_* \mu$ is the push-forward of $\mu$ along $\gamma$. An \tb{invariant probability measure} for a CBER $E$ on $X$ is a probability Borel measure $\mu$ on $X$ such that $f_* \mu = \mu$ for every Borel bijection $f: X \to X$ whose graph is contained in $E$. If $E$ is the orbit equivalence relation induced by a Borel action of a countable group $\Gamma$, then these two notions coincide; see \cite[Proposition~2.1]{KM}. A probability Borel measure $\mu$ on $X$ is \tb{$E$-ergodic} if $\mu(A) \in \{0, 1\}$ for all $E$-invariant Borel sets $A$.

If $\Gamma$ is a countable group and $X$ is standard Borel, the \tb{shift action} of $\Gamma$ on $X^\Gamma$ is given by $(\gamma \cdot y)(\delta) = y(\gamma^{-1}\delta)$ for $\gamma, \delta \in \Gamma$ and $y \in X^\Gamma$. If $\mu$ is any probability Borel measure on $X$, then the product measure $\mu^\Gamma$ is an invariant probability measure on $X^\Gamma$ which concentrates on the free part, i.e., $\mu^\Gamma(Fr(X^\Gamma)) = 1$.

We say $E$ is \tb{generically ergodic} if $A$ is either meagre or comeagre for all $E$-invariant Borel sets $A$. For example, $E_0$ is generically ergodic, as are the orbit equivalence relations induced by the shift actions of countably infinite groups.

A CBER $E$ on $X$ is \tb{compressible} if there is a Borel map $f: X \to X$ whose graph is contained in $E$, and so that $f(C) \subsetneqq C$ for every $E$-class $C \in X/E$. By \tb{Nadkarni's Theorem}, $E$ is compressible iff it does not admit an invariant probability measure \cite{N,BK}.

\subsection{Structures and expansions on CBER}\label{sec:structures-and-expansions-on-CBER}

Fix a language $\LL$ and a class $\K$ of $\LL$-structures.

Let $E$ be an aperiodic CBER on a standard Borel space $X$. A \tb{(Borel) $\LL$-structure on $E$} is an $\LL$-structure $\bbA = (X, R^{\bbA})_{R \in \LL}$ with universe $X$ so that (a) $R^{\bbA} \subseteq X^n$ is Borel for each $n$-ary $R \in \LL$, and (b) each $R^{\bbA}$ only relates elements in the same $E$-class, i.e.,
\[R^{\bbA}(x_0, \dots, x_{n-1}) \implies x_0 E \cdots E x_{n-1}.\]
Given such a structure $\bbA$ and an $E$-class $C \in X/E$, we let $\res{\bbA}{C}$ denote the \tb{restriction of $\bbA$ to $C$}, which is a countable $\LL$-structure. We call $\bbA$ a \tb{Borel $\K$-structuring of $E$} if $\res{\bbA}{C} \in \K$ for every $C \in X/E$.

Consider now $\LL^* \supseteq \LL$ and a class $\K^*$ of $\LL^*$-structures. If $\bbA^*$ is an $\LL^*$-structure on $E$, the reduct $\res{\bbA^*}{\LL}$ of $\bbA$ is an $\LL$-structure on $E$. We say that $\bbA^*$ is an \tb{expansion} of an $\LL$-structure $\bbA$ on $E$ if $\res{\bbA^*}{\LL} = \bbA$.

Let $\bbA$ be a Borel $\K$-structuring of $E$. We say $\bbA$ is \tb{Borel expandable to $\K^*$} if it admits an expansion which is a Borel $\K^*$-structuring of $E$; we say $E$ is \tb{Borel expandable for $(\K, \K^*)$} if this holds for all such $\bbA$. When $E$ lives on a Polish space $X$, we say $\bbA$ is \tb{generically expandable to $\K^*$} if its restriction to a comeagre invariant Borel set is Borel expandable to $\K^*$; we say $E$ is \tb{generically expandable for $(\K, \K^*)$} if this holds for all such $\bbA$. If $\mu$ is a probability Borel measure on $X$, we say $\bbA$ is \tb{$\mu$-a.e. expandable to $\K^*$} if its restriction to a $\mu$-conull invariant Borel set is Borel expandable to $\K^*$; we say $(E, \mu)$ is \tb{a.e. expandable for $(\K, \K^*)$} this holds for all such $\bbA$.

\section{General results}\label{sec:general-results}

In this section, we assume that all classes of structures are Borel.

\subsection{Universal structurings of CBER}\label{sec:descriptions}

We begin by describing an alternate characterization of Borel structures and expansions on CBER which will be useful later; see also \cite[Definition~3.1]{BC}.

Let $E$ be an aperiodic CBER on a standard Borel space $X$. A \tb{Borel family of enumerations of $E$} is a Borel map $g: X \to X^N$, where $N$ is some countably infinite set, so that $g_x = g(x): N \to X$ is a bijection of $N$ with $[x]_E$ for all $x \in X$; such maps always exist by the Lusin--Novikov Theorem (c.f. \cite[\nopp18.15]{CDST}). A map $\rho: E \to G$ from $E$ to a group $G$ is a \tb{cocycle} if
\[\rho(y, z) \rho(x, y) = \rho(x, z)\]
for all $x E y E z$. If $g: X \to X^N$ is a Borel enumeration of $X$, there is an associated Borel cocycle $\rho_g: E \to S_N$ given by $\rho_g(x, y) = g_y^{-1} \circ g_x$.

Let $\LL$ be a language, $\K$ be a class of $\LL$-structures and $g: X \to X^N$ be a Borel enumeration of $E$. Given an $\LL$-structure $\bbA$ on $E$, one gets a map $F = F_g^\bbA: X \to \Mod_{\LL}(N)$ given by setting
\[R^{F(x)}(n_1, \dots, n_k) \iff R^{\bbA}(g_x(n_1), \dots, g_x(n_k))\]
for $k$-ary $R \in \LL$ and $n_1, \dots, n_k \in N$. Note that if $x E y$ then
\begin{align*}
	R^{\rho_g(x, y) \cdot F(x)}(n_1, \dots, n_k) &\iff R^{F(x)}(\rho_g(x, y)^{-1}(n_1), \dots, \rho_g(x, y)^{-1}(n_k)) \\
	&\iff R^\bbA(g_x(\rho_g(x, y)^{-1}(n_1)), \dots, g_x(\rho_g(x, y)^{-1}(n_k))) \\
	&\iff R^\bbA(g_y(n_1), \dots, g_y(n_k)) \\
	&\iff R^{F(y)}(n_1, \dots, n_k),
\end{align*}
that is, $\rho_g(x, y) \cdot F(x) = F(y)$. Call such a map \tb{$\rho_g$-equivariant}. Note that $g_x: F(x) \cong \res{\bbA}{[x]_E}$ for $x \in X$.

Conversely, given a $\rho_g$-equivariant map $F: X \to \Mod_{\LL}(N)$, one can define a Borel $\LL$-structure $\bbA$ on $E$ by setting
\[R^\bbA(x_1, \dots, x_k) \iff R^{F(x)}(g_x(x_1)^{-1}, \dots, g_x^{-1}(x_k))\]
for any $k$-ary $R \in \LL$ and $x E x_1 E \dots E x_k$. It is easy to verify, using the fact that $\rho_g(x, y) F(x) = F(y)$ for $x E y$, that this definition does not depend on the choice of $x$ and that $g_x: F(x) \cong \res{\bbA}{[x]_E}$ for $x \in X$.

We therefore have that the map $\bbA \mapsto F^\bbA_g$ is a bijective correspondence between $\LL$-structures on $E$ and $\rho_g$-equivariant maps $F: X \to \Mod_{\LL}(N)$. It is easy to see that $\bbA$ is Borel iff $F^\bbA_g$ is Borel, and that $\bbA$ is a $\K$-structuring of $E$ iff the image of $F^\bbA_g$ is contained in $\K(N)$, whenever $\K$ is a class of $\LL$-structures.

We remark that if $\LL \subseteq \LL^*$ are languages and $\bbA, \bbA^*$ are $\LL, \LL^*$-structures on $E$, then $\bbA^*$ is an expansion of $\bbA$ iff $F^\bbA_g = \pi \circ F^{\bbA^*}_g$, where $\pi: \Mod_{\LL^*}(N) \to \Mod_{\LL}(N)$ is the reduct. That is, $\bbA$ admits a Borel expansion iff there is a $\rho_g$-equivariant Borel lift $F$ of $F^\bbA$ to $\K^*(X)$:
\[
\begin{tikzcd}
	& K^*(X) \ar[d, "\pi"] \\
	X \ar[r, "F^\bbA"] \ar[ur, dashed, "F"] & K(X)
\end{tikzcd}
\]

Fix now a class of $\LL$-structures $\K$. In \cite[Theorem~4.1, Remark~4.3]{CK}, a \emph{universal} $\K$-structurable CBER lying over $E$ is constructed, which we denote $E \ltimes_g \K$. Explicitly, $E \ltimes_g \K$ lives on $X \times \K(N)$, and is given by
\[(x, \bA) (E \ltimes_g \K) (y, \bB) \iff x E y \And \rho_g(x, y) \cdot \bA = \bB.\]
The projection $\pi_0: X \times \K(N) \to X$ is a class-bijective homomorphism $E \ltimes_g \K \to E$, along which $g$ can be pulled back to a Borel enumeration $\tilde{g}$ of $E \ltimes_g \K$. If $\rho_{\tilde{g}}$ is the associated Borel cocycle, then $\rho_{\tilde{g}} = \rho_g \circ (\pi_0 \times \pi_0)$, and it follows that
\[\rho_{\tilde{g}}((x, \bA), (y, \bB)) \cdot \bA = \rho_g(x, y) \cdot \bA = \bB.\]
The canonical $\K$-structure $\bbA$ on $E \ltimes_g \K$ is then the one induced by the projection $X \times \K(N) \to \K(N)$, which we have observed is $\rho_{\tilde{g}}$-equivariant. Note that while $E \ltimes_g \K$ depends only on the cocycle $\rho_g$, $\bbA$ depends on $g$.

The above constructions do not depend on the choice of $g$, up to canonical Borel isomorphism. That is, given Borel enumerations $g: X \to X^M$, $h: X \to X^N$ of $E$, let $\tau(x) = \tau_{g, h}(x) = h_x^{-1} \circ g_x$. Then $\tau: X \to N^M$ is Borel and each $\tau(x): M \to N$ is a bijection. Moreover, we have
\[\tau(y) \circ \rho_g(x, y) = \rho_h(x, y) \circ \tau(x)\]
for $x E y$ (i.e., $\rho_g, \rho_h$ are cohomologous, as witnessed by $\tau$). In particular, if $\bbA$ is a Borel $\LL$-structure on $E$ and $F^\bbA_g, F^\bbA_h$ the corresponding maps, then $F^\bbA_h(x) = \tau(x) \cdot F^\bbA_g(x)$ for $x \in X$. Additionally, the map $(x, \bA) \mapsto (x, \tau(x) \cdot \bA)$ is a Borel isomorphism $E \ltimes_g \K \cong_B E \ltimes_h \K$.

\subsection{Expansions on CBER induced by free actions}\label{sec:free-actions}

Fix an expansion problem $(\K, \K^*)$ and a countably infinite group $\Gamma$. We wish to relate the existence of $\Gamma$-equivariant expansions with Borel expansions on CBER induced by free actions of $\Gamma$.

Suppose $E = E^X_\Gamma$ is a CBER induced by a free Borel action of $\Gamma$ on a standard Borel space $X$. This gives rise to a Borel enumeration $g: X \to X^\Gamma$ of $E$, namely $g_x(\gamma) = \gamma^{-1} \cdot x$. In this case, $\rho_g(x, \gamma x) = \gamma \in S_\Gamma$, so $\rho_g: E \to \Gamma$. In particular, we have that $E \ltimes_g \K = E^{X \times \K(\Gamma)}_\Gamma$ is the orbit equivalence relation induced by the diagonal action of $\Gamma$ on $X \times \K(\Gamma)$.

Thus, the characterization of Borel $\K$-structurings of $E$ described in \cref{sec:descriptions} gives:

\begin{prop}\label{prop:correspondence-sructures-maps}
	Let $\LL$ be a language, $\K$ be a Borel class of $\LL$-structures and $\Gamma$ be a countably infinite group. Fix a free Borel action of $\Gamma$ on a standard Borel space $X$. There is a canonical bijection $\bbA \mapsto F^\bbA$ between the set of Borel $\K$-structurings of $E^X_\Gamma$ and the set of $\Gamma$-equivariant Borel maps $X \to \K(\Gamma)$, defined by setting
	\begin{align*}\label{eq:correspondence}
		R^{F^\bbA(x)}(\gamma_1, \dots, \gamma_n) \iff R^\bbA(\gamma_1^{-1} \cdot x, \dots, \gamma_n^{-1} \cdot x)
	\end{align*}
	for $x \in X$, Borel $\K$-structurings $\bbA$ on $E^X_\Gamma$, $n$-ary relation symbols $R \in \LL$ and $\gamma_1, \dots, \gamma_n \in \Gamma$.
\end{prop}

\begin{prop}\label{prop:correspondence-expansion-maps}
	Let $(\K, \K^*)$ be an expansion problem and $\Gamma$ be a countably infinite group. Fix a free Borel action of $\Gamma$ on a standard Borel space $X$ and a Borel $\K$-structuring $\bbA$ of $E^X_\Gamma$, and let $f: X \to \K(\Gamma)$ be the corresponding equivariant Borel map. There is a canonical bijection between Borel expansions $\bbA^*$ of $\bbA$ and equivariant Borel maps $g: X \to \K^*(\Gamma)$ satisfying $\pi \circ g = f$, where $\pi: \K^*(\Gamma) \to \K(\Gamma)$ is the reduct from $\LL^*$ to $\LL$.
\end{prop}

\begin{rmk}
	If $X \subseteq Fr(\K(\Gamma))$ is invariant and Borel then there is a canonical Borel $\K$-structuring of $E^X_\Gamma$ corresponding to the inclusion $X \to \K(\Gamma)$. More generally, if $Z$ is a standard Borel space on which $\Gamma$ acts and $X \subseteq Z \times \K(\Gamma)$ is Borel, invariant and free for the diagonal $\Gamma$ action on the product, then there is a canonical Borel $\K$-structuring of $E^X_\Gamma$ corresponding to the projection $Z \times \K(\Gamma) \supseteq X \to \K(\Gamma)$.
\end{rmk}

In particular, this gives a weak correspondence between equivariant expansions on $\Gamma$ and expansions of CBER induced by free actions of $\Gamma$.

\begin{prop}\label{prop:weak-correspondence-expansions}
	Let $(\K, \K^*)$ be an expansion problem and $\Gamma$ be a countably infinite group.
	\begin{enumerate}[label=(\arabic*)]
		\item If $\K(\Gamma)$ admits a Borel equivariant expansion, then every CBER induced by a free Borel action of $\Gamma$ is Borel expandable.

		\item An invariant Borel set $X \subseteq Fr(\K(\Gamma))$ admits a Borel equivariant expansion iff the canonical $\K$-structuring of $E^X_\Gamma$ admits a Borel expansion.

		\item Suppose there is a free Borel action of $\Gamma$ on a standard Borel space $Z$ admitting an invariant measure $\mu$, and so that the canonical $\K$-structure $\bbA$ on $Z \times \K(\Gamma)$ is $\lambda$-a.e. expandable for every $\Gamma$-invariant probability Borel measure $\lambda$ on $Z \times \K(\Gamma)$ whose push-forward to $Z$ is $\mu$. Then $\Gamma$ admits random expansions.
	\end{enumerate}
\end{prop}

\begin{proof}
	(1) Let $h$ be such an expansion. For any Borel equivariant $f: X \to \K(\Gamma)$, $h \circ f: X \to \K^*(\Gamma)$ is Borel, equivariant, and satisfies $\pi \circ h \circ f = f$, so this follows by \cref{prop:correspondence-expansion-maps}.

	(2) This follows immediately from \cref{prop:correspondence-expansion-maps}.

	(3) For any invariant random $\K$-structure $\nu$ on $\Gamma$, take $\lambda = \nu \times \mu$. Then $\lambda$ is $\Gamma$-invariant, so there is a $\lambda$-conull invariant Borel set $X \subseteq \K(\Gamma) \times Z$ and an expansion $\bbA^*$ of $\res{\bbA}{X}$. Let $f: X \to \K^*(\Gamma)$ be the corresponding equivariant map, and let $\kappa = f_*(\lambda)$. Then $\kappa$ is an invariant random $\K^*$-structure on $\Gamma$, and $\pi_* \kappa = (\pi \circ f)_* \lambda = (\proj_{\K(\Gamma)})_* \lambda = \nu$.
\end{proof}

\subsection{Uniform random expansions}\label{sec:canonical-random-expansions}

Let $(\K, \K^*)$ be an expansion problem, $\Gamma$ be a countably infinite group, and $\mu$ be an invariant random $\K$-structure on $\Gamma$. In general, it is possible that $\mu$ admits an invariant random expansion to $\K^*$, but $\K$ does not admit equivariant expansions to $\K^*$ $\mu$-a.e. That is, there may be some invariant random expansion $\nu$ of $\mu$ to $\K^*$, but no Borel equivariant expansion map $f: \K(\Gamma) \supseteq Z \to \K^*(\Gamma)$, defined on a $\mu$-conull set $Z$, so that $\nu = f_* \mu$ (see e.g. \cref{rmk:no-canonical-linearization}).

On the other hand, we will see that for \cref{eg:bijections,eg:colourings,eg:zline}, every invariant random $\K$-structure on $\Gamma$ which admits an invariant random expansion to $\K^*$ admits such an expansion of the form $f_* \mu$, where $f: \K(\Gamma) \supseteq Z \to \K^*(\Gamma)$ is a Borel equivariant expansion map defined on a $\mu$-conull set $Z$. Moreover, in these cases, we shall see that this holds \emph{uniformly in $\mu$}: there is a single function $f$ that works for all such $\mu$. That is, there is an invariant Borel set $Z \subseteq \K(\Gamma)$ and Borel equivariant expansion map $f: Z \to \K^*(\Gamma)$ so that $\mu$ admits an invariant random expansion to $\K^*$ iff $\mu(Z) = 1$, in which case $f_* \mu$ gives such an expansion.

One can view such a function $f$ as both classifying exactly when invariant random expansions exist, as well as giving uniformly all possible invariant random expansions. When such an $f$ exists, one can further characterize exactly when a Borel $\K$-structuring of a CBER is a.e. expandable, for any CBER induced by a free Borel action of $\Gamma$:

\begin{prop}\label{prop:canonical-random-expansions}
	Let $(\K, \K^*)$ be an expansion problem and $\Gamma$ be a countably infinite group. Suppose that there is a Borel $\Gamma$-invariant set $Z \subseteq \K(\Gamma)$ which admits a $\Gamma$-equivariant expansion to $\K^*$, and such that the following holds: For every invariant random $\K$-structure $\mu$ on $\Gamma$, $\mu$ admits an invariant random expansion to $\K^*$ iff $\mu(Z) = 1$.

	Then for any CBER $E$ on $X$ induced by a free Borel action of $\Gamma$, any $E$-invariant probability Borel measure $\mu$ on $X$ and any Borel $\K$-structuring $\bbA$ of $E$, $\bbA$ is $\mu$-a.e. expandable to $\K^*$ if and only if $F^\bbA(x) \in Z$ for $\mu$-almost every $x \in X$, where $F^\bbA: X \to \K(\Gamma)$ is the equivariant map from \cref{prop:correspondence-sructures-maps}.
\end{prop}

\begin{proof}
	If $F^\bbA(x) \in Z$ for $\mu$-almost every $x \in X$, then by \cref{prop:correspondence-expansion-maps}, the composition of $F^\bbA$ with the expansion map $Z \to \K^*(\Gamma)$ gives a Borel expansion of $\bbA$ to $\K^*$ on an invariant $\mu$-conull set.

	Conversely, suppose that $Y \subseteq X$ is an $E$-invariant Borel $\mu$-conull set and $\bbA^*$ is a Borel expansion of $\res{\bbA}{Y}$ to $\K^*$. Let $\pi: \K^*(\Gamma) \to \K(\Gamma)$ denote the reduct, and note that $F^\bbA = \pi \circ F^{\bbA^*}$, so that $F^{\bbA^*}_* \mu$ is an invariant random expansion of $F^\bbA_* \mu$. By assumption, this implies that $F^\bbA_* \mu (Z) = 1$, i.e., $F^\bbA(x) \in Z$ for $\mu$-a.e. $x \in X$.
\end{proof}

\subsection{Invariant random expansions on CBER}\label{sec:inv-rand-expansion-cber}

We consider now a notion of invariant random structures and random expansions on CBER. When $E$ is a CBER arising from a free Borel action of $\Gamma$ with an invariant measure, this will correspond to a weakening of the hypotheses of \cref{prop:weak-correspondence-expansions}(3), and we will show in this case that the existence of random expansions on $E$ corresponds exactly to the existence of random expansions on $\Gamma$. Crucially, this notion of random expansion will be purely in terms of the CBER with no reference to $\Gamma$, allowing us to compare the existence of random expansion between various groups.

Let $(\K, \K^*)$ be an expansion problem, fix an aperiodic CBER $E$ on a standard Borel space $X$ and let $g: X \to X^N$ be a Borel enumeration of $E$.

An \tb{invariant random $\K$-structuring of $E$ (with respect to $g$)} is an invariant probability Borel measure for $E \ltimes_g \K$. If $\mu$ is an invariant probability Borel measure for $E$, an \tb{invariant random $\K$-structuring of $(E, \mu)$ (with respect to $g$)} is an invariant random $\K$-structuring $\nu$ of $E$ whose push-forward along the projection $X \times \K(N) \to X$ is $\mu$. We say an invariant random $\K^*$-structuring $\kappa$ of $E$ is an \tb{expansion} of an invariant random $\K$-structuring $\nu$ of $E$ if the push-forward of $\kappa$ along the reduct $X \times \K^*(N) \to X \times \K(N)$ is $\nu$.

If $h: X \to X^M$ is another Borel enumeration of $E$ and $\tau = \tau_{g, h}: X \to M^N$ the induced Borel map described in \cref{sec:descriptions}, we recall that $(x, \bA) \mapsto (x, \tau(x) \cdot \bA)$ is a Borel isomorphism $E \ltimes_g \K \cong_B E \ltimes_h \K$. It is easy to see that the following diagram commutes, where $\pi$ denotes the reduct from $\K^*$ to $\K$.
\[
\begin{tikzcd}
	X \times \K^*(N) \ar[r, "{(\mathrm{id}_X, \pi)}"]\ar[d, "{(\mathrm{id}_X, \tau)}"] & X \times \K(N) \ar[r, "\proj_X"]\ar[d, "{(\mathrm{id}_X, \tau)}"] & X \ar[d, "\mathrm{id}_X"] \\
	X \times \K^*(M) \ar[r, "{(\mathrm{id}_X, \pi)}"] & X \times \K(M) \ar[r, "\proj_X"] & X
\end{tikzcd}
\]

If $\mu$ is an invariant probability Borel measure on $E$, we say \tb{$(E, \mu)$ admits random expansions from $\K$ to $\K^*$} if for every invariant random $\K$-structure $\nu$ on $(E, \mu)$ there is an invariant random $\K^*$-structure $\kappa$ on $(E, \mu)$ that is an expansion of $\nu$. We say \tb{$E$ admits random expansions from $\K$ to $\K^*$} if $(E, \mu)$ admits invariant random expansions from $\K$ to $\K^*$ for every invariant probability Borel measure $\mu$ on $E$. By the prior remarks, these definitions do not depend on the choice of Borel enumeration of $E$.

The following key fact relates invariant random expansions between CBER and groups. Special cases of this have been shown, for example with linearizations in \cite{A}; we note here that it holds more generally for all expansion problems.

\begin{prop}\label{prop:random-expansion-cber}
	Let $(\K, \K^*)$ be an expansion problem and $\Gamma$ be a countably infinite group. The following are equivalent:
	\begin{enumerate}
		\item $\Gamma$ admits random expansions from $\K$ to $\K^*$;
		\item Every CBER $E$ induced by a free Borel action of $\Gamma$ admits random expansions from $\K$ to $\K^*$;
		\item There is a CBER $E$ induced by a free Borel action of $\Gamma$ and an $E$-invariant probability Borel measure $\mu$, such that $(E, \mu)$ admits random expansions from $\K$ to $\K^*$.
	\end{enumerate}
\end{prop}

\begin{proof}
	$(1) \implies (2)$: Let $E$ be a CBER on $X$ induced by a free Borel action of $\Gamma$. By the remarks at the start of \cref{sec:free-actions}, we may assume that $E \ltimes \K$ is the orbit equivalence relation on $X \times \K(\Gamma)$ arising from the diagonal $\Gamma$ action. If $\nu$ is an invariant probability Borel measure on $X \times \K(\Gamma)$, then the push-forward of $\nu$ along the projection $X \times \K(\Gamma) \to \K(\Gamma)$ gives an invariant random $\K$-structure $\nu'$ on $\Gamma$. By (1), there is an invariant random $\K^*$-structure $\kappa$ on $\K^*(\Gamma)$ which is an expansion of $\nu'$.

	Let now $\{\kappa_\bA\}_{\bA \in \K(\Gamma)}$ be the measure disintegration of $\kappa$ with respect to $\nu'$ over the reduct $\K^*(\Gamma) \to \K(\Gamma)$, and $\{\nu_\bA \times \delta_{\bA}\}_{\bA \in \K(\Gamma)}$ be a measure disintegration of $\nu$ with respect to $\nu'$ over the projection $X \times \K(\Gamma) \to \K(\Gamma)$ (see for example \cite[\nopp17.35]{CDST}). Then it is straightforward to check that $\lambda = \int (\nu_\bA \times \kappa_\bA) d \nu'(\bA)$ is a $\Gamma$-invariant probability Borel measure on $X \times \K^*(\Gamma)$ which is an extension of $\nu$.

	$(2) \implies (3)$: Consider e.g. the free part of the shift of $\Gamma$ on $2^\Gamma$.

	$(3) \implies (1)$: Suppose $(E, \mu)$ admits random expansions form $\K$ to $\K^*$, where $E$ is induced by a free Borel action of $\Gamma$ on $X$ preserving the probability Borel measure $\mu$. Let $\nu$ be an invariant random $\K$-structure on $\Gamma$, and consider $\mu \times \nu$ on $X \times \K(\Gamma)$. This is $\Gamma$-invariant, hence it admits an expansion $\kappa$, and the push-forward of $\kappa$ along the projection $X \times \K^*(\Gamma) \to \K^*(\Gamma)$ is an invariant random expansion of $\nu$.
\end{proof}

Recall now that two countably infinite groups $\Gamma, \Lambda$ are \tb{orbit equivalent} if there is a CBER $E$ on a standard Borel space $X$ admitting an invariant probability Borel measure and free Borel actions of $\Gamma, \Lambda$ on $X$ with $E = E^X_\Gamma = E^X_\Lambda$. A CBER is said to be \tb{measure-hyperfinite} if, for every invariant probability Borel measure $\mu$, it is hyperfinite when restricted to a Borel invariant $\mu$-conull set (see also \cite[Sections~8.5, 9]{CBER}). For example, all countable infinite amenable groups are orbit equivalent and their actions generate measure-hyperfinite CBER \cite{Dye,OW}.

\begin{thm}\label{thm:invariant-random-expansions}
	Let $(\K, \K^*)$ be an expansion problem and $\Gamma, \Lambda$ be countably infinite groups. If $\Gamma$, $\Lambda$ are orbit equivalent, then $\Gamma$ admits random expansions from $\K$ to $\K^*$ iff $\Lambda$ admits random expansions from $\K$ to $\K^*$.

	In particular, the following are equivalent:
	\begin{enumerate}
		\item $\Gamma$ admits random expansions from $\K$ to $\K^*$ for all amenable groups $\Gamma$;
		\item $\Gamma$ admits random expansions from $\K$ to $\K^*$ for some amenable group $\Gamma$; and
		\item every aperiodic measure-hyperfinite CBER admits random expansions from $\K$ to $\K^*$.
	\end{enumerate}
\end{thm}

\begin{proof}
	Let $E$ be a CBER on a standard Borel space $X$ admitting an invariant probability Borel measure $\mu$, and fix free Borel actions of $\Gamma, \Lambda$ on $X$ inducing $E$. Suppose $\Gamma$ admits random expansions from $\K$ to $\K^*$. By $(1) \implies (2)$ of \cref{prop:random-expansion-cber} $(E, \mu)$ admits random expansions from $\K$ to $\K^*$, and hence by $(3) \implies (1)$ of \cref{prop:random-expansion-cber} so does $\Lambda$.

	By \cref{prop:random-expansion-cber}, it is clear that $(1) \iff (2) \impliedby (3)$. If $E$ is an aperiodic measure-hyperfinite CBER and $\mu$ is an $E$-invariant probability Borel measure, then by restricting to a $\mu$-conull set we see that $E$ is hyperfinite, hence generated by a free $\Z$ action. Thus $(1) \implies (3)$ by \cref{prop:random-expansion-cber}.
\end{proof}

We note that converse holds as well:

\begin{prop}\label{prop:converse-ire}
	Let $\Gamma$, $\Lambda$ be countably infinite groups and suppose that for every expansion problem $(\K, \K^*)$, $\Gamma$ admits random expansions from $\K$ to $\K^*$ iff $\Lambda$ admits random expansions from $\K$ to $\K^*$. Then $\Gamma$, $\Lambda$ are orbit equivalent.
\end{prop}

\begin{proof}
	Let $\K$ be the class of countable sets (in the empty language), and take $\K^*$ to be a class of structures so that $\K^*$-structurability of a CBER $E$ corresponds exactly to $E$ arising from a free Borel action of $\Gamma$ (see e.g. \cite[Section~3.1]{CK}).

	We claim that $\Gamma$ admits an invariant random $\K^*$-structure. To see this, fix a free Borel action of $\Gamma$ on a standard Borel space $X$ preserving a probability Borel measure $\mu$ (e.g. the free part of the shift on $2^\Gamma$). By definition, $E^X_\Gamma$ admits a Borel $\K^*$-structure $\bbA^*$. By \cref{prop:correspondence-sructures-maps} this corresponds to a $\Gamma$-equivariant Borel map $F^\bbA: X \to \K^*(\Gamma)$, so $F^\bbA_* \mu$ is an invariant random $\K^*$-structure on $\Gamma$. As $\K(\Gamma)$ is a singleton, $\Gamma$ admits invariant random expansions from $\K$ to $\K^*$.

	Thus, by our assumption, $\Lambda$ admits random expansions from $\K$ to $\K^*$. Fix a free Borel action of $\Lambda$ on a standard Borel space $Y$ preserving a probability Borel measure $\nu$. By $(1) \implies (2)$ of \cref{prop:random-expansion-cber} there is an $(E \ltimes \K^*)$-invariant probability Borel measure $\kappa$ on $Y \times \K^*(\Lambda)$. Note that $E \ltimes \K^*$ admits a Borel $\K^*$-structuring, hence arises from a free Borel action of $\Gamma$. Thus the actions of $\Gamma, \Lambda$ on $Y \times \K^*(\Lambda)$ witness the orbit equivalence of $\Gamma, \Lambda$.
\end{proof}

\begin{rmk}
	This is really an observation about invariant random structures on groups: Two countably infinite groups $\Gamma, \Lambda$ are orbit equivalent if and only if, for every class $\K$ of structures, there is an invariant random $\K$-structure on $\Gamma$ exactly when there is an invariant random $\K$-structure on $\Lambda$.
\end{rmk}

Finally, we remark that random expansions always exist for closed classes of structures on amenable groups.

\begin{prop}
	Let $(\K, \K^*)$ be an expansion problem. If $\K^*$ is closed, then $\Gamma$ admits random expansions from $\K$ to $\K^*$ for every countably infinite amenable group $\Gamma$.
\end{prop}

\begin{proof}
	Let $\mu$ be an invariant random $\K$-structure on $\Gamma$, and let $\pi: \K^*(\Gamma) \to \K(\Gamma)$ denote the reduct. Then $\pi_*: P(\K^*(\Gamma)) \to P(\K(\Gamma))$, where $P(X)$ denotes the space of probability Borel measures on a space $X$. Note that $\pi_*$ is continuous and $P(\K^*(\Gamma))$ is compact metrizable \cite[\nopp17.22, 17.28]{CDST}, so in particular $A = (\pi_*)^{-1}(\mu) \subseteq P(\K^*(\Gamma))$ is compact metrizable. By the invariance of $\mu$ and the equivariance of $\pi$, we see that $A$ is $\Gamma$-invariant. It is also easy to see that it is non-empty: as $\pi$ is continuous every fibre is compact, so we may choose in a Borel way some $\nu_x \in P(\pi^{-1}(x))$ for $x \in \K(\Gamma)$ and let $\nu = \int \nu_x d\mu \in A$ (c.f. \cite[\nopp28.8]{CDST}). As $\Gamma$ is amenable there is a fixed point in $A$, which is an invariant random expansion of $\mu$.
\end{proof}

\subsection{Generic equivariant expansions}\label{sec:generic-expansions}

If $\K$ is $G_\delta$ classes of structures of $\LL$-structures and $\Phi$ is an isomorphism-invariant property of $\LL$-structures, we say \tb{the generic element of $\K$ satisfies $\Phi$} if for some countably infinite set $X$, $\K_\Phi(X) = \{\bA \in \K(X) : \Phi(\bA)\}$ is comeagre in $\K(X)$. We note that this does not depend on the choice of $X$: If $f: X \to Y$ is a bijection, this induces a homeomorphism $\K(X) \to \K(Y)$ taking $\K_\Phi(X)$ to $\K_\Phi(Y)$.

In this section, we show that if $\K$ is a $G_\delta$ class of structures and the generic element of $\K$ has trivial algebraic closure, then the question of whether $\K$ admits $\Gamma$-equivariant expansions to $\K^*$ generically does not depend on the group $\Gamma$.

\begin{defn}\label{def:TAC}
	Let $\LL$ be a language and $\bA$ be a countable $\LL$-structure with universe $X$. We say $\bA$ has the \tb{weak duplication property (WDP)} if, for every $\bA_0 \in \age_X(\bA)$ and any finite $F \subseteq X$, there is an embedding of $\bA_0$ into $\bA$ whose image is disjoint from $F$.

	We say that $\bA$ is \tb{definable from equality} if for all $n$-ary $R \in \LL$ and $n$-tuples $\bar{x}, \bar{y}$ in $X$ with the same equality type, $R^\bA(\bar{x}) \iff R^\bA(\bar{y})$. (The \tb{equality type} of $\bar{x}$ is the set of pairs $(i, j)$ with $\bar{x}_i = \bar{x}_j$.)

	For $F \subseteq X$, let $\aut_F(\bA)$ denote the group of automorphisms of $\bA$ that fix $F$ pointwise, i.e., such that $f(x) = x$ for $x \in F$. We say $\bA$ has \tb{trivial algebraic closure (TAC)} if, for every finite $F \subseteq X$, the action of $\aut_F(\bA)$ on $X \setminus F$ has no finite orbits. See e.g. \cite[Section~4.2]{Hodges} or \cite[50]{CK} for alternative characterizations.
\end{defn}

\begin{eg}\label{eg:strong-homogeneity}
	Let $\K$ be the class described in one of \cref{eg:bijections,eg:ramsey,eg:linearizations,eg:trees,eg:zline}. The generic element of $\K$ has TAC and is not definable from equality. On the other hand, if $\K$ is the class of connected graphs of maximum degree $d$ (c.f. \cref{eg:colourings,eg:vizing}) then no element of $\K$ has TAC: If $G \in \K$ and $(u, v)$ is an edge of $G$, then the orbit of $v$ under the action of $\aut_{\{u\}}(G)$ has size at most $d$.
\end{eg}

\begin{thm}\label{thm:generic-expansions}
	Let $(\K, \K^*)$ be an expansion problem. Suppose that $\K$ is $G_\delta$ and the generic element of $\K$ has TAC and is not definable from equality. Then the following are equivalent:
	\begin{enumerate}
		\item For every countably infinite group $\Gamma$, $\K$ admits $\Gamma$-equivariant expansions to $\K^*$ generically.
		\item There exists a countably infinite group $\Gamma$ for which $\K$ admits $\Gamma$-equivariant expansions to $\K^*$ generically.
	\end{enumerate}
\end{thm}

In particular, this applies to \cref{eg:bijections,eg:ramsey,eg:linearizations,eg:trees,eg:zline} by \cref{eg:strong-homogeneity}.

\begin{lem}\label{lem:TAC-equivalent}
	Let $\bA$ be an $\LL$-structure on a countably infinite set $X$. Then $\bA$ has TAC if and only if, for every $\bA_0 \in \age_X(\bA)$ with universe $F$, every $F_0 \subseteq F$, every embedding $f: F_0 \to X$ of $\res{\bA_0}{F_0}$ into $\bA$ and every finite $G \subseteq X$ there is an embedding $g$ of $\bA_0$ into $\bA$ which extends $f$ and such that $g(F \setminus F_0) \cap G = \emptyset$.
\end{lem}

\begin{proof}
	$(\implies)$ Let $\bA_0, F, F_0, G$ be as in the lemma and suppose that $\bA$ has TAC. By Fra\"iss\'e's Theorem and \cite[Theorem~7.1.8]{Hodges} $\bA$ is homogeneous, so we may assume wlog that $f$ is the identity. By Neumann's Separation Lemma, there is some $g \in \aut_{F_0}(\bA)$ such that $g(F \setminus F_0) \cap G = \emptyset$, in which case we may take $\res{g}{F}$.

	$(\impliedby)$ Note first that $\bA$ is homogeneous, i.e., every isomorphism between finite substructures of $\bA$ extends to an automorphism of $\bA$. To see this, by \cite[Lemma~7.1.4]{Hodges} it suffices to show that if $\bA_0 \in \age_X(\bA)$ has universe $F$, $F_0 \subseteq F$ and $f: F_0 \to X$ is an embedding of $\res{\bA_0}{F_0}$ into $\bA$, then $f$ can be extended to an embedding of $\bA_0$ into $\bA$, which follows from our assumption (taking $G = \emptyset$).

	Let now $F$ be a finite set and $x \in X \setminus F$ in order to show that $x$ has infinite orbit under $\aut_F(\bA)$. Let $C \subseteq X$ be finite and let $\bA_0 = \res{\bA}{(F \cup \{x\})}$. By assumption, there is an embedding $f$ of $\bA_0$ into $\bA$ that is the identity on $F$ and such that $f(x) \notin C$. By homogeneity, this can be extended to an automorphism of $\bA$, so that in particular the orbit of $x$ is not contained in $C$. As $C, x$ were arbitrary, we conclude that $\bA$ has TAC.
\end{proof}

In particular, if $\bA$ has TAC then $\bA$ is homogeneous and has the WDP (to see the latter, take $F_0 = \emptyset$ and $G = F$ in the lemma).

\begin{rmk}
	If $\K$ is a class of structures and $X$ a countably infinite set, the set of structures in $\K(X)$ which have the WDP (resp. are definable from equality, have TAC) is $G_\delta$ (resp. closed, $G_\delta$) in $\K(X)$.
\end{rmk}

\begin{lem}\label{lem:wdp-free}
	Let $\K$ be a class of structures with the WDP that are not definable from equality, and $\Gamma$ be a countably infinite group. For any $\bA \in \K(\Gamma)$ there is some $\bB \in Fr(\K(\Gamma))$ isomorphic to $\bA$. If $\K$ is $G_\delta$, then $Fr(\K(\Gamma))$ is a dense $G_\delta$ set in $\K(\Gamma)$.
\end{lem}

\begin{proof}
	Let $\bA \in \K(\Gamma)$ and fix an enumeration $\{\gamma_n\}_{n \in \N}$ of $\Gamma$. Let $R \in \LL$ be a relation such that $R^\bA$ is not definable from equality. We will construct an increasing sequence of finite partial bijections $f_n: \Gamma \to \Gamma$ so that $\gamma_n$ is in the domain of $f_{2n}$ and in the range of $f_{2n+1}$, and moreover so that for all $n$ there is a tuple $\bar{x}$ such that $\bar{x}, \gamma_n\bar{x}$ are contained in the domain of $f_{2n}$ and $R^\bA(f_{2n}(\bar{x}) \iff \lnot R^\bA(f_{2n}(\gamma_n\bar{x}))$. Assuming this has been done, we let $f = \bigcup_n f_n$ and $\bB = f^{-1}(\bA)$. Then $\bB \cong \bA \in \K(\Gamma)$, and for all $\gamma \in \Gamma$ there is some $\bar{x}$ such that
	\[R^\bB(\bar{x}) \iff R^\bA(f(\bar{x})) \iff \lnot R^\bA(f(\gamma \bar{x})) \iff \lnot R^\bB(\gamma \bar{x}),\]
	so that $\gamma \bB \neq \bB$ for all $\gamma \in \Gamma$.

	We construct these maps as follows. Set $f_{-1} = \emptyset$ for convenience. Given $f_{2n}$, we let $f_{2n+1}$ be an arbitrary extension with $\gamma_n$ in its range. Suppose now we have constructed $f_{2n-1}: F \to G$. Because $\bA$ has the WDP and $R^\bA$ is not definable from equality, there are tuples $\bar{y}, \bar{z}$ with the same equality type so that $R^\bA(\bar{y}) \iff \lnot R^\bA(\bar{z})$ and the sets $\bar{y}, \bar{z}, G$ are pairwise disjoint. Since $\Gamma$ is infinite, we can find a tuple $\bar{x}$ with the same equality type as $\bar{y}$ so that $\bar{x}, \gamma_n \bar{x}, F$ are pairwise disjoint. We then define $f_{2n}$ to extend $f_{2n-1}$ by sending $\bar{x} \mapsto \bar{y}$, $\gamma_n \bar{x} \mapsto \bar{z}$ and, if $f_{2n}(\gamma_n)$ has not already been defined, setting it to any element of $\Gamma$ not already in the range of $f_{2n}$.

	If $\K$ is $G_\delta$, then $Fr(\K(\Gamma))$ is a $G_\delta$ set as the action of $\Gamma$ is continuous. Moreover, if $\bA_0 \in \age_\Gamma(\K)$ has universe $F$ and $\bA \in N(\bA_0)$, then the same construction starting instead with the $f_{-1}$ as the identity on $F$ gives $\bB \in N(\bA_0) \cap Fr(\K(\Gamma))$, so the free part is dense.
\end{proof}

\begin{lem}\label{lem:homogeneous-generic}
	Let $\K$ be a $G_\delta$ class of structures with TAC and $\Gamma$ be a countably infinite group. The set of $\bA \in \K(\Gamma)$ with the following extension property is a dense $G_\delta$ set in $\K(\Gamma)$:
	\begin{enumerate}
		\item[$(*)$] Let $\bA_0, \bA_1 \in \age_\Gamma(\K)$ be $\LL'$-structures with disjoint universes $F, G$ respectively. Let $F_0 \subseteq F$ and $f: (F \setminus F_0) \cup G \to \Gamma$ be an injection. If $\res{\bA_0}{F_0} \sqsubseteq \bA$ and $\bA_0, \bA_1$ embed into $\bA$, then there is some $\gamma \in \Gamma$ such that the map $F_0 \ni x \mapsto x$, $(F \setminus F_0) \cup G \ni x \mapsto \gamma f(x)$ embeds both $\bA_0, \bA_1$ into $\bA$.
	\end{enumerate}
\end{lem}

\begin{proof}
	Fix $\bA_0, \bA_1, F, G, F_0, f$ as in $(*)$ and let $a: F \to \Gamma, b: G \to \Gamma$ be injections. Let $U$ be the set of all $\bA \in \K(\Gamma)$ such that, if $\res{\bA_0}{F_0} \sqsubseteq \bA$ and $a, b$ are embeddings of $\bA_0, \bA_1$ into $\bA$, then the conclusion of $(*)$ holds for $\bA$. We will show that $U$ is open and dense. As there are only countably many choices for $\bA_0, \bA_1, F, G, F_0, f, a, b$, the intersection of all such $U$ is a dense $G_\delta$ set whose elements satisfy $(*)$.

	It is clear that $U$ is open. To see that it is dense, fix $\bB_0 \in \age_\Gamma(\K)$ and let $\bA \in N(\bB_0)$. If $\res{\bA_0}{F_0} \nsubseteq \bA$ or some $a, b$ is not an embedding of $\bA_0, \bA_1$ into $\bA$ then $\bA \in U$. Otherwise, let $\bB_0$ have universe $H$, and assume wlog that $F_0 \subseteq H$. Because $\bA$ has TAC and by \cref{lem:TAC-equivalent}, there is an embedding $g_0: F \to \Gamma$ of $\bA_0$ into $\bA$ extending the identity on $F_0$ so that $g_0(F \setminus F_0) \cap H = \emptyset$. By the WDP, there is an embedding $g_1: \bA_1 \to \bA$ whose image is disjoint from $H \cup g_0(F)$. Fix $\gamma$ so that $H \cap \gamma f((F \setminus F_0) \cup G) = \emptyset$, and let $h: \Gamma \to \Gamma$ be a bijection so that the following hold: $h$ is the identity on $H$, $h(\gamma f(x)) = g_0(x)$ for $x \in F \setminus F_0$ and $h(\gamma f(x)) = g_1(x)$ for $x \in G$. Let $\bB = h^{-1}(\bA)$. Then $\bB \in N(\bB_0)$, $\bB \cong \bA$ and $\bB$ satisfies $(*)$ as witnessed by $\gamma$, so $\bB \in U$.
\end{proof}

\begin{proof}[Proof of \cref{thm:generic-expansions}]
	Clearly $(1) \implies (2)$.

	$(2) \implies (1)$: As the class $\K'$ of elements of $\K$ with TAC that are not definable from equality is a comeagre $G_\delta$ set in $\K$, we may assume wlog that $\K = \K'$.

	Let $\Gamma, \Lambda$ be countably infinite groups and $Z \subseteq \K(\Lambda)$ be a Borel comeagre $\Lambda$-invariant set that admits a $\Lambda$-equivariant expansion to $\K^*$.

	Let $f_n \in S_\Lambda$ be a dense sequence of bijections, i.e., a sequence such that for every finite partial bijection $\Lambda \to \Lambda$ there is some $f_n$ extending it. Since $S_\Lambda$ acts on $\K(\Lambda)$ by homeomorphisms, we can assume wlog that $Z$ is $G_\delta$ and that $f_n(\bA) \in Z$ for all $\bA \in Z, n \in \N$. By \cref{lem:wdp-free}, we may also assume that $Z \subseteq Fr(\K(\Lambda))$.

	Let $X = Fr(\K(\Gamma))$, which by \cref{lem:wdp-free} is a dense $G_\delta$ set in $\K(\Gamma)$, and let $E = E^X_\Gamma$. Let $\bbA$ be the canonical $\K$-structuring of $E$. We will show that there is a Borel expansion of $\bbA$ restricted to a comeagre $\Gamma$-invariant set $Y \subseteq X$, and hence by \cref{prop:weak-correspondence-expansions}(2) there is a Borel $\Gamma$-equivariant expansion map $Y \to \K^*(\Gamma)$. As $\Gamma$ was arbitrary, this proves $(1)$.

	Our proof strategy is as follows: We find a Borel $\Gamma$-invariant comeagre set $Y \subseteq X$ and a free $\Lambda$-action on $Y$ so that $E^Y_\Lambda = E^Y_\Gamma$. By \cref{prop:correspondence-sructures-maps} applied to $\res{\bbA}{Y}$, this gives a $\Lambda$-equivariant Borel map $F: Y \to \K(\Lambda)$. We will ensure that the image of $F$ is contained in $Z$. By \cref{prop:correspondence-expansion-maps}, this implies the existence of a Borel expansion of $\res{\bbA}{Y}$, completing the proof.

	Below, we let $\forall^*$ denote the category quantifier: If $W$ is a topological space and $A \subseteq W$ is Baire-measurable, $\forall^* w A(x)$ means that $A$ is comeagre (see \cite[8.J]{CDST}).

	Let $\mathcal{G}$ denote the \tb{intersection graph} of $E$. That is, the vertices of $\mathcal{G}$ are finite subsets of $X$ which are contained in a single $E$-class, and
	\[a \mathcal{G} b \iff a \neq b \And (a \cap b \neq \emptyset).\]
	By the proof of \cite[Lemma~7.3]{KM}, we may fix a countable Borel colouring $c$ of $\mathcal{G}$.

	Let $(R_n, \gamma_n, \bar{\delta}_n)$ be a sequence of triples so that (1) the tuples in $\{\bar{\delta}_n, \gamma_n \bar{\delta}_n : n \in \N\}$ are pairwise disjoint, (2) for every $n$, $R_n \in \LL$, and if $R_n$ has arity $k$ then $\bar{\delta} \in \Gamma^k$, and (3) for every $R \in \LL$ of arity $k$ and every equality type of tuples of length $k$, there are infinitely many $n$ with $R_n = R$ such that $\bar{\delta}_n$ has this equality type. Let
	\[O_n = \{\bA \in X : R^\bA_n(\bar{\delta}_n) \And \lnot R^\bA_n(\gamma_n \bar{\delta}_n)\}\]
	and $B_{n+1} = O_n \setminus \bigcup_{i < n} O_i$, $n \in \N$. We also set $B_0 = X \setminus \bigcup_n O_n$.

	\begin{claim}\label{claim:generic-intersection}
		Suppose $\bA \in \K(\Gamma)$ satisfies $(*)$ from \cref{lem:homogeneous-generic}. Then $|\Gamma \cdot \bA \cap B_n| = \infty$ for infinitely many $n$.
	\end{claim}

	\begin{claimproof}
		Let $N \in \N$ be arbitrary. We may find some $n \geq N$ so that $R_n^\bA$ is not definable from equality, and there are tuples $\bar{x}, \bar{y}$ with the same equality type as $\bar{\delta}_n$ so that $R^\bA_n(\bar{x}) \And \lnot R^\bA_n(\bar{y})$. By the WDP, we may assume $\bar{x}, \bar{y}$ are disjoint. For $i < n$, find $\bar{x}_i$ with the same equality type as $\bar{\delta}_i$ which are disjoint from each other and from $\bar{x}, \bar{y}$. By $(*)$, there is some $\gamma$ so that $R^\bA_n(\gamma \bar{\delta}_n)$, $\lnot R^\bA_n(\gamma \gamma_n \bar{\delta}_n)$ and for $i < n$, if $R^\bA_i(\bar{x}_i)$ then $R^\bA_i(\gamma \gamma_i \bar{\delta}_i)$ and otherwise $\lnot R^\bA(\gamma \bar{\delta}_i)$. It follows that $\gamma^{-1} \bA \in B_n$.
	\end{claimproof}

	For $\alpha \in \N^\N$, we define an equivalence relation $E^\alpha$ on $X$ as follows: We set $E^\alpha_0$ to be equality. Given $E^\alpha_n$, we set $x E^\alpha_{n+1} y$ if either $x E^\alpha_n y$ or $x \cancel{E^\alpha_n} y$, $c([x]_{E^\alpha_n} \cup [y]_{E^\alpha_n}) = \alpha(n)$, and $x, y \in \bigcup_{i < n} B_i$. We then set $E^\alpha = \bigcup_n E^\alpha_n$.

	Note that if $[x]_{E^\alpha_n}$ is not a singleton for some $\alpha, n$, then $x \in \bigcup_{i < n} B_i$. We note also that we may analogously construct $E^s_i$ for $s \in \N^{< \N}, i \leq |s|$, and that $E^\alpha_i = E^s_i$ for all such $s, i$ and $\alpha \supseteq s$. We set $E^s = E^s_{|s|}$.

	\begin{claim}\label{claim:generic-equal-class}
		$\forall x \in X \forall^* \alpha ([x]_E = [x]_{E^\alpha})$.
	\end{claim}

	\begin{claimproof}
		Fix $x \in X$ and let $y E x$. We show that the set of $\alpha$ with $y \in [x]_{E^\alpha}$ is open and dense. It is clearly open, as if $y E^\alpha x$ then $y E^\alpha_n x$ for some $n$. To see it is dense, fix $s \in \N^n$. We may assume wlog that $x, y \in \bigcup_{i < n} B_i$. But then any $\alpha \supseteq s$ with $\alpha(n) = c([x]_{E^s} \cup [y]_{E^s})$ satisfies $y E^\alpha x$. As there are only countably many $y E x$, the set of all $\alpha$ for which $[x]_E = [x]_{E^\alpha}$ is dense $G_\delta$.
	\end{claimproof}

	Let now $\LL_0 = (f_\lambda)_{\lambda \in \Lambda}$, where each $f_\lambda$ is a binary relation, and let $\bm{\Lambda} = (\Lambda, f^{\bm{\Lambda}}_\lambda)_{\lambda \in \Lambda}$ be the $\LL_0$-structure where $f^{\bm{\Lambda}}_\lambda(\delta) = \lambda \delta$ is interpreted as (the graph of) multiplication on the left by $\lambda$.

	Let $\alpha, \beta \in \N^\N$. We define an $\LL_0$-structure $\mathbbl{\Lambda}^{\alpha, \beta}$ on $X$ as follows. We will define an increasing sequence of structures $\mathbbl{\Lambda}^{\alpha, \beta}_n$ on $E^\alpha_n$ and then take $\mathbbl{\Lambda}^{\alpha, \beta} = \bigcup_n \mathbbl{\Lambda}^{\alpha, \beta}_n$. We will ensure at every stage $n$ of this process that, if $C$ is an $E^\alpha_n$-class, then $\res{\mathbbl{\Lambda}^{\alpha, \beta}_n}{C}$ will be isomorphic to a substructure of $\bm{\Lambda}$.

	Fix a Borel linear order $<$ on $X$ and an enumeration of $\Lambda$. We define $\mathbbl{\Lambda}^{\alpha, \beta}_0$ by setting $f_{id}^{\mathbbl{\Lambda}^{\alpha, \beta}_0}(x) = x$ for all $x \in X$, and leaving the other relations undefined. Suppose now that we have defined $\mathbbl{\Lambda}^{\alpha, \beta}_n$, in order to define $\mathbbl{\Lambda}^{\alpha, \beta}_{n+1}$. Let $C$ be an $E^\alpha_{n+1}$-class. If $C$ is an $E^\alpha_n$-class, then we define $\mathbbl{\Lambda}^{\alpha, \beta}_{n+1}$ to be equal to $\mathbbl{\Lambda}^{\alpha, \beta}_n$ on $C$. Otherwise, $C$ is the union of two $E^\alpha_n$-classes $C_0, C_1$. Order them so that the $<$-least element of $C_0$ is $<$-below the $<$-least element of $C_1$. For each $C_i$, as $\res{\mathbbl{\Lambda}^{\alpha, \beta}_n}{C_i}$ is isomorphic to a substructure of $\bm{\Lambda}$, there is a unique embedding $f_i: C_i \to F_i \subseteq \Lambda$ taking the $<$-least element of $C_i$ to the identity in $\Lambda$. We then take $\lambda$ to be the $\beta(n)$-th element of $\Lambda$, in our fixed enumeration, satisfying $F_0 \cap F_1 \cdot \lambda = \emptyset$. We define now $\res{\mathbbl{\Lambda}^{\alpha, \beta}_{n+1}}{C}$ to be the pullback of $\res{\bm{\Lambda}}{(F_0 \cup F_1 \cdot \lambda)}$ via the injection $(f_0 \cup f_1 \cdot \lambda): (C_0 \cup C_1) \to (F_0 \cup F_1 \cdot \lambda)$, where $f_1 \cdot \lambda$ denotes the map $x \mapsto f_1(x) \cdot \lambda$. (Note that we are multiplying $F_1, f_1$ by $\lambda$ on the right. This is because $\bm{\Lambda}$ is defined in terms of multiplication on the left, and this commutes with multiplication on the right.)

	As with the $E^\alpha$, we may define $\mathbbl{\Lambda}^{s, t}_k$ for $s, t \in \N^n$ and $k \leq n$, and we let $\mathbbl{\Lambda}^{s, t} = \mathbbl{\Lambda}^{s, t}_n$. Note that $\mathbbl{\Lambda}^{\alpha, \beta}_n = \mathbbl{\Lambda}^{\res{\alpha}{n}, \res{\beta}{n}}$ for all $\alpha, \beta$.

	\begin{claim}\label{claim:generic-equal-structure}
		$\forall^* x \in X \forall^* \alpha, \beta (\res{\mathbbl{\Lambda}^{\alpha, \beta}}{[x]_{E^\alpha}} \cong \bm{\Lambda})$.
	\end{claim}

	\begin{claimproof}
		Fix $x \in X$ satisfying $(*)$ of \cref{lem:homogeneous-generic}. Note that by construction, $\res{\mathbbl{\Lambda}^{\alpha, \beta}}{[x]_{E^\alpha}}$ embeds into $\bm{\Lambda}$ for all $\alpha, \beta$, and hence there is a unique embedding $f^{\alpha, \beta}: \res{\mathbbl{\Lambda}^{\alpha, \beta}}{[x]_{E^\alpha}} \to \bm{\Lambda}$ taking $x$ to the identity. Thus it suffices to show that the set of all $\alpha, \beta$ for which $f^{\alpha, \beta}$ is surjective is a dense $G_\delta$ set. To see this, fix $\lambda \in \Lambda$. We will show that the set of all $\alpha, \beta$ so that $\lambda \in f^{\alpha, \beta}([x]_{E^\alpha})$ is a dense open set. As there are only countably many such $\lambda$, this completes the proof.

		Note that we may define similarly $f^{s, t}: [x]_{E^s} \to \Lambda$ for $s, t \in \N^n$. Then $f^{\alpha, \beta} = \bigcup_n f^{\res{\alpha}{n}, \res{\beta}{n}}$, so it is clear that the set of $\alpha, \beta$ with $\lambda$ in its image is an open set. To see that it is dense, fix $s, t \in \N^n$ and consider $f^{s, t}$. If $\lambda$ is in the image of $f^{s, t}$ then any $\alpha, \beta$ extending $s, t$ satisfies that $\lambda$ is in the image of $f^{\alpha, \beta}$. So suppose otherwise. By \cref{claim:generic-intersection}, there is some $y E x$ so that $[y]_{E^s}$ is a singleton. It is easy to see that if $u$ is an extension of $s$ whose new values are all $c([x]_{E^s} \cup \{y\})$, then for sufficiently long $u$ we have $y E^u x$. Pick such a $u$ of minimal length, so that at stage $|u|$ of the construction of $E^u$ we merge $[x]^{E^s}$ with $\{y\}$. Let $C_0, C_1$ denote these two sets, ordered as in the construction of $\mathbbl{\Lambda}^{\alpha, \beta}$, and let $f_i: C_i \to F_i \subseteq \Lambda$ be the corresponding embeddings.

		If $v$ is any extension of $t$ of length $|u|$, $m = v(|u|-1)$ and $\lambda_m$ is the $m$-th element of $\Lambda$ such that $F_0 \cap F_1 \cdot \lambda_m = \emptyset$, then $\res{\mathbbl{\Lambda}^{u, v}}{C} \cong \res{\bm{\Lambda}}{(F_0 \cup F_1 \cdot \lambda_m)}$ via the map $f = f_0 \cup f_1 \cdot \lambda_m$. In particular, if $x \in C_0$ and $f_0(x) = \delta$, then $f^{u, v} = f \cdot \delta^{-1}$. On the other hand, if $x \in C_1$ and $f_1(x) = \delta$, then $f^{u, v} = f \cdot \lambda_m^{-1} \delta^{-1}$. We will show that we can choose $m$ so that $\lambda = f^{u, v}(y)$, and hence such that $\lambda$ is in the image of $f^{\alpha, \beta}$ for all $\alpha, \beta$ extending $u, v$. As $s, t$ were arbitrary, the set of all such $\alpha, \beta$ is dense and we are done.

		Consider now two cases. If $x \in C_0$ and $f_0(x) = \delta$, then by assumption $\lambda \delta \notin F_0$. Since $C_1$ is a singleton, $F_1$ contains only the identity. Pick $m$ so that $\lambda_m = \lambda \delta$. For such an $m$ we have $f^{u, v}(y) = f_1(y) \lambda_m \delta^{-1} = \lambda \delta \delta^{-1} = \lambda$. On the other hand, if $x \in C_1$ and $f_1(x) = \delta$, then by assumption $\lambda \delta \notin F_1$. In this case again $F_0$ contains only the identity. Pick $m$ with $\lambda_m = \delta^{-1} \lambda^{-1}$. For such an $m$ we have $f^{u, v}(y) = f_0(y) \lambda_m^{-1} \delta^{-1} = \lambda \delta \delta^{-1} = \lambda$.
	\end{claimproof}

	For all $\alpha, \beta$, let $X^{\alpha, \beta} \subseteq X$ denote the set of all $x$ for which $[x]_{E^\alpha} = [x]_E$ and $\res{\mathbbl{\Lambda}^{\alpha, \beta}}{[x]_E} \cong \bm{\Lambda}$. There is a free Borel action of $\Lambda$ on this set, namely the action where $\lambda \cdot x = f^{\mathbbl{\Lambda}^{\alpha, \beta}}_\lambda(x)$. By \cref{prop:correspondence-sructures-maps}, the structure $\bbA$ on $\res{E}{X^{\alpha, \beta}} = E^{X^{\alpha, \beta}}_\Lambda$ gives rise to a $\Lambda$-equivariant Borel map $F^{\alpha, \beta}: X^{\alpha, \beta} \to \K(\Lambda)$.

	\begin{claim}\label{claim:generic-in-Z}
		$\forall^* x \in X \forall^* \alpha, \beta (x \in X^{\alpha, \beta} \implies F^{\alpha, \beta}(x) \in Z)$.
	\end{claim}

	\begin{claimproof}
		Fix any $\bA \in X$ that is isomorphic to an element of $Z$ and satisfies $(*)$ of \cref{lem:homogeneous-generic}. Note that the set of all such $\bA$ is comeagre (for the first condition, note that any bijection $\Lambda \to \Gamma$ gives a homeomorphism $\K(\Lambda) \to \K(\Gamma)$ and consider the image of $Z$). If $\bA \in X^{\alpha, \beta}$, let $\bB^{\alpha, \beta} = F^{\alpha, \beta}(\bA)$. We will show that the set of all $\alpha, \beta$ for which $\bA \in X^{\alpha, \beta} \implies \bB^{\alpha, \beta} \in Z$ is a dense $G_\delta$ set.

		Fix $s, t \in \N^n$ and let $C^{s, t} = [\bA]_{E^s}$. Let $f^{s, t}: C^{s, t} \to I^{s, t} \subseteq \Lambda$ be the unique embedding of $\res{\mathbbl{\Lambda}^{s, t}}{C^{s, t}}$ into $\bm{\Lambda}$ which takes $\bA$ to the identity (note that this $f^{s, t}$ is the same as the one described in the proof of the previous claim). Let also $D^{s, t} = \{\gamma : \gamma^{-1} \bA \in C^{s, t}\}$, let $g^{s, t}: D^{s, t} \to C^{s, t}$ be the map $\gamma \mapsto \gamma^{-1} \bA$ and let $h^{s, t} = f^{s, t} \circ g^{s, t}$. Let $\bB^{s, t} = h^{s, t}(\res{\bA}{D^{s, t}})$. It is easy to see, by \cref{prop:correspondence-sructures-maps}, that $\bB^{s, t} \sqsubseteq \bB^{\alpha, \beta}$ whenever $\alpha \supseteq s, \beta \supseteq t$ and $\bA \in X^{\alpha, \beta}$.

		Let now $U_n$ be a sequence of dense open sets in $\K(\Lambda)$ so that $Z = \bigcap_n U_n$. We will show that for all $N$, the set of $\alpha, \beta$ so that $\bA \in X^{\alpha, \beta} \implies \bB^{\alpha, \beta} \in U_N$ is dense and open. Since $U_N$ is open,
		\[\bA \in X^{\alpha, \beta} \implies [\bB^{\alpha, \beta} \in U_N \iff \exists \LL' \subseteq \LL \exists n (N(\res{\bB^{\res{\alpha}{n}, \res{\beta}{n}}}{\LL'}) \subseteq U_N)].\]
		Thus the set of all such $\alpha, \beta$ is exactly the set of $\alpha, \beta$ satisfying
		\[\exists \LL' \subseteq \LL \exists n (N(\res{\bB^{\res{\alpha}{n}, \res{\beta}{n}}}{\LL'}) \subseteq U_N),\]
		which is clearly open, so it remains to show that it is dense. That is, we need to show that for all $s, t$, there are $u, v$ extending $s, t$ so that $N(\res{\bB^{u, v}}{\LL'}) \subseteq U_N$ for some finite $\LL' \subseteq \LL$.

		Fix $s, t \in \N^n$. By assumption, there is some $\bA' \in Z$ that is isomorphic to $\bA$. As $Z$ is closed under the functions $f_n$ described at the start of the proof, we may assume that $\bB^{s, t} \sqsubseteq \bA'$. Since $\bA' \in Z \subseteq U_N$, there is some $\LL'$-structure $\bB_0 \in \age_\Lambda(\K)$ so that $\bA' \in N(\bB_0) \subseteq U_N \cap N(\res{\bB^{s, t}}{\LL'})$. Let $F$ be the universe of $\bB_0$, so that $\bB_0 = \res{(\res{\bA'}{\LL'})}{F}$ and wlog $F \supseteq I^{s, t}$. We will show that there are $u, v$ extending $s, t$ so that $\res{\bB^{u, v}}{\LL'} = \bB_0$, which would complete the proof.

		We will show how to do this assuming that $F = I^{s, t} \sqcup \{\lambda\}$. By repeating this argument recursively we can handle all finite $F$.

		By the proof of \cref{claim:generic-intersection}, there is some $m > n$, a finite structure $\bA_0 \in \age_\Gamma(\bA)$ and an injection $f$ from the universe of $\bA_0$ to $\Gamma \setminus \{1_\Gamma\}$ so that if $x \mapsto \gamma f(x)$ is an embedding of $\bA_0$ into $\bA$, then $\gamma^{-1} \bA \in B_m$. It is clear from the proof that we can also assume that the universe of $\bA_0$ is disjoint from $D^{s, t}$. Now $\bB_0$ embeds into $\bA$, as it embeds into $\bA' \cong \bA$, so by $(*)$ there is some $\gamma$ so that (a) the map $x \mapsto \gamma f(x)$ is an embedding of $\bA_0$ into $\bA$ and (b) the map $(h^{s, t})^{-1} \cup \{(\lambda, \gamma)\}$ is an embedding $F \to \Gamma$ of $\bB_0$ into $\bA$.

		By (a) and our choice of $\bA_0$, $\gamma^{-1} \bA \in B_m$ so the $E^s$-class of $\gamma^{-1} \bA$ is a singleton. Extend $s$ to a sequence $u$ of length $m+1$ by setting the new values to be $c([\bA]_{E^s} \cup \{\gamma^{-1} \bA\})$. Thus $[\bA]_{E^u_m} = [\bA]_{E^s}$ and $[\bA]_{E^u} = [\bA]_{E^s} \cup \{\gamma^{-1} \bA\}$. By the proof of the previous claim, there is an extension $v$ of $t$ of length $m+1$ so that $f^{u, v}(\gamma^{-1} \bA) = \lambda$.

		We claim now that $\res{\bB^{u, v}}{\LL'} = \bB_0$. To see this, note that $D^{u, v} = D^{s, t} \cup \{\gamma\}$, $I^{u, v} = I^{s, t} \cup \{\lambda\} = F$ and $h^{u, v} = (h^{s, t} \cup \{(\gamma, \lambda)\})$. Now $\bB^{u, v} = h^{u, v}(\res{\bA}{D^{u, v}})$, so this follows immediately from (b).
	\end{claimproof}

	By the Kuratowski--Ulam Theorem \cite[\nopp8.41]{CDST} and the claims above, we may fix some $\alpha, \beta$ so that $X^{\alpha, \beta}$ is comeagre and $x \in X^{\alpha, \beta} \implies F^{\alpha, \beta}(x) \in Z$ for the generic $x \in X$. In particular, there is a comeagre $\Gamma$-invariant Borel set $Y \subseteq X^{\alpha, \beta}$ such that $F^{\alpha, \beta}(Y) \subseteq Z$, which proves $(1)$ by the remarks at the start of the proof.
\end{proof}

\subsection{Enforcing smoothness}\label{sec:enfocing-smoothness}

Let $(\K, \K^*)$ be an expansion problem. We are broadly interested in relating the class of Borel expandable CBER with natural classes of CBER such as being smooth or compressible. In this section, we give some sufficient conditions for an expansion problem $(\K, \K^*)$ to be Borel expandable for exactly the class of smooth CBER.

\begin{prop}\label{prop:smooth-expandable}
	Let $(\K, \K^*)$ be an expansion problem. If there is a Borel expansion map $f: \K(\N) \to \K^*(\N)$, then every smooth aperiodic CBER is Borel expandable for $(\K, \K^*)$.
\end{prop}

Note that we do not require $f$ to satisfy any additional properties (such as equivariance).

\begin{proof}
	Let $E$ be a smooth aperiodic CBER on $X$. Since $E$ is smooth, we can identify every $E$-class with $\N$ in a Borel way, i.e., there is a Borel enumeration $g: X \to X^\N$ so that if $x E y$ then $g(x) = g(y)$ (for example take any Borel enumeration $h$ of $E$ and a selector $s$ for $E$ and let $g = h \circ s$). If $F: X \to \K(\N)$ is Borel and $E$-invariant, then composing this with $f$ gives a Borel $E$-invariant map $F^* = f \circ F: X \to \K^*(\N)$ so that $F^*(x)$ is an expansion of $F(x)$ for all $x$. By the correspondence described in \cref{sec:descriptions}, it follows that $E$ is Borel expandable for $(\K, \K^*)$.
\end{proof}

\begin{rmk}\label{rmk:smooth-expandable}
	In many cases of interest (including all of the examples in \cref{sec:expansions-def}), a Borel expansion map $\K(\N) \to \K^*(\N)$ can easily be shown to exist, for example by recursively constructing an expansion for a given $\bA \in \K(\N)$, or via an application of the Compactness Theorem (see e.g. \cite[\nopp28.8]{CDST}). Note that such constructions depend crucially on the given enumeration of the universe of $\bA$, and hence are not in general equivariant.
\end{rmk}

\begin{defn}
	Let $(\K, \K^*)$ be an expansion problem and $\mathcal{E}$ be a class of aperiodic CBER. We say $(\K, \K^*)$ \tb{enforces $\mathcal{E}$} if an aperiodic CBER $E$ belongs to $\mathcal{E}$ whenever $E$ is $\K$-structurable and Borel expandable for $(\K, \K^*)$.

	When $\mathcal{E}$ is the class of aperiodic smooth CBER we say that such $(\K, \K^*)$ \tb{enforces smoothness}.
\end{defn}

\begin{prop}\label{prop:enforce-smoothness-generic-exp}
	Let $(\K, \K^*)$ be an expansion problem. If some hyperfinite, compressible, aperiodic CBER is not Borel expandable for $(\K, \K^*)$, then $(\K, \K^*)$ enforces smoothness. In particular, this holds if some aperiodic CBER is not generically expandable for $(\K, \K^*)$.
\end{prop}

\begin{proof}
	Let $E$ be a hyperfinite, compressible, aperiodic CBER and let $F$ be any non-smooth aperiodic CBER. By the Glimm-Effros Dichotomy and compressibility we have $E \sqsubseteq^i_B F$ (c.f. \cite{DJK}). It is now easy to see that if $F$ is Borel expandable then so is $E$, hence if $E$ is not Borel expandable then neither is $F$. The second part follows from the first by \cite[Theorem~12.1, Corollary~13.3]{KM}.
\end{proof}

\begin{cor}\label{cor:enforcing-smoothness-generic}
	Let $(\K, \K^*)$ be an expansion problem. Suppose $\K$ is $G_\delta$ and there is a countably infinite group $\Gamma$ with $Fr(\K(\Gamma)) \neq \emptyset$ so that there is no Borel equivariant expansion map $X \to \K^*(\Gamma)$ for any comeagre invariant Borel set $X \subseteq Fr(\K(\Gamma))$. Then $(\K, \K^*)$ enforces smoothness.
\end{cor}

\begin{proof}
	Let $\bbA$ be the canonical Borel $\K$-structuring of $E = E^{Fr(\K(\Gamma))}_\Gamma$. By our assumption on $\Gamma$ and \cref{prop:weak-correspondence-expansions}(2), $\bbA$ does not admit a Borel expansion when restricted to any comeagre $E$-invariant Borel set, and in particular $E$ is not generically expandable. The conclusion follows by \cref{prop:enforce-smoothness-generic-exp}.
\end{proof}

\begin{rmk}
	By \cref{thm:generic-expansions}, if the generic element of $\K$ has TAC and not definable from equality then the hypotheses of \cref{cor:enforcing-smoothness-generic} hold for some group $\Gamma$ iff they hold for all groups $\Gamma$.
\end{rmk}

We note the following weak converse:

\begin{prop}
	Let $(\K, \K^*)$ be an expansion problem. If some non-smooth CBER admits a Borel $\K$-structuring and $\K$ admits a Borel $\Gamma$-equivariant expansion to $\K^*$ for some countably infinite group $\Gamma$ then $(\K, \K^*)$ does not enforce smoothness.
\end{prop}

\begin{proof}
	Let $E$ be a hyperfinite compressible CBER. If some non-smooth CBER admits a Borel $\K$-structuring, then so does $E$ (as $E$ invariantly embeds into any non-smooth CBER). Now consider the orbit equivalence relation $E$ of the shift of $\Gamma$ on $2^\Gamma$. This action is generically ergodic, hence by \cite[Theorem~12.1, Theorem~13.3]{KM} it is hyperfinite, compressible and non-smooth on an invariant dense $G_\delta$ set $Y \subseteq Fr(2^\Gamma)$. Thus $E^Y_\Gamma$ admits a Borel $\K$-structuring, and by \cref{prop:weak-correspondence-expansions}(1) it is Borel expandable, so $(\K, \K^*)$ does not enforce smoothness.
\end{proof}

\section{Examples}\label{sec:examples}

In this section, we will consider in detail definable expansion problems for \cref{eg:bijections,eg:colourings,eg:linearizations,eg:ramsey,eg:trees,eg:vizing,eg:zline,eg:matchings}. We summarize what is known for these problems in \cref{fig:table}.

\begin{table}[H]
	\centering
	\begin{threeparttable}
		\caption*{Expansions on CBER}
		\begin{tabular}{@{}llll@{}}
			\toprule
			& \begin{tabular}[c]{@{}c@{}}Borel\\ expandable\end{tabular} & \begin{tabular}[c]{@{}c@{}}Generically\\ expandable\end{tabular} & \begin{tabular}[c]{@{}c@{}}Expandable\\ a.e.\end{tabular} \\
			\midrule
			Bijections & Smooth\tnote{1} & All & Classified \\
			\rowcolor[HTML]{EFEFEF}
			Ramsey's Theorem & Smooth \cite{GX} & ? (CE) & ? (CE) \\
			Linearizations & Smooth & ? (CE) & ? (CE) \\
			\rowcolor[HTML]{EFEFEF}
			Vertex colouring & All\tnote{1} & All\tnote{1} & All\tnote{1} \\
			Spanning trees & ?\tnote{3} & All & ?\tnote{3} \\
			\rowcolor[HTML]{EFEFEF}
			$\Z$-lines & Smooth & ? (CE) & Classified \\
			Vizing's Theorem & Smooth\tnote{5} & ? (PP\tnote{6} ) & All \cite{GP} \\
			\rowcolor[HTML]{EFEFEF}
			Matchings & Smooth \cite{CJMSTD} & ? (CE\tnote{7}, PP\tnote{8} ) & ? (CE\tnote{7}, PP\tnote{8} ) \\
			\bottomrule
		\end{tabular}
	\end{threeparttable}

	\bigskip

	\begin{threeparttable}
		\caption*{Equivariant expansions on groups}
		\begin{tabular}{@{}lllll@{}}
			\toprule
			& \begin{tabular}[c]{@{}c@{}}Borel\\expansions\end{tabular} & \begin{tabular}[c]{@{}c@{}}Generic\\expansions\end{tabular} & \begin{tabular}[c]{@{}c@{}}Expansions\\a.e.\end{tabular} & \begin{tabular}[c]{@{}c@{}}Random\\expansions\end{tabular} \\
			\midrule
			Bijections & None\tnote{1} & All & Classified & Classified \\
			\rowcolor[HTML]{EFEFEF}
			Ramsey's Theorem & None & None & ? (CE${}_\Gamma$) & ? (CE${}_\Gamma$) \\
			Linearizations & None & None & ? (CE${}_\Gamma$) & Amen. \cite{A} \\
			\rowcolor[HTML]{EFEFEF}
			Vertex colouring & All\tnote{1,2} & All\tnote{1} & Classified\tnote{1} & All\tnote{1} \\
			Spanning trees & ?\tnote{4} & All & ?\tnote{4} & ?\tnote{4} \\
			\rowcolor[HTML]{EFEFEF}
			$\Z$-lines & None & None & Classified & Classified \\
			Vizing's Theorem & None & ? (PP\tnote{6} ) & All\tnote{2}\space\space \cite{GP} & All \cite{GP} \\
			\rowcolor[HTML]{EFEFEF}
			Matchings & None \cite{CJMSTD} & All\tnote{9} & ? (CE${}_\Gamma$\tnote{7}, PP\tnote{8} ) & ? (CE${}_\Gamma$\tnote{7}, PP\tnote{8} ) \\
			\bottomrule
		\end{tabular}

		\smallskip

		\begin{tablenotes}
			\footnotesize
			\item [] \tb{Classified}: In the sense of \cref{sec:canonical-random-expansions,prop:canonical-random-expansions}.
			\item [] \tb{CE}: There are counterexamples coming from free continuous actions of $\Gamma$, for all $\Gamma$.
			\item [] \tb{CE${}_\Gamma$} There are counterexamples for all countably infinite groups $\Gamma$.
			\item [] \tb{PP}: There are partial positive results (see the corresponding section for details and references).
			\item [1] Essentially \cite{KST}.
			\item [2] On the free part $Fr(\K(\Gamma))$.
			\item [3] This lies somewhere between hyperfinite and treeable.
			\item [4] This lies somewhere between amenable and treeable.
			\item [5] Smooth for $d \geq 3$ \cite{CJMSTD}, All for $d = 2$ \cite{KST}.
			\item [6] Subexponential growth \cite{BD} and bipartite \cite{BW}.
			\item [7] See \cite{Lac,Conley-Kechris,Kun,BKS}.
			\item [8] See \cite{LN,MU,CM,BKS,BCW,BPZ,KL}.
			\item [9] All for graphs of bounded degree $d > 2$ (None for $d = 2$).
		\end{tablenotes}
	\end{threeparttable}
	\caption{A summary of known results for \cref{eg:bijections,eg:colourings,eg:linearizations,eg:matchings,eg:ramsey,eg:trees,eg:vizing,eg:zline}}
	\label{fig:table}
\end{table}

\subsection{Bijections}

Fix $\K, \K^*$ as in \cref{eg:bijections}, that is,
\begin{align*}
	\K &= \{(X, R, S) \mid R, S \subseteq X \And X, R, S ~ \text{are all countably infinite}\},\\
	\K^* &= \{(X, R, S, T) \mid (X, R, S) \in \K \And T ~ \text{is the graph of a bijection} ~ R \to S\}.
\end{align*}

\begin{prop}[Essentially {\cite{KST}}]
	$(\K, \K^*)$ enforces smoothness, and $(E, \mu)$ is not a.e. expandable for any CBER $E$ and any $E$-invariant probability Borel measure $\mu$.
\end{prop}

\begin{proof}
	Let $E$ be a non-smooth aperiodic CBER. By the Glimm--Effros Dichotomy, $E_0 \sqsubseteq_B E$. By \cite[Theorem~1.1]{KST} $E_0$ is not Borel expandable, and it follows that $E$ is not Borel expandable either. If $\mu$ is an $E$-invariant probability Borel measure, then $(E, \mu)$ is not a.e. expandable by the same argument as in \cite[Section~1]{KST} for the shift on $2^\Z$.
\end{proof}

More generally, if $\mu$ is an ergodic invariant probability Borel measure for a CBER $E$ on a standard Borel space $X$ and $A, B \subseteq X$ are Borel, then $\mu(A) = \mu(B)$ iff $\bbA = (X, A, B)$ admits a Borel expansion $\mu$-a.e. (see e.g. \cite[Lemma~7.10]{KM}). A similar proof gives a characterization of the invariant random $\K$-structures that admit invariant random expansions.

\begin{prop}\label{prop:bijection-concretely-classifiable}
	Let $\Gamma$ be a countably infinite group. There is a Borel $\Gamma$-invariant set $X \subseteq \K(\Gamma)$ and a Borel equivariant expansion map $f: X \to \K^*(\Gamma)$ such that, for all invariant random $\K$-structures $\mu$ on $\Gamma$, $\mu$ admits a random expansion to $\K^*$ if and only if $\mu(X) = 1$, in which case $f_* \mu$ gives such an expansion.

	Moreover, let $\bA = (\Gamma, A, B) \sim \mu$ for an invariant random $\K$-structure $\mu$ on $\Gamma$. If $\mu$ admits an invariant random expansion then $\P[1_\Gamma \in A] = \P[1_\Gamma \in B]$, and the converse holds when $\mu$ is ergodic.

	In particular, if $E$ is a CBER on $Z$ induced by a free Borel action of $\Gamma$, $\mu$ is an $E$-invariant probability Borel measure and $\bbA$ is a Borel $\K$-structuring of $E$, then $\bbA$ is $\mu$-a.e. expandable to $\K^*$ iff $F^\bbA(z) \in X$ for $\mu$-a.e. $z \in Z$.
\end{prop}

\begin{proof}
	The ``in particular'' part follows immediately from \cref{prop:canonical-random-expansions}.

	Let $\Gamma = \{\gamma_n\}$ be an enumeration of $\Gamma$. For $A, B \subseteq \Gamma$, define sets $X^{A, B}_n$ recursively by
	\[X_n^{A, B} = \left( A \setminus \bigcup_{m < n} X_m^{A, B} \right) \cap \left( B \setminus \bigcup_{m < n} X_m^{A, B} \cdot \gamma_m \right) \cdot \gamma_n^{-1}.\]
	The collections $\{X^{A, B}_n\}, \{X^{A, B}_n \cdot \gamma_n\}$ consist of pairwise disjoint sets, and the map taking $\gamma \in X^{A, B}_n$ to $\phi^{A, B}(\gamma) = \gamma \cdot \gamma_n$ is a bijection from $\bigcup_n X^{A, B}_n \subseteq A$ to $\bigcup_n X^{A, B}_n \cdot \gamma_n \subseteq B$.

	It is easy to see by induction that $X^{\gamma \cdot A, \gamma \cdot B}_n = \gamma \cdot X^{A, B}_n$ for $\gamma \in \Gamma$, so the map $(A, B) \mapsto \phi^{A, B}$ is $\Gamma$-equivariant. Additionally, either $\dom(\phi^{A, B}) = A$ or $\ran(\phi^{A, B}) = B$: If $\gamma \in A \setminus \dom(\phi^{A, B}), \gamma' \in B \setminus \ran(\phi^{A, B})$ and $\gamma_n = \gamma^{-1} \gamma'$ then $\gamma \in X^{A, B}_n$, a contradiction.

	We let $X \subseteq \K(\Gamma)$ be the set of all $\bA = (\Gamma, A, B)$ such that $\phi^{A, B}$ is a bijection $A \to B$. It is clear that $X$ is invariant and Borel, and that $\bA \mapsto (\bA, \phi^{A, B})$ defines a Borel equivariant expansion $f: X \to \K^*(\Gamma)$.

	Now let $\mu$ be an invariant random $\K$-structure on $\Gamma$. If $\mu(X) = 1$ then $f_* \mu$ is an invariant random expansion of $\mu$. If $\nu$ is an invariant random expansion of $\mu$ and $(\Gamma, A, B, T) \sim \nu$ then
	\begin{equation}\label{eq:prop-bij-rand-exp}
		\begin{aligned}
			\P[1_\Gamma \in A] &= \P[\exists \gamma((1_\Gamma, \gamma) \in T)] = \sum_\gamma \P[(1_\Gamma, \gamma) \in T] = \sum_\gamma \P[(\gamma^{-1}, 1_\Gamma) \in T] \\
			&= \P[\exists \gamma ((\gamma, 1_\Gamma) \in T)] = \P[1_\Gamma \in B],
		\end{aligned}
	\end{equation}
	and since $\nu$ is a random expansion of $\mu$ we have that $\P[1_\Gamma \in A] = \P[1_\Gamma \in B]$ for $(\Gamma, A, B) \sim \mu$.

	Suppose now $\mu$ is ergodic and $\P[1_\Gamma \in A] = \P[1_\Gamma \in B]$ for $(\Gamma, A, B) \sim \mu$. As in \cref{eq:prop-bij-rand-exp}, we find that $\P[1_\Gamma \in \dom(\phi^{A, B})] = \P[1_\Gamma \in \ran(\phi^{A, B})]$. If $\P[A = \dom(\phi^{A, B})] = 1$ then
	\[\P[1_\Gamma \in B] = \P[1_\Gamma \in A] = \P[1_\Gamma \in \dom(\phi^{A, B})] = \P[1_\Gamma \in \ran(\phi^{A, B})],\]
	and it follows that $\P[B = \ran(\phi^{A, B})] = 1$. Similarly, if $\P[B = \ran(\phi^{A, B})] = 1$ then $\P[A = \dom(\phi^{A, B})] = 1$. By ergodicity, one of these must hold, and so
	\[\P[A = \dom(\phi^{A, B})] = \P[B = \ran(\phi^{A, B})] = 1\]
	and hence $\mu(X) = 1$.

	It remains to show that if $\mu$ admits an invariant random expansion then $\mu(X) = 1$, and for this it suffices to prove that if $\nu$ is an invariant random $\K^*$-structure on $\Gamma$ and $(\Gamma, A, B, T) \sim \nu$ then $\P[(A, B) \in X] = 1$. By considering an ergodic decomposition of $\K^*(\Gamma)$ (cf. \cite[Theorem~5.12]{CBER}) we may assume $\nu$ is ergodic, in which case this follows by the same argument as in the previous paragraph.
\end{proof}

It is not hard to verify that the set $X$ constructed in the proof of \cref{prop:bijection-concretely-classifiable} is a dense $G_\delta$ set in $\K(\Gamma)$, so that the canonical $\K$-structuring of $E^{Fr(\K(\Gamma))}_\Gamma$ admits a Borel expansion on a comeagre invariant Borel set. More generally, we have the following:

\begin{prop}\label{prop:bijection-generically-exp}
	Every aperiodic CBER is generically expandable. In particular, $\K$ admits $\Gamma$-equivariant expansions to $\K^*$ generically for every countably infinite group $\Gamma$.
\end{prop}

\begin{proof}
	The second part follows from the first by \cref{prop:weak-correspondence-expansions}. In order to prove the first part, let $E$ be an aperiodic CBER on a Polish space $X$ and let $R, S \subseteq X$ be Borel sets that have infinite intersection with every $E$-class. Let
	\[x F y \iff x E y \And (x \in R \iff y \in R) \And (x \in S \iff y \in S).\]
	By (the proof of) \cite[Corollary~13.3]{KM}, there is a comeagre $E$-invariant Borel set for which the aperiodic part of $\res{F}{C}$ is compressible. As the finite part $A$ of $\res{F}{C}$ is smooth, and hence so is $\res{E}{[A]_E}$, it suffices to prove the following:

	\begin{claim}
		Suppose $F$ is compressible. Then there is a Borel bijection $R \to S$ whose graph is contained in $E$.
	\end{claim}

	\begin{claimproof}
		We note first that we may assume that $R \setminus S, S \setminus R$ have infinite intersection with every $E$-class. Indeed, $E$ is smooth on the set of points for which this is false, and one can easily construct expansions on smooth CBER. By taking our bijection to be the identity on $R \cap S$, we may therefore assume that $R, S$ are disjoint.

		Fix a Borel action of a countable group $\Gamma$ on $X$ so that $x E y \iff \exists \gamma \in \Gamma (\gamma x = y)$. Fix an enumeration $(\gamma_n)_{n \in \N}$ of $\Gamma$ and for $x \in R$ let $n(x)$ be the least $n$ such that $\gamma_n x \in S$. Let $f(x) = (\gamma_{n(x)}, n(x))$, so that $f: R \to S \times \N$ is a Borel embedding of $\res{F}{R}$ into $\res{F}{S} \times I_\N$, where $I_\N$ is the equivalence relation on $\N$ with a single equivalence class. Since $\res{F}{S}$ is compressible, there is a Borel isomorphism $g: \res{F}{S} \times I_\N \to \res{F}{S}$ such that $x F g(x, n)$ for all $x \in S, n \in \N$ (see e.g. the proof of \cite[Proposition~2.5]{DJK}). Thus $g \circ f: R \to S$ is a Borel injection whose graph is contained in $E$. Since $\res{F}{R}$ is compressible, the proof of \cite[Proposition~2.3]{DJK} applied to $g \circ f$ gives a Borel bijection $R \to S$ whose graph is also contained in $E$.
	\end{claimproof}
\end{proof}

\subsection{Ramsey's Theorem}

Fix $\K, \K^*$ as in \cref{eg:ramsey}, that is,
\begin{align*}
	\K &= \{(X, R, S) \mid R, S ~ \text{partition} ~ [X]^2\},\\
	\K^* &= \{(X, R, S, T) \mid (X, R, S) \in \K \And T \subseteq X ~ \text{is infinite}\\
	& \qquad\qquad\qquad\qquad \text{and homogeneous for the partition} ~ R, S\}.
\end{align*}

\begin{prop}\label{prop:ramsey-generic-expansion}
	Let $\Gamma$ be a countably infinite group. Then $\K$ does not admit $\Gamma$-equivariant expansions to $\K^*$ generically.

	In particular, $(\K, \K^*)$ enforces smoothness.
\end{prop}

\begin{proof}
	The proof of the second part follows from the first by \cref{cor:enforcing-smoothness-generic}.

	Suppose now by way of contradiction that $f: X \to \K^*(\Gamma)$ was a Borel equivariant expansion on an invariant dense $G_\delta$ set $X \subseteq \K(\Gamma)$. Since an expansion $f(x)$ in this case is just a choice of a subset of $\Gamma$ that is homogeneous for the partition given by $x \in X$, we may view $f$ as an equivariant Borel map $X \to 2^\Gamma$.

	Note that the generic element of $\K$ has TAC and is not definable from equality, so by \cref{lem:homogeneous-generic} we may assume that for any $\bA_0 \in \age_\Gamma(\K)$ and any $\bA \in X$ there is some $\gamma \in \Gamma$ such that $\gamma \bA_0 \sqsubseteq \bA$. By shrinking $X$ further still we may assume that $f$ is continuous.

	Now fix an arbitrary $\bA \in X$ and let $T = f(\bA) \subseteq \Gamma$. Fix $\gamma_0 \neq \gamma_1 \in T$. By continuity there is some $\bA_0 \in \age_\Gamma(\K)$ so that $\bA_0 \sqsubseteq \bA$ and $\gamma_0, \gamma_1 \in f(\bB)$ whenever $\bB \in X \cap N(\bA_0)$. By equivariance, $\gamma \gamma_0, \gamma \gamma_1 \in f(\bB)$ whenever $\bB \in X \cap N(\gamma \bA_0)$.

	Let now $F$ be the universe of $\bA_0$, and assume wlog that $\gamma_0, \gamma_1 \in F$. We consider the case where $[T]^2 \subseteq R^\bA$; the case where $[T]^2 \subseteq S^\bA$ is handled identically. Fix $\gamma \in \Gamma$ so that $F \cap \gamma F = \emptyset$, and let $\bA_1 = (F \cup \gamma F, R^{\bA_1}, S^{\bA_1}) \in \age_\Gamma(\K)$ satisfy $\bA_0 \sqsubseteq \bA_1, \gamma \bA_0 \sqsubseteq \bA_1$ and $\{\gamma_0, \gamma \gamma_0\} \in S^{\bA_1}$. Let $\delta$ be such that $\delta \bA_1 \sqsubseteq \bA$. Then $\delta \bA_0 \sqsubseteq \bA, \delta \gamma \bA_0 \sqsubseteq \bA$ so $\delta \gamma_0, \delta \gamma_1, \delta \gamma \gamma_0 \in T$. Also, $\{\delta \gamma_0, \delta \gamma_1\} \in R^\bA$ and $\{\delta \gamma_0, \delta \gamma \gamma_0\} \in S^\bA$. This contradicts the fact that $T$ is homogeneous for the partition $(R^\bA, S^\bA)$.
\end{proof}

\begin{rmk}
	In \cite{GX}, it is shown that an aperiodic CBER is Borel expandable for $(\K, \K^*)$ iff it is smooth, and in particular that $(\K, \K^*)$ enforces smoothness. \Cref{prop:ramsey-generic-expansion}, along with \cref{prop:smooth-expandable,rmk:smooth-expandable}, give an alternative proof of this result. \Citeauthor{GX} consider more generally the case of $k$-colourings of sets of size $n$ (with the appropriate modifications made to $\K, \K^*$) for $k, n \geq 2$ \cite[Theorem~1.3]{GX}; we note that our proof of \cref{prop:ramsey-generic-expansion} goes through in this more general setting as well (where one takes in this case $\bA_1$ to contain $n$ copies of $\bA_0$).

	In \cite[Section~3]{GX}, a variation of the Ramsey extension property is introduced that enforces (and is actually equivalent to) hyperfiniteness. This extension property involves choosing in an ``almost invariant'' way (see \cite[Definition~3.1]{GX}) an expansion on each $E$-class, and in particular does not fit into our framework of expansion problems.
\end{rmk}

\begin{eg}
	We showed above that there is a Borel $\K$-structuring of a CBER $X$ which does not admit an expansion on any comeagre set. Our proof gave a $\K$-structure $(X, R, S)$ such that, when viewing $(X, R)$ as a graph, each connected component is isomorphic to the Rado graph.

	Below are two more such examples. In the first, $(X, R)$ is acyclic, and in the second it is bipartite.
	\begin{enumerate}
		\item By \cite[Proposition~6.2]{KST} there is a Borel acyclic graph $G_0 \subseteq E_0$ on $2^\N$ for which every Borel independent set is meagre. Let $\bbA = (2^\N, G_0, E_0 \setminus G_0)$, and note that this is a Borel $\K$-structuring of $E_0$. If $X \subseteq 2^\N$ is Borel and $E_0$-invariant and $\bbA^* = (\res{\bbA}{X}, T)$ is a Borel expansion of $\res{\bbA}{X}$, then $T \cap C$ is $G_0$-independent for every $E_0$-class $C \subseteq X$ (as $G_0$ is acyclic), so $T$ is meagre. Since $E_0$ is generated by the action of a countable group of automorphisms of $2^\N$, $X = [T]_{E_0}$ is meagre as well. In particular, $X$ is not comeagre, and hence $E_0$ is not generically expandable.

		\item (Kechris) Consider an irrational rotation $R$ on the circle $\bb{T}$ and let $E = E^\bb{T}_R$. Define $f: E \setminus \Delta_{\bb{T}} \to 2$ by $f(x, y) = 1$ iff $R^n(x) = y$ for some even $n \in \Z$, where $\Delta_{\bb{T}} \subseteq \bb{T}^2$ denotes the diagonal. Let $\bbA = (\bb{T}, f^{-1}(0), f^{-1}(1))$ and note that $\bbA$ is a Borel $\K$-structuring of $E$. If $X \subseteq \bb{T}$ is Borel and $E$-invariant and $\bbA^* = (\res{\bbA}{X}, T)$ is a Borel expansion of $\res{\bbA}{X}$. then clearly $f$ takes the value $1$ on $T$. One can easily extend $T$ to a Borel set $T \subseteq A \subseteq X$ so that $R^2(A) = A$. It follows, as $R^2$ is generically ergodic, that $A$, and hence $A \cup R(A) = X$, are meagre. In particular, $X$ is not comeagre, and hence $E$ is not generically expandable.
	\end{enumerate}
	By \cref{prop:enforce-smoothness-generic-exp}, these examples give alternative proofs that $(\K, \K^*)$ enforces smoothness.

	Additionally, both of these examples are not a.e. expandable for the Haar measure, by a similar ergodicity argument. \Cref{prop:rand-graph-no-ramsey} provides another example that is not expandable a.e.
\end{eg}

\begin{prop}\label{prop:rand-graph-no-ramsey}
	Let $\Gamma$ be a countably infinite group and let $\mu$ be the law of the partition $(R, S)$ of $[\Gamma]^2$ obtained by including every pair $\{\gamma, \delta\}$ in $R$ independently with probability $p$, $0 < p < 1$. Then $\mu$ is an invariant random $\K$-structure on $\Gamma$ that does not admit an invariant random expansion.
\end{prop}

\begin{proof}
	It is clear that $\mu$ is $\Gamma$-invariant. Suppose there was an invariant random expansion $\nu$ of $\mu$, and let $(R, S, T) \sim \nu$. We may view $R$ as a graph on $\Gamma$, which is almost surely isomorphic to the Rado graph, and view $T$ as either an infinite clique or an infinite independent set.

	For any finite $F \subseteq \Gamma$, $\res{R}{F}$ is the random graph on $F$ whose edges are included independently with probability $p$. The expected size of the largest clique in $\res{R}{F}$ is $\Theta(\log(|F|))$, and hence so is the expected size of the largest independent set \cite[Theorem~11.4]{RG}. Thus $\E[|F \cap T|] \in O(\log(|F|))$. On the other hand,
	\[\E[|F \cap T|] = \E[\sum_{\gamma \in F} \mathbbl{1}_{\gamma \in T}] = \sum_{\gamma \in F} \P[\gamma \in T] = |F| \cdot \P[1_\Gamma \in T]\]
	by invariance of $\nu$. Taking $|F| \to \infty$ we find that $\P[1_\Gamma \in T] = 0$, and hence that $T = \emptyset$ almost surely, a contradiction.
\end{proof}

\subsection{Linearizations}

Fix
\begin{align*}
	\K &= \{(X, P) \mid P ~ \text{is a partial order on} ~ X\},\\
	\K^* &= \{(X, P, L) \mid (X, P) \in \K \And P \subseteq L \And L ~ \text{is a linear order on} ~ X\}.
\end{align*}
as in \cref{eg:linearizations}.

The expansion problem for invariant random $\K$-structures on groups has been studied in \cite{GLM,A}. In particular, \citeauthor{A} has shown that

\begin{thm}[{\cite[Corollary~1.1]{A}}]
	A countable group $\Gamma$ admits random expansions from $\K$ to $\K^*$ if and only if $\Gamma$ is amenable.
\end{thm}

\begin{rmk}\label{rmk:no-canonical-linearization}
	Note that even in the case when $\Gamma$ is amenable and all invariant random partial orders admit invariant random extensions to linear orders, we cannot expect to find an equivariant map $f: \K(\Gamma) \to \K^*(\Gamma)$ for which an extension is given by the pushfoward measure along $f$. This is unlike the case of bijections, or (as we will see below) for vertex colourings or $\Z$-lines. As a trivial example, note that the empty partial order is a fixed point in $\K(\Gamma)$, and there is no equivariant expansion of this partial order when $\Gamma$ is not left-orderable. We give a more interesting example in \cref{prop:linearization-not-ae}.
\end{rmk}

\begin{prop}\label{prop:linearization-not-ae}
	Let $E$ be an aperiodic CBER. There is a Borel $\K$-structuring of $E$ that is not $\mu$-a.e. expandable for any $E$-invariant probability Borel measure $\mu$.
\end{prop}

\begin{proof}
	We may assume that $E$ is not smooth, as no smooth CBER admits an invariant probability Borel measure.

	It suffices to show this for the CBER $F = \Delta_{2^\N} \times E_0$. Indeed, suppose $\bbA$ were a Borel $\K$-structuring of $F$ with this property. By \cite[Corollary~8.14]{CBER}, we can assume that $E$ lives on $2^\N \times 2^\N$ and that $F \subseteq E$. Then every $E$-invariant measure is also $F$-invariant, $\bbA$ is also a Borel $\K$-structuring of $E$, and if $\bbA^*$ is an expansion of $\bbA$ on $E$ (restricted to some invariant Borel set) then its restriction to each $F$-class is an expansion of $\bbA$ on $F$. It follows immediately that $\bbA$ witnesses that this holds for $E$ as well.

	Let $F_0$ be the index $2$ subequivalence relation of $E_0$ given by
	\[x F_0 y \iff x E_0 y \And |\{i : x(i) \neq y(i)\}| = 0 \mod{2}.\]
	We define a Borel $\K^*$-structuring $L$ of $F_0$ as follows: For $x F_0 y$, we say $x L y$ if either $x = y$ or $\sum_{i < n} x(i) = 0 \mod{2}$, where $n$ is maximal with $x(n) \neq y(n)$. This is clearly reflexive and anti-symmetric. To see that it is transitive, let $x L y L z$ and let $l, m, n$ be maximal such that $x(l) \neq y(l)$, $y(m) \neq z(m)$ and $x(n) \neq z(n)$. It is clear that $l \neq m$. If $l < m$, then $m = n$ and $\sum_{i < m} x(i) = \sum_{i < m} y(i) = 0 \mod{2}$, so $x L z$. Otherwise $l > m$, so $l = n$ and $\sum_{i < l} x(i) = 0 \mod{2}$, and so $x L z$.

	Note that in particular, for all $x L z$ there are only finitely many $y$ with $x L y L z$. Indeed, if $n$ is maximal with $x(n) \neq z(n)$, then $y$ must agree with either $x$ or $z$ at all coordinates $m > n$. Additionally, it is clear that the restriction of $L$ to every $F_0$-class is total.

	Below, we identify $i^k$ for $i \in 2, k \leq \infty$ with the constant sequence of length $k$ with value $i$, and let ${}^\frown$ denote concatenation of sequences. (We also abuse notation and write $i$ for $i^1$.)

	We view $L$ as a Borel $\K$-structuring of $E_0$. Suppose that $X \subseteq 2^\N$ is Borel and $E_0$-invariant, and that $L'$ is a Borel expansion of $\res{L}{X}$ to $\K^*$. We claim that $\mu(X) = 0$, where $\mu$ is the Haar measure on $2^\N$. To see this, define $f: 2^\N \setminus \{1^\frown 0^\infty\} \to 2^\N \setminus \{0^\infty\}$ by
	\[f(0^\frown i^\frown x) = 1^\frown (1-i)^\frown x, \qquad f(1^\frown 0^n {}^\frown 1^\frown i^\frown x) = 0^{n+1} {}^\frown 1^\frown (1-i)^\frown x.\]
	It is clear that this is a Borel function whose graph is contained in $F_0$. Moreover, it is not hard to verify that $f(x)$ is the immediate successor of $x$ in $L$, whenever this is defined. It follows in particular that $L$ orders every $F_0$-class with order-type $\Z$, except for $[0^\infty]_{F_0}, [1^\frown 0^\infty]_{F_0}$, and that $f$ generates $L$ (i.e., $x L y \iff \exists n f^n(x) = y$).

	Let $g: 2^\N \to 2^\N$ be the map which flips the first coordinate of every sequence. Then $g$ is a Borel involution, its graph is contained in $E_0$, and $x \cancel{F_0} g(x)$ for all $x$. It is also easy to verify that $g$ is an isomorphism $F_0 \cong F_0$, and that it is order-reversing for $L$.

	We will use $L'$ to find a Borel $F_0$-invariant set $Y \subseteq X$ so that every $\res{E_0}{X}$-class contains exactly one $F_0$-class in $Y$. It is easy to see that $\mu$ is $F_0$-invariant and $F_0$-ergodic, so this implies that $Y$ is either $\mu$-null or $\mu$-conull. It clearly cannot be $\mu$-conull, so $\mu(Y) = 0$ and hence $\mu(X) = 0$ as well.

	Let now $C$ be an $E_0$-class in $X$. We show how to choose an $F_0$-class from $C$ in a uniformly Borel way, and then take $Y$ to be the set of all such choices. If $0^\infty \in C$, then we choose $[0^\infty]_{F_0}$, so suppose this is not the case. Let $C_0, C_1$ be the two $F_0$-classes in $C$. If some $C_i$ is an initial segment of $\res{L'}{C}$, then we choose $C_i$. Otherwise, we claim that for $i \in 2$ there is a unique pair $(x_i, y_i) \in C_i$ so that $x_i L y_i$, $x_i L' g(x_i)$ and $g(y_i) L' y_i$ (see \cref{fig:F0-linearization-counterexample}). We then take $x$ to be the lexicographically least element of $\{x_0, x_1, y_0, y_1\}$ and choose $[x]_{F_0}$.

	We show this for $i = 0$, the case $i = 1$ being symmetric. As $C_1$ is not an initial segment of $\res{L'}{C}$, there are $x \in C_0, y \in C_1$ with $x L' y$. If $y L g(x)$, then $x L' g(x)$. Otherwise, $g(x) L y$ so $g(y) L x$ (as $g$ is order-reversing) and so $g(y) L' y$. Thus, by possibly setting $x = g(y)$, we may assume that $x L' g(x)$. Clearly there is an $L$-maximal such $x$, as we have assumed that $C_0$ is not an initial segment of $\res{L'}{C}$. We may then take $x_0 = x$, $y_0 = f(x)$.

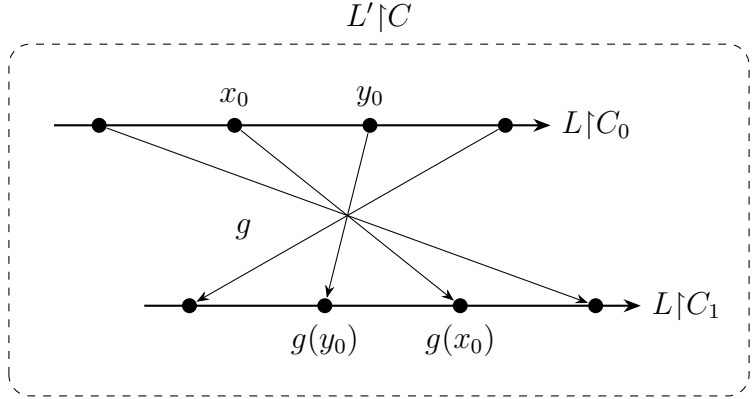
\begin{figure}
\begin{center}
\begin{tikzpicture}[>=Stealth, scale=1.2]
	\coordinate (Tstart) at (-1,2);
	\coordinate (Tend)   at (4.5,2);
	\coordinate (Bstart) at (0,0);
	\coordinate (Bend)   at (5.5,0);

	\draw[thick, ->] (Tstart) -- (Tend);
	\draw[thick, ->] (Bstart) -- (Bend);

	\node [circle, fill, inner sep=2pt] 			 (T0) at (-0.5,2) {};
	\node [circle, fill, inner sep=2pt, label=$x_0$] (T1) at (1,2) {};
	\node [circle, fill, inner sep=2pt, label=$y_0$] (T2) at (2.5,2) {};
	\node [circle, fill, inner sep=2pt] 			 (T3) at (4,2) {};

	\node [circle, fill, inner sep=2pt] 					  (B0) at (0.5,0) {};
	\node [circle, fill, inner sep=2pt, label=below:$g(y_0)$] (B1) at (2,0) {};
	\node [circle, fill, inner sep=2pt, label=below:$g(x_0)$] (B2) at (3.5,0) {};
	\node [circle, fill, inner sep=2pt] 					  (B3) at (5,0) {};

	\draw[->] (T0) -- (B3);
	\draw[->] (T1) -- (B2);
	\draw[->] (T2) -- (B1);
	\draw[->] (T3) -- (B0);

	\node at (1.1,0.85) {\(g\)};
	\node[right] at (Tend) {$\res{L}{C_0}$};
	\node[right] at (Bend) {$\res{L}{C_1}$};

	\draw[dashed, rounded corners=8pt] (-1.5,-1.0) rectangle (6.7,2.9);
	\node at (2.6,3.2) {$\res{L'}{C}$};
\end{tikzpicture}
\end{center}
\caption{There is some $\res{L}{C_0}$-maximal $x_0$ for which $x_0 L' g(x_0)$}
\label{fig:F0-linearization-counterexample}
\end{figure}

	Let now $\bbA$ be the $\K$-structuring of $\Delta_{2^\N} \times E_0$ given by pulling back $L$ along the projection $\proj_2: 2^\N \times 2^\N \to 2^\N$ to the second coordinate (note that this is a class-bijective map $\Delta_{2^\N} \times E_0 \to^{cb}_B E_0$). We claim that for any $(\Delta_{2^\N} \times E_0)$-invariant probability Borel measure $\mu$ and any invariant Borel set $X \subseteq 2^\N \times 2^\N$, if there is a Borel expansion of $\res{\bbA}{X}$ to $\K^*$ then $\mu(X) = 0$. By considering an ergodic decomposition we may assume that $\mu$ is ergodic, in which case it is equal to the Haar measure on $\{x\} \times 2^\N$ for some $x \in 2^\N$, so this follows by the analogous fact for $L$ on $E_0$.
\end{proof}

With respect to category, we have the following:

\begin{prop}\label{prop:linearization-generic-expansion}
	Let $\Gamma$ be a countably infinite group. Then $\K$ does not admit $\Gamma$-equivariant expansions to $\K^*$ generically.

	In particular, $(\K, \K^*)$ enforces smoothness.
\end{prop}

\begin{proof}
	The second part follows from the first by \cref{cor:enforcing-smoothness-generic}.

	Suppose by way of contradiction that there was such an expansion $f: X \to \K^*(\Gamma)$. By shrinking $X$, we may assume that $X$ is $G_\delta$, $f$ is continuous, and by \cref{lem:homogeneous-generic} (and the fact that there is a unique generic partial order) that for every $\bA_0 \in \age_\Gamma(\K)$ and every $\bA \in X$ there is some $\gamma \in \Gamma$ with $\gamma \bA_0 \sqsubseteq \bA$.

	Fix an arbitrary $\bA = (\Gamma, P) \in X$ and let $(\Gamma, L) = f(\bA)$, so that $P \subseteq L$. Since $P$ contains a copy of every element of $\age_\Gamma(\K)$ we may in particular find $\gamma_0, \gamma_1 \in \Gamma$ that are $P$-incomparable. Suppose wlog that $\gamma_0 \mathop{L} \gamma_1$. By continuity and equivariance, there is some $\bA_0 \in \age_\Gamma(\K)$ so that $\bA_0 \sqsubseteq \bA$ and for all $\gamma \in \Gamma$ and $\bB \in X$, if $\gamma \bA_0 \sqsubseteq \bB$ then $\gamma \gamma_0$ is less than $\gamma \gamma_1$ in $f(\bB)$.

	Let now $F$ be the universe of $\bA_0$ and assume wlog that $\gamma_0, \gamma_1 \in F$. Find some $\gamma \in \Gamma$ so that $F \cap \gamma F = \emptyset$, and let $\bA_1 = (F \cup \gamma F, P_1) \in \age_\Gamma(\K)$ be a structure with universe $F \cup \gamma F$ so that $\bA_0 \sqsubseteq \bA_1, \gamma \bA_0 \sqsubseteq \bA_1$ and such that $\gamma_1 \mathop{P_1} \gamma \gamma_0$ and $\gamma \gamma_1 \mathop{P_1} \gamma_0$. Let $\delta$ be such that $\delta \bA_1 \sqsubseteq \bA$. Then $\delta \bA_0 \sqsubseteq \bA, \delta \gamma \bA_0 \sqsubseteq \bA$ so $\delta \gamma_0 \mathop{L} \delta \gamma_1$ and $\delta \gamma \gamma_0 \mathop{L} \delta \gamma \gamma_1$. On the other hand, as $\delta P_1 \subseteq P \subseteq L$, we have $\delta \gamma_1 \mathop{L} \delta \gamma \gamma_0$ and $\delta \gamma \gamma_1 \mathop{L} \delta \gamma_0$ (see \cref{fig:generic-linearization-failure}.) It follows that $\delta \gamma_0 \mathop{L} \delta \gamma_1 \mathop{L} \delta \gamma_0$, a contradiction.

\begin{figure}
\begin{center}
\begin{tikzpicture}
	\draw[thick, rounded corners] (-2.5,1) rectangle (0,-1);
	\node[above] at (-1,1.2) {$\delta \bA_0$};

	\fill (-1,0.5) circle (2pt);
	\node[anchor=east] at (-1.1,0.5) {$\delta \gamma_0$};
	\fill (-1,-0.5) circle (2pt);
	\node[anchor=east] at (-1.1,-0.5) {$\delta \gamma_1$};

	\draw[thick, rounded corners] (2,1) rectangle (4.5,-1);
	\node[above] at (3,1.2) {$\delta \gamma \bA_0$};

	\fill (3,0.5) circle (2pt);
	\node[anchor=west] at (3.1,0.5) {$\delta \gamma \gamma_0$};
	\fill (3,-0.5) circle (2pt);
	\node[anchor=west] at (3.1,-0.5) {$\delta \gamma \gamma_1$};

	\draw[thick, decoration={markings, mark=at position 0.7 with {\arrow{Stealth[length=3mm]}}}, postaction=decorate]
	(3,-0.5) to[out=180, in=0, looseness=1] (-1,0.5);
	\draw[thick, decoration={markings, mark=at position 0.7 with {\arrow{Stealth[length=3mm]}}}, postaction=decorate]
	(-1,-0.5) to[out=0, in=180, looseness=1] (3,0.5);

	\draw[dashed, decoration={markings, mark=at position 0.6 with {\arrow{Stealth[length=2mm]}}}, postaction=decorate]
	(3,0.5) to (3,-0.5);
	\draw[dashed, decoration={markings, mark=at position 0.6 with {\arrow{Stealth[length=2mm]}}}, postaction=decorate]
	(-1,0.5) to (-1,-0.5);
\end{tikzpicture}
\end{center}
\captionsetup{justification=centering,margin=1.5cm}
\caption{A copy of $\delta \bA_1$ in $\bA$. The solid arrows are relations in $\bA_1$, and the dashed arrows are relations that are forced to exist in $L$.}
\label{fig:generic-linearization-failure}
\end{figure}
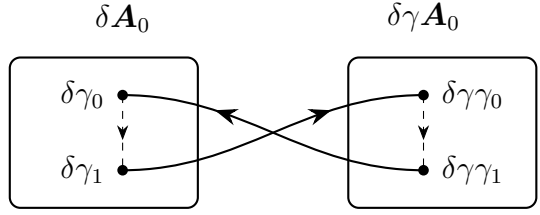
\end{proof}

We consider now expansions on CBER.

\begin{prop}
	Let $E$ be a CBER on a standard Borel space $X$ and let $P$ be a Borel partial order on $E$. Then the set of Borel subsets $Y \subseteq X$ for which $\res{P}{Y}$ admits a Borel linearization on $\res{E}{Y}$ forms a $\sigma$-ideal.
\end{prop}

\begin{proof}
	This class is clearly closed under taking Borel subsets. Suppose now that $Y = \bigcup_n Y_n$ and for every $n$ there is a Borel linearization $L_n$ of $\res{P}{Y_n}$ on $\res{E}{Y_n}$. Let $Z_n = \bigcup_{i < n} Y_n$. We will recursively construct an increasing sequence $\bar{L}_n$ of Borel linearizations of $\res{P}{Z_n}$ on $\res{E}{Z_n}$ and then take $L = \bigcup_n \bar{L}_n$.

	We begin by setting $\bar{L}_0 = \emptyset$. Suppose now that we have constructed $\bar{L}_n$, and let $C$ be an $\res{E}{Z_{n+1}}$-class in order to define $\res{\bar{L}_{n+1}}{C}$. For $x \in C \setminus Z_n$, let
	\[I_x = \{y \in C \cap Z_n : \exists z \in C \cap Z_n (z \mathop{P} x \And y \mathop{\bar{L}_n} z)\}.\]
	Note that $I_x$ is an initial segment of $\res{\bar{L}_n}{(C \cap Z_n)}$. Now for $x, y \in C$, we say $x \mathop{\bar{L}_{n+1}} y$ iff one of the following hold (c.f. \cref{fig:extending-lin-orders}):
	\begin{enumerate}
		\item $x, y \in Z_n$ and $x \mathop{\bar{L}_n} y$,
		\item $x \in Z_n, y \notin Z_n$ and $x \in I_y$,
		\item $x \notin Z_n, y \in Z_n$ and $y \notin I_x$,
		\item $x, y \notin Z_n$ and $I_x \subsetneqq I_y$, or
		\item $x, y \notin Z_n$, $I_x = I_y$ and $x \mathop{L_{n+1}} y$.
	\end{enumerate}
	One easily checks that $\bar{L}_{n+1}$ is a linear order on $Z_{n+1}$.

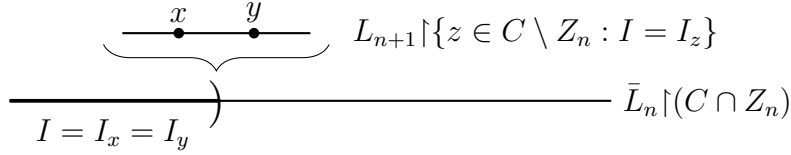
\begin{figure}
\begin{center}
\begin{tikzpicture}
	\draw[thick] (0,0) -- (8,0) node[right] {$\res{\bar{L}_n}{(C \cap Z_n)}$};

	\draw[thick, line width=0.5mm] (0,0) -- (2.80,0);
	\node[scale=1.5] at (2.75,-0.05) {$)$};
	\node[below] at (1.375,-0.1) {$I = I_x = I_y$};

	\draw [decorate,decoration={brace,mirror,amplitude=10pt}] (1.25,0.75) -- (4.25,0.75)
	node[midway, yshift=-10pt] {};

	\draw[thick] (1.5,0.9) -- (4.0,0.9) node[right] {$\quad \res{L_{n+1}}{\{z \in C \setminus Z_n : I = I_z\}}$};
	\fill (2.25,0.9) circle (2pt) node[above] {$x$};
	\fill (3.25,0.9) circle (2pt) node[above] {$y$};
\end{tikzpicture}
\end{center}
\caption{Extending $\bar{L}_n$ to $\bar{L}_{n+1}$.}
\label{fig:extending-lin-orders}
\end{figure}
\end{proof}

In particular, if a partial order $P$ on a CBER $E$ can be decomposed into a countable union of chains and antichains, then $P$ admits a Borel linearization on $E$.

\begin{prop}
	Let $T$ be a Borel locally countable directed tree on a standard Borel space $X$ whose (undirected) connected components form a CBER $E$, and let $P$ be the smallest partial order containing $T$. Then $(E, P)$ admits a Borel linearization.
\end{prop}

\begin{proof}
	For all $x E y$, define $d(x, y)$ as follows: Consider the unique (undirected) path from $x$ to $y$ in $T$. Weigh each edge in this path by $1$ if it occurs in $T$, and by $-1$ if its reverse appears in $T$, and let $d(x, y)$ be the sum of the edge weights along this path. Fix a Borel linear order $\leq$ on $X$ and define $x \mathop{L} y$ if either $d(x, y) > 0$, or $d(x, y) = 0$ and $x \leq y$. It is straightforward to verify that $L$ is a linear order on $E$ extending $P$ (for transitivity, note that $d(x, z) = d(x, y) + d(y, z)$ whenever these are defined).
\end{proof}

\subsection{Vertex colourings}

In this section, we fix $d \geq 2$ and let
\begin{align*}
	\K &= \{(X, E) \mid (X, E) ~ \text{is a connected graph of max degree} ~ \leq d\},\\
	\K^* &= \{(X, E, S_0, \dots, S_d) \mid (X, E) \in \K \And S_0, \dots, S_d ~ \text{is a vertex colouring of} ~ (X, E)\}.
\end{align*}
as in \cref{eg:colourings}.

It was shown in \cite[Proposition~4.6]{KST} that every CBER is Borel expandable for $(\K, \K^*)$, and therefore that $(\K, \K^*)$ admits random expansions by \cref{prop:weak-correspondence-expansions}. Their proof essentially establishes the following (see also \cite[Proposition~4.29]{BC}).

\begin{prop}[Essentially {\cite{KST}}]\label{prop:KST-colouring}
	Let $\Gamma$ be a countably infinite group. There is an invariant dense $G_\delta$ set $Fr(\K(\Gamma)) \subseteq X \subseteq \K(\Gamma)$ which admits a Borel equivariant expansion, and such that $X$ is maximal with this property: For any invariant $X \supsetneqq Y$, there is no equivariant expansion $Y \to \K^*(\Gamma)$.
\end{prop}

Interpreted in the language of \cite{BC}, $X$ consists of exactly the set of orbits which satisfy an appropriate separation axiom (as used for example in the proof of \cite[Proposition~4.29]{BC}).

\begin{proof}
	We say $G \in \K(\Gamma)$ is \tb{bad} if there are $\gamma, \delta \in \Gamma$ so that $\delta \mathop{G} \gamma \delta$ and $\gamma G = \gamma$, and $G$ is \tb{good} otherwise.

	If $G$ is bad, then so is every graph in its orbit $\Gamma \cdot G$, and in this case there is no equivariant expansion map $\Gamma \cdot G \to \K^*(\Gamma)$. Indeed, suppose $c$ were such an expansion map, and view $c(\gamma G)$ as a map $\Gamma \to d+1$ that is a colouring of $\gamma G$ for $\gamma \in \Gamma$. Fix some $\gamma, \delta$ so that $\delta \mathop{G} \gamma \delta$ and $\gamma G = G$. Then
	\[c(G)(\gamma \delta) = c(\gamma G)(\gamma \delta) = (\gamma c(G))(\gamma \delta) = c(G)(\delta)\]
	by equivariance of $c$, a contradiction.

	We take $X$ to be the set of good graphs. Clearly $X$ is a $G_\delta$ set containing the free part of $\K(\Gamma)$. It is also not hard to see that it is dense (for example, the free part is dense by \cref{lem:wdp-free} as the generic element of $\K$ satisfies the WDP and is not definable from equality). It thus remains only to show that $X$ admits a Borel equivariant expansion.

	To see this, it suffices to construct a Borel map $f: X \to d+1$ so that $f(G) \neq f(\gamma^{-1} G)$ whenever $G \in X, \gamma \in \Gamma$ and $1_\Gamma \mathop{G} \gamma$. Indeed, thinking of $f(G)$ as the colour of $1_\Gamma$ in $G$, this extends uniquely to an equivariant map $g: X \to (d+1)^\Gamma$ sending $G$ to the colouring
	\[g(G)(\gamma) = f(\gamma^{-1} G).\]
	It is clear that $g$ is equivariant and Borel, and $g(G)$ is a proper colouring of $G$ because if $\gamma \mathop{G} \gamma \delta$ then $1_\Gamma \mathop{\gamma^{-1} G} \delta$ so
	\[g(G)(\gamma) = f(\gamma^{-1} G) \neq f(\delta^{-1} \gamma^{-1} G) = g(G)(\gamma \delta).\]

	Let $H_n$ be an enumeration of $\age_\Gamma(\K)$ and let $X_n$ be the set of all $G \in X \cap N(H_n)$ such that $\gamma^{-1} \mathop{G} \notin N(H_n)$ for all neighbours $\gamma$ of $1_\Gamma$ in $G$. It is clear that at most one of $G, \gamma^{-1} G \in X_n$ whenever $G \in X$ and $1_\Gamma \mathop{G} \gamma$. We claim that $X = \bigcup_n X_n$. Indeed, as $G \in X$ is good and has bounded degree, there is some finite $F \subseteq \Gamma$ so that $\res{G}{F} \neq \res{\gamma^{-1} G}{F}$ for all neighbours $\gamma$ of $1_\Gamma$ in $G$, in which case $G \in X_n$ for $n$ such that $H_n = \res{G}{F}$.

	Let $Y_n = X_n \setminus \bigcup_{i < n} X_i$. The sets $Y_n$ partition $X$ and if $G \in Y_n$ and $1_\Gamma \mathop{G} \gamma$ then $\gamma^{-1} G \notin Y_n$. We now define $f: X \to d+1$ recursively on each $Y_n$ as follows: supposing $f$ has already been defined on $\bigcup_{i < n} Y_i$, we define $f(G)$ for $G \in Y_n$ to be the least element of $d+1$ which is not equal to $f(\gamma^{-1} G)$ for all neighbours $\gamma$ of $1_\Gamma$ in $G$.
\end{proof}

Note that if we replace $\K$ with the class of $d$-regular graphs, then the sets $X_n$ in the previous proof are clopen, so the construction yields a continuous equivariant expansion $X \to \K^*(\Gamma)$.

One can also consider vertex colourings with fewer colours. This has been studied extensively in the case of expansions on CBER; we refer the reader to \cite[Part~I]{KM-survey} for a survey of this topic.

\subsection{Spanning trees}\label{sec:spanning-trees}

In this section, consider
\begin{align*}
	\K &= \{(X, E) \mid (X, E) ~ \text{is a connected graph}\},\\
	\K^* &= \{(X, E, T) \mid (X, E) \in \K \And (X, T) ~ \text{is a spanning subtree of} ~ (X, E)\},
\end{align*}
as in \cref{eg:trees}.

We say a CBER $E$ is \tb{treeable} if there is a Borel $\K'$-structuring of $E$, where $\K'$ is the class of connected trees. Let $\mathcal{T}$ denote the class of treeable CBER, and say an expansion problem \tb{enforces treeability} if it enforces $\mathcal{T}$.

A countably infinite group $\Gamma$ is \tb{antitreeable} if for every free Borel action of $\Gamma$ on a standard Borel space $X$ admitting an invariant probability Borel measure, the CBER $E^X_\Gamma$ is not treeable (c.f. \cite[115]{CBER}).

\begin{prop}\label{prop:spanning-trees}
	Let $E$ be a CBER. If $E$ is hyperfinite then it is Borel expandable for $(\K, \K^*)$, and if it is Borel expandable for $(\K, \K^*)$ then it is treeable.

	In particular,
	\begin{enumerate}
		\item $(\K, \K^*)$ enforces treeability;
		\item $\K$ admits $\Gamma$-equivariant expansions to $\K^*$ generically for all countably infinite groups $\Gamma$;
		\item $\Gamma$ admits random expansions from $\K$ to $\K^*$ for all amenable groups $\Gamma$; and
		\item $\Gamma$ does not admit random expansions from $\K$ to $\K^*$ for all antitreeable groups $\Gamma$.
	\end{enumerate}
\end{prop}

\begin{proof}
	It is clear that if $E$ is Borel expandable for $(\K, \K^*)$ then it is treeable (consider the complete graph on each $E$-class), and in particular that $(1)$ holds.

	To see that hyperfinite CBER are Borel expandable for $\K, \K^*$, let $E$ be a hyperfinite CBER and $\bbA$ be a Borel $\K$-structuring of $E$. Write $E = \bigcup_n E_n$ for an increasing union of finite CBER. We recursively construct an increasing sequence of Borel sets $\bbA^*_n$ so that $\bbA^*_n \subseteq E_n$ and for every $E_n$-class $C$, $\res{\bbA^*_n}{C}$ is a spanning subforest of $\res{\bbA}{C}$. We describe the construction of $\bbA^*_{n+1}$, given $\bbA^*_n$.

	Let $C$ be an $E_{n+1}$-class, $G = \res{\bbA}{C}$, $T = \res{\bbA^*_n}{C}$. Then $T \subseteq G$ is a forest of trees, so we can easily find a spanning forest $T \subseteq T' \subseteq G$. We set $\res{\bbA^*_{n+1}}{C} = T'$. As every $E_{n+1}$-class is finite, it is clear that this can be done in a uniformly Borel way.

	It follows that $Fr(\K(\Z))$ admits an equivariant random expansion to $\K^*$, and by \cref{thm:generic-expansions} this holds for all countably infinite groups $\Gamma$. $(3)$ follows by \cref{prop:weak-correspondence-expansions}(3), as every CBER generated by a Borel action of an amenable group is measure-hyperfinite, and $(4)$ is an immediate consequence of \cref{prop:random-expansion-cber}.
\end{proof}

Thus the class of CBER that are Borel expandable for $(\K, \K^*)$ lies somewhere between the hyperfinite and the treeable CBER.

\begin{prob}
	Is every treeable CBER expandable for $(\K, \K^*)$? Does $(\K, \K^*)$ enforce hyperfiniteness?
\end{prob}

\subsection{\texorpdfstring{$\Z$}{Z}-lines}

Let
\begin{align*}
	\K &= \{(X, L) \mid (X, L) ~ \text{is a linear order without endpoints}\},\\
	\K^* &= \{(X, L, Z) \mid (X, L) \in \K \And Z \subseteq X \And (Z, \res{L}{Z}) \cong (\Z, <)\},
\end{align*}
as in \cref{eg:zline}.

\begin{prop}\label{prop:z-line-generic-expansion}
	Let $\Gamma$ be a countably infinite group. Then $\K$ does not admit $\Gamma$-equivariant expansions to $\K^*$ generically.

	In particular, $(\K, \K^*)$ enforces smoothness.
\end{prop}

\begin{proof}
	The second part follows from the first and \cref{cor:enforcing-smoothness-generic}.

	For $\bA_0 \in \age_\Gamma(\K)$ and $\bA \in \K(\Gamma)$, let $C(\bA_0, \bA) = \{\gamma : \gamma \bA_0 \sqsubseteq \bA\}$. We say $\bA$ \tb{contains $\bA_0$ densely often} if $\res{\bA}{C(\bA_0, \bA)}$ is a dense linear order with at least two points.

	We claim that the generic element of $\K(\Gamma)$ contains every $\bA_0 \in \age_\Gamma(\K)$ densely often. As there are only countably many such $\bA_0$, it suffices to show this for some $\bA_0$. It is easy to see that the set of all $\bA$ for which $C(\bA_0, \bA)$ contains at least two points is open and dense, so we show that the set of $\bA$ for which $\res{\bA}{C(\bA_0, \bA)}$ is dense is a dense $G_\delta$ set.

	To see this, we show that for any fixed $\gamma_0, \gamma_1 \in \Gamma$, the set of $\bA$ satisfying
	\[\gamma_0, \gamma_1 \in C(\bA_0, \bA) \And \gamma_0 \mathop{L^\bA} \gamma_1 \implies \exists \delta (\delta \in C(\bA_0, \bA) \And \gamma_0 \mathop{L^\bA} \delta \mathop{L^\bA} \gamma_1)\]
	is dense and open. This set is clearly open. To see that it is dense, fix $\bB_0 \in \age_\Gamma(\K)$ and let $\bA \in N(\bB_0)$. If $\gamma_i \notin C(\bA_0, \bA)$ for some $i \in 2$ or $\gamma_1 \mathop{L^\bA} \gamma_0$, we are done. Otherwise, we may assume that $\gamma_0, \gamma_1$ are in the universe of $\bB_0$ and that $\gamma_i \bA_0 \sqsubseteq \bB_0$ for $i \in 2$. Let $F$ be the universe of $\bB_0$ and fix $\delta$ so that $F \cap \delta \gamma_0^{-1} F = \emptyset$. Let $\bB_1$ be a linear order with universe $F \cup \delta \gamma_0^{-1} F$ so that $\bB_0 \sqsubseteq \bB_1, \delta \gamma_0^{-1} \bB_0 \sqsubseteq \bB_1$ and $\gamma_0 \mathop{L^{\bB_1}} \delta \mathop{L^{\bB_1}} \gamma_1$. Then for any $\bB \in N(\bB_1) \subseteq N(\bB_0)$ we have $\gamma_0 \bA_0, \gamma_1 \bA_0, \delta \bA_0 \sqsubseteq \bB$ and $\gamma_0 \mathop{L^\bB} \delta \mathop{L^\bB} \gamma_1$.

	Suppose now that there is a comeagre Borel equivariant set $X \subseteq \K(\Gamma)$ and a Borel equivariant expansion $f: X \to \K^*(\Gamma)$. By shrinking $X$, we may assume that it is $G_\delta$ and that $f$ is continuous. We view $f$ as a function $X \to 2^\Gamma$ taking $\bA \in X$ to a subset of $\Gamma$ so that $\res{\bA}{f(\bA)} \cong \Z$.

	Fix now some $\bA \in X$ in which every element of $\age_\Gamma(\K)$ appears densely often and let $\gamma_0 \in f(\bA)$ be arbitrary. By continuity and equivariance, there is some $\bA_0 \in \age_\Gamma(\K)$ so that $\bA_0 \sqsubseteq \bA$ and whenever $\gamma \bA_0 \sqsubseteq \bB \in X$ we have $\gamma \gamma_0 \in f(\bB)$. In particular, $C(\bA_0, \bA) \subseteq f(\bA)$. But $\bA_0$ appears densely often in $\bA$, contradicting the fact that $\res{\bA}{f(\bA)} \cong \Z$.
\end{proof}

We consider now the measurable case. For this, we will need the following lemma, due to \textcite{LS_indistinguishability}, on the existence of ``densities'' of infinite random subsets of a group $\Gamma$ (see also \cite[Section~4.2]{HP_percolation_threshold}).

Let $\Gamma$ be a countably infinite group and let $(Z_n)_{n \in \N}$ be a random walk on $\Gamma$ with symmetric step distribution $\mu$ whose support generates $\Gamma$. (Note that we do not assume $\mu$ to be finitely supported.) For $\gamma \in \Gamma$, let $\bb{P}_\gamma$ denote the law of the random walk $(Z_n)_{n \in \N}$ starting at $\gamma$.

Define $\Omega(\Gamma, \mu)$ to be the set of all $W \subseteq \Gamma$ for which there exists $r \in [0, 1]$ so that
\[\lim_{n \to \infty} \frac{1}{n} \sum_{i = 0}^{n-1} \mathbbl{1}(Z_i \in W) = r, ~\bb{P}_\gamma\text{-a.s. for all } \gamma \in \Gamma,\]
and for $W \in \Omega(\Gamma, \mu)$ we let $\freq_\mu(W)$ be the unique such $r$. We note that $\Omega(\Gamma, \mu) \subseteq 2^\Gamma, \freq_\mu: \Omega(\Gamma, \mu) \to [0, 1]$ are Borel and $\Gamma$-invariant. (Here, $\mathbbl{1}(Z_n \in W)$ is equal to $1$ when $Z_n \in W$ and $0$ otherwise.)

\begin{lem}[Existence of frequencies {\cite[Lemma~4.2]{LS_indistinguishability}}; c.f. {\cite[Lemma~4.4]{HP_percolation_threshold}}]\label{lem:existence-of-frequencies}
	Let $\Gamma$ be a countably infinite group and $\mu$ be a symmetric probability measure on $\Gamma$ whose support generates $\Gamma$. Let $\nu$ be an invariant random equivalence relation on $\Gamma$ and let $E \sim \nu$. Then $\nu$-almost surely, every equivalence class of $E$ is contained in $\Omega(\Gamma, \mu)$.
\end{lem}

We include a proof of \cref{lem:existence-of-frequencies} in \cref{app:proof-of-freq}, for the reader's convenience.

\begin{prop}\label{prop:z-line-classification}
	Let $\Gamma$ be a countably infinite group. There is a Borel $\Gamma$-invariant set $X \subseteq \K(\Gamma)$ and a Borel equivariant expansion map $f: X \to \K^*(\Gamma)$ such that, for all invariant random $\K$-structures $\mu$ on $\Gamma$, $\mu$ admits a random expansion to $\K^*$ if and only if $\mu(X) = 1$, in which case $f_* \mu$ gives such an expansion.

	Moreover, we can choose $f$ so that for all $\bA \in X$, $f(\bA)$ picks out an interval $I$ in $\bA$ with $\res{\bA}{I} \cong \Z$.

	In particular, if $E$ is a CBER on $Z$ induced by a free Borel action of $\Gamma$, $\mu$ is an $E$-invariant probability Borel measure and $\bbA$ is a Borel $\K$-structuring of $E$, then $\bbA$ is $\mu$-a.e. expandable to $\K^*$ iff $F^\bbA(z) \in X$ for $\mu$-a.e. $z \in Z$.
\end{prop}

\begin{proof}
	The ``in particular'' part follows immediately from \cref{prop:canonical-random-expansions}.

	For notational convenience, we identify $\bA \in \K(\Gamma)$ with $L^\bA$. Let $\kappa$ be a fixed symmetric probability measure on $\Gamma$ whose support generates $\Gamma$.

	For a given $L \in \K(\Gamma)$, let $Z_L$ denote the set of all intervals $I$ in $L$ for which $\res{L}{I} \cong \Z$. We define $X \subseteq \K(\Gamma)$ to be the set of all $L$ for which $\sup_{I \in Z_L} \freq_\kappa(I)$ exists, is non-zero and is attained by finitely many $I \in Z_L$. For such $L$, we define $f(L)$ to be the $L$-least interval $I \in Z_L$ maximizing $\freq_\kappa(I)$. It is clear that $f$ gives a Borel equivariant expansion $X \to \K^*(\Gamma)$.

	Suppose now that $\mu$ is an invariant random $\K$-structure on $\Gamma$ admitting a random expansion $\nu$. Let $(L, S) \sim \nu$ be a random variable with law $\nu$. We claim that $\nu$-almost surely, for all $x, y \in S$, there are finitely many points between $x$ and $y$ in $L$. To see this, define $g(x, y, L, S)$, for $x, y \in \Gamma$, by setting $g(x, y, L, S) = 1$ if $y$ is the $L$-least element of $S$ with $x \mathop{L} y$, and $0$ otherwise. Note that $\sum_y g(x, y, L, S)$ is equal to $1$ if $x$ lies between two elements of $S$, and $0$ otherwise. On the other hand, $\sum_x g(x, y, L, S)$ is $0$ when $y \notin S$, and when $y \in S$ it is equal to the size of the interval $(z, y]$ in $L$, where $z$ is the $\res{L}{S}$-predecessor of $y$.

	Let now $G(x, y) = \E[g(x, y, L, S)]$. Then by the mass transport principle (which in this case follows simply from the invariance of $\nu$),
	\[\sum_x G(x, y) = \sum_x G(y, x) = \E[\sum_x g(y, x, L, S)] \leq 1\]
	for all $y \in \Gamma$. It follows that the size of the interval $[y, z]$ in $L$ is almost surely finite for all $y, z \in S$. In particular, if we take $I(L, S)$ to be the smallest interval in $L$ containing $S$, then $I(L, S) \in Z_L$ almost surely. Thus, by replacing $\nu$ with the law of $(L, I(L, S))$, we may assume that $S \in Z_L$.

	We now show that $L \in X$ almost surely, i.e., $\mu(X) = 1$. By \cref{lem:existence-of-frequencies} we may assume that $I \in \Omega(\Gamma, \kappa)$ for all $I \in Z_L$, i.e., that $\freq_\kappa(I)$ is defined for all such intervals. By considering an ergodic decomposition of $\nu$ (cf. \cite[Theorem~5.12]{CBER}), we may assume that $\nu$ is ergodic. Let $\P[1_\Gamma \in S] = r > 0$. We claim that $\freq_\kappa(S) = r$ almost surely. Indeed, by ergodicity $\freq_\kappa(S)$ is constant a.s., and by the Dominated Convergence Theorem and invariance we have
	\[\E[\freq_\kappa(S)] = \lim_{n \to \infty} \frac{1}{n} \sum_{i = 0}^{n-1} \P[Z_i \in S] = \lim_{n \to \infty} \frac{1}{n} \sum_{i = 0}^{n-1} r = r.\]
	It follows that $\sup_{I \in Z_L} \freq_\kappa(I) > 0$, and by Fatou's Lemma
	\[\sum_{I \in Z_L} \freq_\kappa(I) \leq \freq_\kappa(\bigcup Z_L) \leq 1,\]
	so the max is attained by finitely many $I \in Z_L$.
\end{proof}

\subsection{Vizing's Theorem}\label{sec:vizing}

Fix $d \geq 2$ and let
\begin{align*}
	\K &= \{(X, E) \mid (X, E) ~ \text{is a connected graph of max degree} ~ \leq d\},\\
	\K^* &= \{(X, E, S_0, \dots, S_d) \mid (X, E) \in \K \And S_0, \dots, S_d ~ \text{is an edge colouring of} ~ (X, E)\},
\end{align*}
as in \cref{eg:vizing}.

By Vizing's Theorem, every element of $\K$ admits an expansion in $\K^*$. This is false in the Borel context. In particular, \citeauthor{Marks} has shown that there is a $d$-regular acyclic Borel bipartite graph with Borel edge-chromatic number $2d-1$ \cite{Marks}, and in \cite{CJMSTD} it is shown that there are counter-examples even for hyperfinite graphs.

On the other hand, Vizing's Theorem holds in the Borel setting for $d = 2$ \cite{KST} and for graphs of subexponential growth \cite{BD}, in the measurable setting \cite{GP,G}, and in the Baire-measurable setting for bipartite graphs \cite{BW}.

To summarize, we have the following:

\begin{thm}
\begin{enumerate}
	\item[]
	\item \cite{CJMSTD} For $d \geq 3$, $(\K, \K^*)$ enforces smoothness.
	\item \cite{KST} For $d = 2$, every CBER is Borel expandable for $(\K, \K^*)$.
	\item \cite{BD} Let $\K_0 \subseteq \K$ be the subclass of graphs of subexponential growth. Then every CBER is Borel expandable for $(\K_0, \K^*)$.
	\item \cite{BW} Let $\K_1 \subseteq \K$ be the subclass of bipartite graphs. Then every CBER is generically expandable for $(\K_1, \K^*)$. In particular, $\K_1$ admits $\Gamma$-equivariant expansions to $\K^*$ generically for every countably infinite group $\Gamma$.
	\item \cite{GP,G} Every CBER is a.e. expandable for $(\K, \K^*)$ for every (not necessarily invariant) probability Borel measure. In particular, for every countably infinite group $\Gamma$:
	\begin{enumerate}
		\item there is a Borel invariant set $Z \subseteq Fr(\K(\Gamma))$ which admits a Borel $\Gamma$-equivariant expansion to $\K^*$ and such that every invariant random $\K$-structure $\mu$ on $\Gamma$ which concentrates on $Fr(\K(\Gamma))$ satisfies $\mu(Z) = 1$; and
		\item $\Gamma$ admits random expansions from $\K$ to $\K^*$.
	\end{enumerate}
\end{enumerate}
\end{thm}

\begin{proof}
	By \cite[Theorem~1.4]{CJMSTD}, we may fix an aperiodic hyperfinite CBER $E$ and Borel $d$-regular acyclic graph $G$ on $E$ which does not admit a Borel edge colouring with $d+1$ colours. By \cite[Lemma~3.23]{JKL}, $E$ does not admit an invariant probability Borel measure (as $G$ is a treeing of $E$ for which every component has infinitely-many ends), so by Nadkarni's Theorem $E$ is compressible. (1) follows by \cref{prop:enforce-smoothness-generic-exp}.

	For the ``in particular'' parts of (4), (5), note that $Fr(\K_1(\Gamma))$ is dense $G_\delta$ in $\K_1(\Gamma)$ and apply \cref{prop:weak-correspondence-expansions}. For (5a), apply the proof of \cite[Theorem~4.3]{GP} to the canonical $\K$-structuring of $Fr(\K(\Gamma))$, as in the proof of (1) above.
\end{proof}

\begin{rmk}
	When $d = 2$, the same construction as in the proof of \cref{prop:KST-colouring} gives an analogous characterization of exactly when there is a Borel equivariant expansion map from $Z \subseteq \K(\Gamma)$ to $\K^*$.

	We note also that $(\K, \K^*)$ enforces smoothness even if we restrict $\K$ to the class of $n$-regular acyclic bipartite graphs (with a given bipartition), for $n > d/2+1$, by \cite[Theorem~1.4]{CJMSTD}.

	Finally, we remark that the proof of the main result of \cite{GP} (along with Nadkarni's Theorem) gives the stronger fact that for every CBER $E$ on $X$ and Borel $\K$-structuring $\bbA$ of $E$, there is a Borel $E$-invariant set $C$ so that $\res{\bbA}{C}$ admits a Borel expansion to $\K^*$ and $\res{E}{(X \setminus C)}$ is compressible.
\end{rmk}

In the Baire-measurable setting, \citeauthor{QW} have shown that if we replace $\K^*$ with $(d+2)$-edge colourings, then every CBER is generically expandable \cite{QW}. It is open whether every CBER is generically expandable for $(\K, \K^*)$.

\subsection{Matchings}\label{sec:matchings}

As in \cref{eg:matchings}, let
\begin{align*}
	\K &= \{(X, E) \mid (X, E) ~ \text{is a connected, bipartite,}\\
	&\qquad\qquad\qquad\text{locally finite graph satisfying Hall's Condition}\},\\
	\K^* &= \{(X, E, M) \mid (X, E) \in \K \And M \subseteq E ~ \text{is a perfect matching}\}.
\end{align*}
All graphs below are assumed to be in $\K$, unless specified otherwise.

By Hall's Theorem, every element of $\K$ admits an expansion in $\K^*$. This is false in the Borel context. \Citeauthor{Lac} and \citeauthor{Conley-Kechris} have given examples of $d$-regular hyperfinite graphs with Borel chromatic number $2$ which do not admit Borel perfect matchings, even generically or a.e., for $d$ even \cite{Lac,Conley-Kechris}. \Citeauthor{Marks} later showed that there are $d$-regular, acyclic graphs with Borel chromatic number $2$ that do not have Borel perfect matchings for all $d \geq 2$ \cite{Marks}, and in \cite[Theorem~1.4]{CJMSTD} this was extended to hyperfinite graphs. \Citeauthor{Kun} has given examples of such graphs that are not hyperfinite and do not admit Borel perfect matchings a.e. \cite{Kun}, and in \cite{BKS} a hyperfinite one-ended bounded-degree graph with Borel chromatic number $2$ is constructed which does not admit a Borel perfect matching a.e.

On the other hand, if we strengthen our structural assumptions on the graphs one can guarantee the existence of Borel perfect matchings generically or a.e. For instance, \citeauthor{MU} have shown that if we strengthen Hall's Condition to assume that $|N(A)| \geq (1+\varepsilon) |A|$ for some fixed $\varepsilon > 0$ then there is always a Borel perfect matching generically \cite{MU}, and \citeauthor{LN} have shown that Borel perfect matchings exist a.e. for graphs that instead satisfy an analogous expansion property for measure \cite{LN}. \Citeauthor{CM} have shown that acyclic graphs of minimum degree at least $2$ which do not have infinite injective rays of degree $2$ on even vertices have Borel perfect matchings generically, and a.e. when the graph is hyperfinite \cite{CM} (they showed this even for locally countable graphs in the measurable setting). \Citeauthor{BKS} have shown that Borel perfect matchings exist a.e. for hyperfinite measure-preserving regular graphs that are one-ended or have odd degree \cite{BKS}, and in \cite{BCW} the odd-degree case is shown to hold even when the measure is not preserved. Borel perfect matchings also exist generically for regular graphs that are one-ended \cite{BPZ} or have odd degree \cite{BCW}, and for bounded-degree non-amenable vertex-transitive graphs \cite{KL} (note that this last result applies to all graphs, not just those in $\K$).

We note also that Borel perfect matchings have been shown to exist a.e. for some Schreier graphs of free actions of groups; see e.g. \cite{LN,MU,CL,GMP,BKS,GJKS,W-EC} and \cite[Sections~14, 15]{KM-survey}.

\begin{rmk}\label{rmk:fractional-perfect-matching}
	Let $G$ be any graph (not necessarily in $\K$). A \tb{fractional perfect matching} on $G$ is an assignment to each edge of $G$ a weight in $[0, 1]$ so that for every vertex $v$ in $G$, the sum of the weights of the edges incident to $v$ is equal to $1$. Perfect matchings are then the same as $\{0, 1\}$-valued fractional perfect matchings. We say a fractional perfect matching is \tb{non-integral} if it takes values in $(0, 1)$.

	The general strategy employed by \cite{BKS,BPZ,BCW} to find Borel perfect matchings in a Borel locally finite graph $G$ is to start with a Borel non-integral fractional perfect matching on $G$, and then to attempt to round this Borel fractional perfect matching to be $\{0, 1\}$-valued (off of a meagre or null set).

	When $G$ is $d$-regular there is always a Borel non-integral fractional perfect matching on $G$, namely the one giving weight to $1/d$ to every edge. However, Borel non-integral fractional perfect matchings can also be shown to exist (possibly off of a meagre or null set) in other contexts; see \cite{Timar} for an example of this in the measurable setting. The results of these papers can therefore be applied to a larger class of graphs than e.g. the regular ones.

	It may therefore be interesting to consider separately the expansion problems for $(\K, \K')$ and $(\K', \K^*)$, where $\K'$ is the class of graphs equipped with a (non-integral) fractional perfect matching, though we do not explore this here.
\end{rmk}

We summarize a few of the aforementioned results below, in the language and setting of expansions.

Let $\K_d$ (resp. $\K_{d, ac}$) denote the subclass of $\K$ consisting of $d$-regular (resp. $d$-regular acyclic) graphs. Note that these are $G_\delta$ classes of structures. Let $\K_0 \subseteq \K$ denote the class of graphs that are either acyclic with no infinite injective rays of degree $2$ on even vertices, are regular and one-ended, or are regular of odd degree. Let $\K_1 \subseteq \K$ denote the class of graphs that satisfy the strengthening of Hall's Condition for $\varepsilon$-expansion for some $\varepsilon > 0$, or are vertex-transitive, non-amenable and have bounded degree. These are Borel classes of structures.

\begin{thm}
	\begin{enumerate}
		\item[]

		\item \cite{CJMSTD} $(\K_{d, ac}, \K^*)$ enforces smoothness for $d \geq 2$. \cite{Lac,Conley-Kechris} In particular, $(\K_d, \K^*)$ and $(\K, \K^*)$ enforce smoothness.

		\item $\K_2$ does not admit $\Gamma$-equivariant expansions to $\K^*$ generically, for any countably infinite group $\Gamma$.

		\item For every countably infinite group $\Gamma$ and $d > 2$, $\K_d$ admits $\Gamma$-equivariant expansions to $\K^*$ generically. \cite{MU} So does $\K_{d, ac}$.

		\item \cite{MU,CM,BPZ,BCW,KL} Every CBER is generically expandable for $(\K_0 \cup \K_1, \K^*)$.

		\item \cite{CM,BKS} Every hyperfinite CBER is a.e. expandable for $(\K_0, \K^*)$ for every invariant probability Borel measure. In particular, for every countably infinite amenable group $\Gamma$:
		\begin{enumerate}
			\item there is a Borel invariant set $Z \subseteq Fr(\K_0(\Gamma))$ which admits a Borel $\Gamma$-equivariant expansion to $\K^*$ and such that every invariant random $\K_0$-structure on $\Gamma$ which concentrates on $Fr(\K_0(\Gamma))$ satisfies $\mu(Z) = 1$; and
			\item $\Gamma$ admits random expansions from $\K_0$ to $\K^*$.
		\end{enumerate}
	\end{enumerate}
\end{thm}

\begin{proof}
	(1) By \cite[Theorem~1.4]{CJMSTD}, there is an aperiodic hyperfinite CBER and a Borel $d$-regular acyclic graph $G$ on $E$ which does not admit a Borel perfect matching. By \cite[Lemma~3.23]{JKL}, $E$ does not admit an invariant probability Borel measure (as $G$ is a treeing of $E$ for which every component has infinitely-many ends), so by Nadkarni's Theorem $E$ is compressible. We then apply \cref{prop:enforce-smoothness-generic-exp}.

	(2) Suppose otherwise, and let $X \subseteq \K_2(\Gamma)$ be Borel, comeagre and invariant, and let $f: X \to \K^*(\Gamma)$ be a Borel equivariant expansion. It is not hard to see that the set of all $\bA$ for which
	\[\text{for all $\bA_0 \in \age_\Gamma(\K_2)$ there is some $\gamma \in \Gamma$ with $\gamma \bA_0 \sqsubseteq \bA$}\]
	is a dense $G_\delta$ set in $\K_2(\Gamma)$, and we may therefore assume that every element of $X$ has this property. By further shrinking $X$, we may assume that $f$ is continuous.

	Fix now some $\bA \in X$ and $\gamma_0, \gamma_1 \in \Gamma$ so that $\gamma_0, \gamma_1$ are matched in $f(\bA)$. By continuity and equivariance, there is some $\bA_0 \in \age_\Gamma(\K_2)$ so that $\bA_0 \sqsubseteq \bA$, and whenever $\gamma \bA_0 \sqsubseteq \bB \in X$ we have that $\gamma \gamma_0, \gamma \gamma_1$ are matched in $f(\bB)$.

	Let $\bA_1 \in \age_\Gamma(\K_2)$ and $\gamma \in \Gamma$ be such that $\bA_0, \gamma \bA_0 \sqsubseteq \bA_1$, $\bA_1$ is connected, and the unique path in $\bA_1$ whose first edge is $\{\gamma_0, \gamma_1\}$ and whose last edge is $\{\gamma \gamma_0, \gamma \gamma_1\}$ has even length. Let $\delta$ be such that $\delta \bA_1 \sqsubseteq \bA$. Then $\{\delta \gamma_0, \delta \gamma_1\}, \{\delta \gamma \gamma_0, \delta \gamma \gamma_1\} \in f(\bA)$, but the unique path in $\bA$ containing these edges at either end has even length, which is impossible as $\bA$ is a bi-infinite line and $f(\bA)$ is a perfect matching.

	(3) By \cite[Theorem~1.2]{BPZ}, it suffices by \cref{prop:weak-correspondence-expansions}(2) to show that the generic element of $\K_d$ is one-ended (note that $Fr(\K_d(\Gamma))$ is comeagre in $\K_d(\Gamma)$). To see this, note that a $d$-regular graph $G$ is one-ended if and only if for every finite set $F$ of vertices and all vertices $u, v$, one of the following holds:
	\begin{itemize}
		\item There is some finite set of vertices $F'$ such that at least one of $u, v$ is contained in $F'$, and the boundary of $F'$ is contained in $F$.
		\item There is a path in $G$ from $u$ to $v$ which does not include any vertices in $F$.
	\end{itemize}
	It is easy to see that the set of graphs satisfying one of these conditions for any fixed $F, u, v$ is open and dense in $\K_d$, and hence the set of graphs satisfying these conditions for all $F, u, v$ is comeagre.

	For $\K_{d, ac}$, this follows by \cref{prop:weak-correspondence-expansions}(2) and \cite[Theorem~1.3]{MU}.

	(4) is an immediate consequence of the (proofs in) the cited papers, and (5b) follows similarly by \cref{prop:weak-correspondence-expansions}(3).

	(5a) We split $\K_0$ into three parts: The acyclic graphs with no infinite injective rays of degree $2$ on even vertices, the regular one-ended graphs, and the regular odd-degree graphs. We will give some detail for the last case, and then sketch the first two.

	For the regular odd-degree graphs, we argue as follows: We consider each degree $d \geq 3$ separately. Let $X$ be the free part of $\K_0(\Gamma)$ restricted to the regular $d$-degree graphs and let $\bbA$ be the canonical structuring of $X$. By the proof of \cite[Theorem~1.3]{BKS} one can associate to each $t \in 2^\N$ a Borel fractional perfect matching on $\bbA$ so that for every invariant probability Borel measure $\mu$ on $X$, for almost every $t$ the corresponding fractional perfect matching is $\{0, 1\}$-valued for $\mu$-a.e. component of $\bbA$. By \cite[\nopp18.6]{CDST}, we can choose in a uniformly Borel way a Borel fractional perfect matching $f_\mu$ for every ergodic invariant measure $\mu$ on $X$, so that $f_\mu$ is $\{0, 1\}$-valued for $\mu$-a.e. component of $\bbA$. Let $X_\mu$ denote the set of components for which $f_\mu$ is $\{0, 1\}$-valued. By considering an ergodic decomposition of $X$ (cf. \cite[Theorem~5.12]{CBER}), the set $Z = \bigcup_\mu X_\mu$ is Borel, and $f = \bigcup_\mu \res{f_\mu}{X_\mu}$ gives a Borel perfect matching of $\res{\bbA}{Z}$. Moreover, $\mu(Z) = 1$ for every invariant probability Borel measure on $X$. By \cref{prop:weak-correspondence-expansions}(2), we are done.

	For the acyclic graphs with no infinite injective rays of degree $2$ on even vertices, the argument is similar: We consider an ergodic decomposition, and note that the proof of \cite[Theorem~B]{CM} is effective enough that the union of the solutions (and their domains) for all ergodic invariant probability Borel measures is still Borel.

	For the regular one-ended graphs, we again consider an ergodic decomposition and argue that the proof of \cite[Theorem~1.1]{BKS} is sufficiently uniform. For a fixed measure $\mu$, the proof proceeds by constructing a transfinite sequence of fractional perfect matchings, and arguing that this must stabilize at some countable ordinal. We claim that the construction of this sequence is effective (in $\mu$). Then, by the Boundedness Theorem for analytic well-founded relations \cite[\nopp31.1]{CDST} there is a uniform bound on how long these sequences take to stabilize for all (ergodic) invariant probability Borel measures, so we are done by the same argument as in the previous two cases.

	The verification that the construction is effective is tedious but straightforward. The most subtle step is in the use of the Choquet--Bishop--de Leeuw Theorem, which is sufficiently effective for separable metrizable spaces as this essentially boils down to an application of compact uniformization, see e.g. \cites[Section~3]{Choquet-book}[IV.9]{Simpson}[\nopp28.8]{CDST}.
\end{proof}

\section{Problems}\label{sec:problems}

\begin{prob}
	Does the conclusion of \cref{thm:generic-expansions} hold for classes of structures without TAC?
\end{prob}

In \cite{CK}, it is shown that for many natural classes of aperiodic CBER $\mathcal{E}$, there is a Borel class of structures $\K$ so that $E \in \mathcal{E}$ if and only if $E$ admits a Borel $\K$-structuring. Nonetheless, it is interesting whether there is any ``natural'' class of problems (e.g. problems that are studied in finite combinatorics) that carve out interesting classes $\mathcal{E}$ of CBER. \Cref{eg:zline} was an attempt to characterize hyperfiniteness, though we have seen that it actually enforces smoothness.

\begin{prob}
	Let $\mathcal{E}$ be a class of aperiodic CBER such as those that are hyperfinite, (non)-compressible or treeable. Is there a ``natural'' expansion problem $(\K, \K^*)$ for which an aperiodic CBER $E$ is Borel expandable if and only if $E \in \mathcal{E}$?
\end{prob}

In \cite{GX} a problem is described for which $E$ admits solutions exactly when $E$ is hyperfinite. However, this does not fit the framework of expansion problems, as it involves finding ``approximate'' solutions.

We note that the spanning tree example (\cref{eg:trees}) corresponds to a class of CBER that lies somewhere between hyperfinite and treeable.

\begin{prob}
	What is the class of CBER that are Borel expandable for the spanning tree problem?
\end{prob}

In general, it would be interesting to answer the problems remaining in \cref{fig:table}. We highlight a few of these below.

\begin{prob}
	For the Ramsey expansion problem (\cref{eg:ramsey}), when an invariant random structure admits an invariant random expansion? Can we characterize a.e. expansions (in the sense of \cref{sec:canonical-random-expansions})? Under what assumptions to generic expansions exist on CBER?
\end{prob}

\begin{prob}
	Can we say more about when a Borel structuring of a CBER by partial orders is expandable to a Borel structuring by linear orders? In particular, is there a characterization of exactly which invariant random expansions come from push-forwards along a.e. equivariant expansion maps?
\end{prob}

\begin{prob}
	Does Vizing's Theorem hold generically?
\end{prob}

There are many open problems regarding the existence of perfect matchings; see \cref{sec:matchings} for details. As noted in \cref{rmk:fractional-perfect-matching}, one can often find perfect matchings by first finding non-integral fractional perfect matchings, and then rounding them.

\begin{prob}
	What can be said about the expansion problem of finding a non-integral fractional perfect matching on a Borel graph? When can we round non-integral fractional perfect matchings to perfect matchings?
\end{prob}

See e.g. \cite{BKS,Timar,BPZ} for some partial results and examples.

\begin{appendices}
\crefalias{section}{appendix}

\renewcommand{\thesection}{\Alph{section}.}
\section{Existence of frequencies}
\renewcommand{\thesection}{\Alph{section}}
\label{app:proof-of-freq}

The purpose of this appendix is to prove \cref{lem:existence-of-frequencies} on the existence of frequencies. The proof is essentially the same as that of \cite[Lemma~4.2]{LS_indistinguishability}, though we work here in a more general setting; we also thank Minghao Pan for sharing with us his notes about this proof.

We begin by recalling some definitions.

Let $\Gamma$ be a countably infinite group and let $(Z_n)_{n \in \N}$ be a random walk on $\Gamma$ with symmetric step distribution $\mu$ whose support generates $\Gamma$. (Note that we do not assume $\mu$ to be finitely supported.) For $\gamma \in \Gamma$, let $\bb{P}_\gamma$ denote the law of the random walk $(Z_n)_{n \in \N}$ starting at $\gamma$.

Define $\Omega(\Gamma, \mu)$ to be the set of all $W \subseteq \Gamma$ for which there exists $r \in [0, 1]$ so that
\[\lim_{n \to \infty} \frac{1}{n} \sum_{i = 0}^{n-1} \mathbbl{1}(Z_i \in W) = r, ~\bb{P}_\gamma\text{-a.s. for all } \gamma \in \Gamma,\]
and for $W \in \Omega(\Gamma, \mu)$ we let $\freq_\mu(W)$ be the unique such $r$. We note that $\Omega(\Gamma, \mu) \subseteq 2^\Gamma, \freq_\mu: \Omega(\Gamma, \mu) \to [0, 1]$ are Borel and $\Gamma$-invariant. (Here, $\mathbbl{1}(Z_n \in W)$ is equal to $1$ when $Z_n \in W$ and $0$ otherwise.)

Note that if $\lim_{n \to \infty} \frac{1}{n} \sum_{i = 0}^{n-1} \mathbbl{1}(Z_i \in W)$ converges $\bb{P}_\gamma$-almost surely for some $\gamma$, then it does for all $\gamma$. To see this, note that the support of $\mu$ generates $\Gamma$, so if the sequence diverges with positive probability for a random walk starting at some $\gamma$, then this happens with positive probability for a random walk starting at any $\gamma$. Similarly, we see that the value of the limit (should it exist) does not depend on the choice of $\gamma$.

Let now $\K$ denote the class of equivalence relations. Let $\nu$ be an invariant random equivalence relation on $\Gamma$, i.e. an invariant random $\K$-structure on $\Gamma$, and let $E \sim \nu$. Let $e$ denote the identity in $\Gamma$. We will show that $\nu$-almost surely, $\lim_{n \to \infty} \frac{1}{n} \sum_{i = 0}^{n-1} \mathbbl{1}(Z_i \in C)$ converges to a constant value $\bb{P}_e$-a.s. for every $E$-class $C$. By the previous remark, this proves \cref{lem:existence-of-frequencies}.

A \emph{two-sided} random walk starting at $\gamma$ is a sequence of random variables $(Z_n)_{n \in \Z}$ so that $(Z_n)_{n \in \N}$ and $(Z_{-n})_{n \in \N}$ are random walks starting at $\gamma$. Let $\hat{\bb{P}}_\gamma$ denote the law of the two-sided random walk starting at $\gamma$.

Note that $\Gamma$ acts on $\Gamma^\Z$ by coordinate-wise multiplication, so that we may consider $\K(\Gamma) \times \Gamma^\Z$ with the diagonal action of $\Gamma$:
\[\gamma \cdot (E, (Z_n)_{n \in \Z}) = (\gamma \cdot E, (\gamma \cdot Z_n)_{n \in \Z}).\]
We also define the \emph{shift map} $S: \K(\Gamma) \times \Gamma^\Z \to \K(\Gamma) \times \Gamma^\Z$ to be the map $S(E, (Z_n)_{n \in \Z}) = (E, (Z_{n+1})_{n \in \Z})$. Note that the actions of $\Gamma, S$ on $\K(\Gamma) \times \Gamma^\Z$ commute.

Let $\mathcal{I}$ denote the $\sigma$-algebra of $\Gamma$-invariant Borel sets in $\K(\Gamma) \times \Gamma^\Z$, and set $\lambda = \nu \times \hat{\bb{P}}_e$. Also, for $E \in \K(\Gamma)$, let $\Gamma / E$ denote the set of equivalence classes of $E$.

\begin{claim}\label{claim:appendix-proof}
	If $A \in \mathcal{I}$, then $\lambda(A) = \lambda(S \cdot A)$.
\end{claim}

\begin{claimproof}
	Let $W^n_\gamma = \{Z = (Z_n)_{n \in \Z} : Z_n = \gamma\}$. Note that by the symmetry of $\mu$,
	\[\sum_{\gamma \in \Gamma} \hat{\bb{P}}_\gamma [W^j_{\gamma_j} \cap \cdots \cap W^k_{\gamma_k}] = \prod_{i = j}^{k-1} \mu(\gamma_i^{-1} \gamma_{i+1})\]
	for all $j < 0 < k \in \Z$ and $\gamma_j, \dots, \gamma_k \in \Gamma$. It follows that $\sum_{\gamma \in \Gamma} \hat{\bb{P}}_\gamma$ is shift-invariant.

	Let now $\kappa = \nu \times \sum_{\gamma \in \Gamma} \hat{\bb{P}}_\gamma$. Note that $\kappa$ is $\Gamma$-invariant. If $A \in \mathcal{I}$, then by $\Gamma$-invariance we have
	\[\kappa(A \cap W^0_e) = \sum_\gamma \kappa(A \cap W^0_e \cap W^{-1}_\gamma) = \sum_\gamma \kappa(A \cap W^0_\gamma \cap W^{-1}_e) = \kappa(A \cap W^{-1}_e).\]
	It follows that
	\[\lambda(S  A) = \kappa (S  A \cap W^0_e) = \kappa(S  A \cap W^{-1}_e) = \kappa (S  (A \cap W^0_e)) = \kappa(A \cap W^0_e) = \lambda(A).\]
\end{claimproof}

For $C \subseteq \Gamma$, $Z = (Z_n)_{n \in \Z} \in \Gamma^\Z$, $m < n \in \Z$, let
\[\alpha^n_m(C, Z) = \frac{1}{n-m} \sum_{i = m}^{n-1} \mathbbl{1}(Z_i \in C),\]
and for $(E, Z) \in \K(\Gamma) \times \Gamma^\Z$, $n, k \in \N$ let
\[F^n_k(E, Z) = n \cdot \max\{\alpha^n_0 (C_0, Z) + \cdots + \alpha^n_0(C_{k-1}, Z) : C_0, \dots, C_{k-1} ~\text{are distinct $E$-classes}\}.\]
It is easy to see that $F^{i+j}_k(E, Z) \leq F^i_k(E, Z) + F^j_k(S^i \cdot (E, Z))$ and that each $F^n_k$ is $\Gamma$-invariant. By \cref{claim:appendix-proof} and Kingman's Subadditive Ergodic Theorem (see e.g. \cite{subadditive-ergodic-thm}) there are $\Gamma, S$-invariant maps $F_k, k \in \N$ so that $F_k(E, Z) = \lim_{n \to \infty} \frac{F^n_k(E, Z)}{n}$ $\lambda$-a.s.

Let now
\[A^n(E, Z) = \{|\alpha^m_0(C, Z) - \alpha^k_0(C, Z)| : k, m \geq n \And C \in \Gamma / E\}.\]

\begin{claim}
	$\lim_{n \to \infty} \max(A^n(E, Z)) = 0$ for almost every $(E, Z)$.
\end{claim}

\begin{claimproof}
	Fix $(E, Z)$ for which $F_k(E, Z) = \lim_{n \to \infty} \frac{F^n_k(E, Z)}{n}$ for all $k \in \N$.

	For any $E$-class $C$ and $n \in \N$, there is some $k$ so that $\alpha^n_0(C, Z) = \frac{F^n_k(C, Z)}{n} - \frac{F^n_{k-1}(C, Z)}{n}$, namely the $k$ for which $C$ is the $k$-th most frequently visited $E$-class in the first $n$ steps of $Z$. It follows that
	\[\alpha^m_0(C, Z) \in S_n = \left\{\frac{F^m_k(E, Z)}{m} - \frac{F^m_{k-1}(E, Z)}{m} : k \geq 1, m \geq n\right\}\]
	for $m \geq n$.

	Let $S_n(\delta)$ denote the $\delta$-neighbourhood of $S_n$ in $[0, 1]$. Since $|\alpha^{m+1}_0(C, Z) - \alpha^m_0(C, Z)| \leq \frac{1}{m}$, the set $\{\alpha^m_0(C, Z) : m \geq n\}$ is contained in a single connected component of $S_n(\frac{1}{n})$, for every $E$-class $C$. It therefore suffices to show that for all $\varepsilon > 0$, there is some $n$ sufficiently large that $S_n(\frac{1}{n})$ has length at most $\varepsilon$.

	Fix now $\varepsilon > 0$, and fix $k$ sufficiently large that $F_j(E, Z) - F_{j-1}(E, Z) \leq \varepsilon$ for $j \geq k$. Fix $n$ so that $|F_j(E, Z) - \frac{F^m_j(E, Z)}{m}| \leq \frac{\varepsilon}{k+1}$ for all $j \leq k$ and $m \geq n$. It follows that every point in $S_n$ is within distance $\varepsilon$ from $0$, or $\frac{\varepsilon}{k+1}$ from $F_j(E, Z) - F_{j-1}(E, Z)$ for some $j \leq k$, so that $S_n(\frac{1}{n})$ has length at most $3(\varepsilon + \frac{1}{n})$. Since $\varepsilon$ was arbitrary, this proves the claim.
\end{claimproof}

It follows that for almost every $(E, Z)$, $(\alpha^n_0(C, Z))_{n \in \N}$ is Cauchy for every $E$-class $C$, and hence this sequence converges. Symmetrically, $(\alpha^0_{-n}(C, Z))_{n \in \N}$ converges a.s. for every $E$-class $C$.

Note that with probability $1$
\[\max_{C \in \Gamma / E} |\alpha^{2n}_n(C, Z) - \alpha^n_0(C, Z)| = 2 \cdot \max_{C \in \Gamma / E} |\alpha^{2n}_0(C, Z) - \alpha^n_0(C, Z)| \xrightarrow{n \to \infty} 0,\]
so we may fix a sequence $n_k$ such that
\[\bb{P}\left[\max_{C \in \Gamma / E} |\alpha^{2n_k}_{n_k}(C, Z) - \alpha^{n_k}_0(C, Z)| \geq 2^{-k}\right] \leq 2^{-k}.\]
By \cref{claim:appendix-proof},
\[\bb{P}\left[\max_{C \in \Gamma / E} |\alpha^{2n_k}_{n_k}(C, Z) - \alpha^{n_k}_0(C, Z)| \geq 2^{-k}\right] = \bb{P}\left[\max_{C \in \Gamma / E} |\alpha^{n_k}_0(C, Z) - \alpha^0_{-n_k}(C, Z)| \geq 2^{-k}\right],\]
so by the Borel--Cantelli Lemma we have that for almost every $(E, Z)$,
\[\max_{C \in \Gamma / E} |\alpha^{n_k}_0(C, Z) - \alpha^0_{-n_k}(C, Z)| < 2^{-k}\]
for all but finitely many $k$. It follows that
\[\lim_{n \to \infty} \alpha^n_0(C, Z)  = \lim_{n \to \infty} \alpha^0_{-n}(C, Z)\]
almost surely for all $C \in \Gamma / E$. But for any fixed $C \subseteq \Gamma$, $\alpha^n_0(C, Z), \alpha^0_{-n}(C, Z)$ are independent, so the limits are independent and hence must be constant a.s.

\end{appendices}

\printbibliography

\end{document}